\documentclass[a4paper,11pt]{article}
\usepackage{amssymb}
\usepackage{amsmath,amsfonts,amsthm,amssymb}
\usepackage[dvips]{graphics}
\usepackage{epsfig}
\usepackage{indentfirst}
\usepackage{color}
\usepackage{amsmath}
\usepackage{relsize}
\usepackage{comment}
\allowdisplaybreaks

\newtheoremstyle{theorem}
  {10pt}          
  {10pt}  
  {\sl}  
 {}
  {\bf}  
  {. }    
  { }    
  {}     
\theoremstyle{theorem}

\newtheorem{theorem}{Theorem}[section]

\newtheorem{definition}{Definition}[section]
\newtheorem{lemma}{Lemma}[section]

\newtheorem{remark}{Remark}[section]

\numberwithin{equation}{section}

\newtheoremstyle{defi}
{10pt}  
{10pt}  
{\rm}   
{}      
{\bf}   
{. }    
{ }     
{}      
\theoremstyle{defi}

\newcommand{\D}{\partial}
\renewcommand{\div}{\textrm{div}\,}

\newcommand{\bee}{\mathbf{u}^E}
\newcommand{\bv}{\mathbf{v}}

\newcommand{\ren}{\varrho_{\varepsilon, n}}

\definecolor{mypink}{RGB}{219, 48, 122}
\definecolor{myblue}{RGB}{0, 0, 122}

\begin{document}
\baselineskip = 13.5pt

\title{Inviscid limit for the compressible Navier-Stokes equations
  with density dependent viscosity}

\author{Luca Bisconti$^{1}$ \ \ Matteo Caggio$^{2*}$
  \\
  {\small  1. Università degli Studi di Firenze Dipartimento di Matematica e Informatica “U. Dini”} \\
  {\small Viale Morgagni 67/a, I-50134 Firenze, Italia}\\
  {\small  2. Institute of Mathematics of the Academy of Sciences of the Czech Republic,} \\
  {\small \v Zitn\' a 25, 11567, Praha 1, Czech Republic}\\
  {\small * corresponding author: matteocaggio@gmail.com}\\
  {\small luca.bisconti@unifi.it}
\date{}
}

\maketitle

\begin{abstract}
  We consider the compressible Navier-Stokes system describing the
  motion of a barotropic fluid with density dependent viscosity
  confined in a three-dimensional bounded domain $\Omega$. We show the 
  convergence of the weak solution to the compressible Navier-Stokes system to
  the strong solution to the compressible Euler system when the
  viscosity and the damping coefficients tend to zero.  
\end{abstract}

\medskip


  \medskip


\section{Introduction and main results}\setcounter{equation}{0}
In the three-dimensional smooth bounded domain $\Omega \subset \mathbb{R}^3$ we consider the compressible
Navier-Stokes system describing the motion of a barotropic fluid with
density dependent viscosity,

\begin{gather}
  \label{cont}
  \partial_{t}\varrho_{\varepsilon}+\textrm{div}_{x}\left(\varrho_{\varepsilon}\mathbf{u}_{\varepsilon}\right)=0,
\\[0.2 cm]
  \begin{aligned}
  \partial_{t}\left(\varrho_{\varepsilon}\mathbf{u}_{\varepsilon}\right)+\textrm{div}_{x} &
  \left(\varrho_{\varepsilon}\mathbf{u}_{\varepsilon}\otimes\mathbf{u}_{\varepsilon}\right)
  +  \nabla_{x}p\left(\varrho_{\varepsilon}\right)
  \\[0.1 cm]
  & - 2 \text{div}_{x} (\mu(\varrho_\varepsilon)\mathbb{D}(\mathbf{u}_\varepsilon))
- \nabla_{x}(\lambda(\varrho_\varepsilon)\textrm{div}_{x}\mathbf{u}_\varepsilon)
  + r_1 |\mathbf{u}_\varepsilon|\mathbf{u}_\varepsilon=0,
\end{aligned}  \label{mom}
\end{gather}
supplemented with the initial conditions
\begin{equation} \label{ic} \varrho_\varepsilon (0, \cdot) =
  \varrho_{0, \varepsilon}, \,\,\,\ \varrho_\varepsilon\mathbf{u}_\varepsilon (0, \cdot) =
  \varrho_{0, \varepsilon}  \mathbf{u}_{0, \varepsilon}
\end{equation}
and the boundary conditions
\begin{equation} \label{bc}
  \varrho_\varepsilon \mathbf{u}_\varepsilon|_{\partial \Omega}=0, \,\,\, 
    \left[ \mu(\varrho_\varepsilon)\nabla \log \varrho_\varepsilon \right]
    \times \mathbf{n} |_{\partial  \Omega}= 0.
\end{equation}
where $\mathbf{n}$ is the unit vector normal to the boundary. 

Here, $\varrho_{\varepsilon}=\varrho_{\varepsilon}\left(x,t\right)$,
$\mathbf{u}_{\varepsilon}=\mathbf{u}_{\varepsilon}\left(x,t\right)$ and
$p=p(\varrho_{\varepsilon}\left(x,t\right))$ represent the mass density,
the velocity vector and the pressure of the fluid respectively. This
last is given by a power law type
\begin{equation} \label{press} p(\varrho) = a \varrho^\gamma, \ \ a>0,
  \ \ \gamma > 1.
\end{equation}
The term $r_1|\mathbf{u}_\varepsilon|\mathbf{u}_\varepsilon$
represents
a damping term, $\mathbb{D}(\mathbf{u}_\varepsilon)=(\nabla_x \mathbf{u}_\varepsilon +
\nabla_x^{\top} \mathbf{u}_\varepsilon)/2$ and
the viscosity coefficients $\mu(\varrho_\varepsilon)$ and
$\lambda(\varrho_\varepsilon)$ satisfy the following algebraic
relation
\begin{equation} \label{lm} \lambda(\varrho_\varepsilon) = 2
  \varrho_\varepsilon \mu'(\varrho_\varepsilon) -2
  \mu(\varrho_\varepsilon).
\end{equation}

In the following we consider the case
$\mu(\varrho_\varepsilon)= \varepsilon \varrho_\varepsilon$ and
$\lambda(\varrho_\varepsilon)=0$, with $\varepsilon > 0$ constant viscosity coefficient.
Consequently, the system (\ref{cont}) and (\ref{mom}) reads as follows
\begin{equation} \label{cont-}
  \partial_{t}\varrho_{\varepsilon}+\textrm{div}_{x}\left(\varrho_{\varepsilon}\mathbf{u}_{\varepsilon}\right)=0,
\end{equation}
\begin{equation} \label{mom-}
  \partial_{t}\left(\varrho_{\varepsilon}\mathbf{u}_{\varepsilon}\right)+\textrm{div}_{x}
  \left(\varrho_{\varepsilon}\mathbf{u}_{\varepsilon}\otimes\mathbf{u}_{\varepsilon}\right)
  +\nabla_{x}p\left(\varrho_{\varepsilon}\right)-2\varepsilon\text{div}_{x}
  (\varrho_\varepsilon\mathbb{D}(\mathbf{u}_\varepsilon))
  + r_1 |\mathbf{u}_\varepsilon|\mathbf{u}_\varepsilon=0.
\end{equation}
Formally, letting $(\varepsilon,r_1)\rightarrow0$, one would expect to obtain the compressible Euler equations
\begin{equation} \label{cont-E}
  \partial_t\varrho^E +\textrm{div}_x
  (\varrho^E \mathbf{u}^E) = 0,
\end{equation}
\begin{equation} \label{mom-E}
  \partial_t (\varrho^E \mathbf{u}^E) +
  \div_x(\varrho^E \mathbf{u}^E \otimes \mathbf{u}^E) + \nabla_x
  p(\varrho^E) =0,
\end{equation}
for which we prescribe the initial conditions
\begin{equation} \label{E-ic}
    \varrho^E(0,\cdot)=\varrho_0^E, \,\,\,
    \varrho^E \mathbf{u}^E (0,\cdot) = \varrho_0^E \mathbf{u}_0^E,
\end{equation}
and the boundary condition
\begin{equation} \label{bc-E}
  \bee\cdot\mathbf{n}|_{\partial \Omega_\delta}=0.
\end{equation}
In the present analysis we aim to prove the convergence of the weak solution to t
he compressible Navier-Stokes system \eqref{cont-}, \eqref{mom-} to the strong solution
to the compressible Euler system \eqref{cont-E}, \eqref{mom-E} in the limit of
$(\varepsilon,r_1)\rightarrow0$.

\begin{remark} \label{bc-rho}
  The boundary condition $\eqref{bc}_{2}$
  has been introduced tacitly assuming that, in the system
  \eqref{cont}--\eqref{mom}, $\lambda(\varrho_\varepsilon)=0$. This
boundary condition expresses that the density should be constant on
each connected component of $\partial \Omega$ and has to be understood
in the weak sense (see the Appendix).  Bresch et al. \cite{BDGV} introduced this boundary
condition in order to preserve the well-known Bresch-Desjardins
entropy inequality on smooth enough bounded domains with Dirichlet and
Navier boundary conditions. For further details the reader can refer
to \cite{BDGV}, Section 3.
\end{remark}

The vanishing limit problem dates to the pioneer work of Prandtl
\cite{Pr} that introduced the concept of \textit{boundary layer} and
for the Navier-Stokes equations with no-slip boundary conditions
derived the so-called \textit{Prandtl equations} describing the
boundary layer generated by an incompressible flow near the physical
boundary. In the mathematical context, many interesting results
concerning the Prandtl equations and the vanishing viscosity limit for
the incompressible Navier-Stokes equations have been developed;
cf. see \cite{Gr}, \cite{GuNg}, \cite{LiWa}, \cite{FiMa}, \cite{Mae},
\cite{OlSa}, \cite{SaCa}, \cite{SaCa-1}, \cite{WaWa}, \cite{WaWi},
\cite{XiXi}.

In particular, an approach for proving the convergence from the
solutions of the Navier-Stokes equations to the solutions of the Euler
equation was introduced by Kato \cite{Ka} that studied the vanishing
viscosity limit of the incompressible viscous flow with no-slip
boundary conditions and proved the following conditional result:
\textit{if the energy dissipation rate of the viscous flow in a
  boundary layer of width proportional to the viscosity vanishes, then
  the solutions of the incompressible Navier–Stokes equations converge
  to some solutions of the incompressible Euler equations in the
  energy space.} In other words, the viscous flow can be approximated
by the inviscid flow in the energy space under a dissipation condition
of energy in a neighborhood of the physical boundary with width
proportional to the viscosity, by constructing an artificial or "fake"
boundary layer.

Since the Kato works, the result has been improved by several
authors. Wang \cite{Wa} relaxed Kato’s dissipation condition of energy
to the case only containing the tangential derivatives of the
tangential or normal velocity, but requiring a thicker boundary
layer. Kelliher \cite{Ke} extended Kato’s result replacing the
gradient of velocity of Kato’s energy condition by only the vorticity
of the flow. Finally, under the assumption of the Oleinik condition of
no back-flow in the trace of the Euler flow and of a lower bound for
the Navier-Stokes vorticity in a Kato-like boundary layer, Constantin
et al. \cite{Co} obtained that the inviscid limit from the
Navier-Stokes equations to the Euler equations holds in energy space.

In the compressible case, not much is known.
Sueur \cite{Su} assumed the following (sufficient) conditions in order for the convergence to hold
\begin{equation} \label{Sueur-cond}
    \varepsilon
    \int_{[0,T]\times \Gamma_\varepsilon}
    \left(
    \frac{\varrho_\varepsilon|\mathbf{u}_\varepsilon|^2}{d_\Omega^2(x)}
    +\frac{\varrho_\varepsilon^2(\mathbf{u}_\varepsilon\cdot \mathbf{n})^2}{d_\Omega^2(x)}
    +\mathbb{S}(\nabla_x \mathbf{u}_\varepsilon)
    \right)dxdt \to 0 \mbox{ as } \varepsilon \to 0.
\end{equation}
Here, $\mathbf{u}_\varepsilon\cdot \mathbf{n}$ is the normal component of $\mathbf{u}_\varepsilon$,
$d_\Omega (x)$ the distance of $x \in \Omega$ to the boundary $\partial \Omega$ and
$\Gamma_{c\varepsilon} = \{ x\in \Omega \ : \ d_\Omega(x) \leq c
\varepsilon \}$ for a constant $c > 0$ (in the case of \eqref{Sueur-cond}, we have $c=1$).
Besides (\ref{Sueur-cond}), Bardos and Nguyen \cite{BaNg} introduced other criteria
(see Theorem~1.8). A similar assumption to that in \cite{Su} reads as follows
\begin{equation} \label{Bardos-Nguyen-cond}
    \int_{[0,T]\times \Gamma_\varepsilon}
    \left(
    \frac{\varrho_\varepsilon^\gamma}{\gamma-1}
    +\varepsilon\frac{\varrho_\varepsilon|\mathbf{u}_\varepsilon|^2}{d_\Omega^2(x)}
    +\varepsilon\mathbb{S}(\nabla_x \mathbf{u}_\varepsilon)
    \right)dxdt \to 0 \mbox{ as } \varepsilon \to 0.
\end{equation}
For other results concerning the vanishing viscosity limit in the case
of linearized Navier-Stokes equations, one-dimensional case and
noncharacteristic boundary layer, the reader can refer to \cite{Gue},
\cite{Ro} and \cite{XiZ}, respectively.  The result in \cite{Su} has
been recently improved by Wang and Zhu \cite{WaZh} in the sense that
the authors assumed as sufficient conditions the tangential or the
normal component of velocity only (see relations (2.5) and (2.6) in
Theorem 2.1) at the cost of increasing the width of the boundary
layer.

As mentioned above, the proposal of our analysis is to study the
vanishing viscosity limit for the compressible Navier-Stokes system
with density dependent viscosity. The strategy adopted in order to
prove the convergence relies to the issue of weak-strong uniqueness by
using relative energy estimates. In particular, we introduce a
relative energy functional ``measuring" the distance between the weak
solution of the compressible Navier-Stokes system and the strong
solution of the compressible Euler system. Consequently, we derive a
relative energy inequality satisfied by the weak solution of the
Navier-Stokes equations. As the weak formulation, the relative energy
inequality involves some test functions that have to satisfy the
no-slip conditions, which are not satisfied by the solution of the
Euler equations. For this reasons, we will introduce a ``correction"
based on the Kato's ``fake" boundary layer in order to work with test
functions satisfying the no-slip conditions at the boundaries.
However, different from the construction of Feireisl et
al. \cite{FNS-2011}, \cite{FNJ-2012}, our relative energy inequality
is derived from an ``augmented version" of the compressible
Navier-Stokes system (see Bresch et al. \cite{BNV-2}, \cite{BDZ},
\cite{BGL}; see also \cite{CD},\cite{Cia},\cite{Cia-1} for recent
applications). The reason for that stands in the fact that a $H^1$
bound for the velocity is no longer available because of the density
dependent viscosity. Consequently, standard application of the Korn’s
inequality in the weak-strong uniqueness context is not possible (see,
for example, Feireisl et al. \cite{FNJ-2012}).  We would like to
mention that, as far as the authors are aware, this is the first
result in this direction. A recent result of similar type has been
proved by Geng et al. \cite{Ge} where the authors establish the
convergence in the vanishing viscosity limit of the Navier-Stokes
equations to the Euler equations for three-dimensional compressible
isentropic flow in the whole space $\mathbb{R}^3$ when the viscosity
coefficients are given as constant multiples of the density’s
power. Moreover, a convergence to dissipative solution of compressible
Euler equations has been analyzed in \cite{BNV-2} in the
three-dimensional torus $\mathbb{T}^3$.

Our paper is organized as follows. In this section we introduce the
weak solutions to the compressible Navier-Stokes system \eqref{cont}--\eqref{bc}
together with the existence result. Subsequently, we
discuss the existence of the strong solution to the compressible Euler
system \eqref{cont-E}--\eqref{bc-E} and the ``augmented" version of the
compressible Navier-Stokes system \eqref{cont}--\eqref{bc}. We
conclude the section presenting our main result (see Theorem
\ref{main} below) together with a preliminary Lemma and the a priori
estimates.  Section 2 is devoted to the proof of the our
result. First, we derive a relative energy inequality satisfied by the
weak solutions of the ``augmented" version of the compressible
Navier-Stokes system \eqref{cont}--\eqref{bc}. Second, we introduce
the Kato ``fake" boundary layer and we discuss its properties. Finally,
we perform the inviscid limit.

\subsection{Weak solutions to the compressible Navier-Stokes system}
We introduce the definition of the weak solution to the compressible Navier-Stokes system.
\begin{definition} \label{def-ws}
  We say that $(\varrho_\varepsilon, \mathbf{u}_\varepsilon)$ is a global weak solution of \eqref{cont-} and
  \eqref{mom-} with boundary conditions \eqref{bc} if it satisfies the following regularity properties
\begin{equation*}
  \varrho_\varepsilon \in L^\infty(0,T;L^\gamma(\Omega)), \,\,\, \sqrt{\varrho_\varepsilon}\mathbf{u}_\varepsilon,
  \ \nabla_x \sqrt{\varrho_\varepsilon} \in L^\infty(0,T;L^2(\Omega)),
\end{equation*}
\begin{equation} \label{reg-prop}
\varrho_\varepsilon^{1/3}\mathbf{u}_\varepsilon \in L^3((0,T) \times \Omega),
\end{equation}
as well as $\eqref{bc}_{1}$ in $L^2(0,T;L^1(\partial \Omega))$ and $\eqref{bc}_{2}$
in $L^2(0,T;L^\infty(\partial \Omega))$. The continuity equation is satisfied in the following sense
\begin{equation} \label{mass-ke} -\int_0^T \int_{\Omega}
\varrho_\varepsilon \partial_t \varphi dxdt - \int_0^T \int_{\Omega}
\varrho_\varepsilon \mathbf{u}_\varepsilon \cdot \nabla_{x} \varphi dxdt
= \int_{\Omega} \varrho_\varepsilon(0,\cdot) \varphi(0,\cdot) dx
\end{equation}
for all $\varphi \in C_c^\infty([0,T)\times\Omega;\mathbb{R})$.
The momentum equation is satisfied in the following sense
\begin{equation*}
  -\int_0^T \int_{\Omega} \varrho_\varepsilon \mathbf{u}_\varepsilon
  \cdot \partial_t \boldsymbol{\varphi} dxdt
  -\int_0^T \int_{\Omega} (\varrho_\varepsilon \mathbf{u}_\varepsilon
  \otimes \mathbf{u}_\varepsilon) : \nabla_{x} \boldsymbol{\varphi} dxdt 
  +2\varepsilon\int_0^T \int_{\Omega} \varrho_\varepsilon
  \mathbb{D}(\mathbf{u}_\varepsilon) : \nabla_{x} \boldsymbol{\varphi} dxdt
\end{equation*}
\begin{equation} \label{mom-ke} 
- \int_0^T \int_{\Omega}
p(\varrho_\varepsilon) \mbox{div}_{x}
\boldsymbol{\varphi} dxdt 
+r_1\int_{0}^T\int_{\Omega}\varrho_\varepsilon|\mathbf{u}_\varepsilon|\mathbf{u}_\varepsilon\cdot \boldsymbol{\varphi}dxdt
= \int_{\Omega}
\varrho_\varepsilon \mathbf{u}_\varepsilon(0,\cdot) \cdot \boldsymbol{\varphi}(0,\cdot) dx
\end{equation}
for all $\boldsymbol{\varphi} \in C_c^\infty([0,T)\times\Omega);\mathbb{R}^3)$, where, for $(i,j=1,2,3)$ the viscous term reads as follows
\begin{equation} \label{visc-terms}
\begin{aligned}
    \varepsilon  \int_0^T & \int_{\Omega} \varrho_\varepsilon \mathbb{D}(\mathbf{u}_\varepsilon) : \nabla_{x} \boldsymbol{\varphi} dxdt
\\
    &= -\varepsilon \int_0^T \int_{\Omega}
    \sqrt{\varrho_\varepsilon}\sqrt{\varrho_\varepsilon}(\mathbf{u}_\varepsilon)_j \partial_{ii} \varphi_j dxdt
    - 2\varepsilon \int_0^T \int_{\Omega}
    \sqrt{\varrho_\varepsilon}(\mathbf{u}_\varepsilon)_j \partial_i \sqrt{\varrho_\varepsilon} \partial_i \varphi_j dxdt
\\
& \quad   - \varepsilon \int_0^T \int_{\Omega}
    \sqrt{\varrho_\varepsilon}\sqrt{\varrho_\varepsilon}(\mathbf{u}_\varepsilon)_i \partial_{ji} \varphi_j dxdt
    - 2\varepsilon \int_0^T \int_{\Omega}
    \sqrt{\varrho_\varepsilon}(\mathbf{u}_\varepsilon)_i \partial_j \sqrt{\varrho_\varepsilon} \partial_i \varphi_j dxdt.
  \end{aligned}
\end{equation}
Moreover, there exists $\Lambda$ such that $\varrho_\varepsilon
\mathbf{u}_\varepsilon=\sqrt{\varrho_\varepsilon}\Lambda$, and
$\mathcal{S}\in L^2((0,T)\times\Omega)$ such that $\sqrt{\varrho_\varepsilon}\mathcal{S}=\mbox{Symm}(\nabla(\varrho_\varepsilon
\mathbf{u}_\varepsilon)-2\nabla\sqrt{\varrho_\varepsilon}\otimes\sqrt{\varrho_\varepsilon}\mathbf{u}_\varepsilon)$
in $\mathcal{D}'$, satisfying the following energy inequality
\begin{equation} \label{ee}
  \begin{aligned}
\sup_{t\in(0,T)}\int_{\Omega}\frac{1}{2}\left|\Lambda\right|^{2}
+H& (\varrho_\varepsilon)dx +2\varepsilon\int_{0}^T\int_{\Omega}\left|\mathcal{S}\right|^2dxdt
+r_1\int_{0}^T\int_{\Omega}\varrho_\varepsilon|\mathbf{u}_\varepsilon|^3dxdt
\\
&\leq \int_{\Omega}\frac{1}{2}\varrho_{0,\varepsilon}\left|\mathbf{u}_{0,\varepsilon}\right|^{2}+H(\varrho_{0,\varepsilon})dx,
\end{aligned}
\end{equation}
and there exists $\mathcal{A}\in L^2((0,T)\times\Omega)$ such that
$\sqrt{\varrho_\varepsilon}\mathcal{A}=\mbox{Asymm}(\nabla(\varrho_\varepsilon
\mathbf{u}_\varepsilon)-2\nabla\sqrt{\varrho_\varepsilon}\otimes\sqrt{\varrho_\varepsilon}
\mathbf{u}_\varepsilon)$ in $\mathcal{D}'$, such that the following
Bresch-Desjardins entropy inequality is satisfied
\begin{equation} \label{BD-entropy}
\begin{aligned}
\sup_{t\in(0,T)}& \int_{\Omega}\frac{1}{2}\left|\Lambda
+2\varepsilon\nabla\sqrt{\varrho_\varepsilon}\right|^{2}
+H(\varrho_\varepsilon)dx
+2\varepsilon\int_{0}^T \int_{\Omega}\left|\mathcal{A}\right|^2dxdt
\\
&\quad +\varepsilon\int_{0}^T\int_{\Omega} \frac{p'(\varrho_\varepsilon)}{\varrho_\varepsilon}\left|\nabla\varrho\right|^2dxdt
+r_1\int_{0}^T\int_{\Omega}\varrho_\varepsilon|\mathbf{u}_\varepsilon|^3dxdt
\\
&\quad +\varepsilon r_1\int_{0}^T\int_{\Omega}|\mathbf{u}_\varepsilon|\mathbf{u}_\varepsilon 
\nabla_x \varrho_\varepsilon dxdt
\\
&\leq \int_{\Omega}\frac{1}{2}\left|\sqrt{\varrho_{0,\varepsilon}}\mathbf{u}_{0,\varepsilon}+
  2\varepsilon\nabla\sqrt{\varrho_{0,\varepsilon}}\right|^{2}+H(\varrho_{0,\varepsilon})dx.
\end{aligned}
\end{equation}
\end{definition}

Here $H(\varrho_\varepsilon)$ is such that
\begin{equation*}
  \varrho_\varepsilon H'(\varrho_\varepsilon) - H(\varrho_\varepsilon) =
  p(\varrho_\varepsilon), \,\,\, H''(\varrho_\varepsilon) = \frac{p'(\varrho_\varepsilon)}{\varrho_\varepsilon}.
\end{equation*}
Consequently, we have
\begin{equation*}
    H(\varrho_\varepsilon) = \frac{\varrho_\varepsilon^\gamma}{\gamma - 1}.
\end{equation*}

\begin{remark}
  In \cite{BDGV}, the authors do not define the weak solution
  introducing $\Lambda$, $\mathcal{S}$ and $\mathcal{A}$.  The reason
  that motivates the Definition \ref{def-ws} is related to the fact
  that the “degenerate” viscosity prevents the velocity field to be
  uniquely determined in the vacuum regions, namely regions where
  $\{\varrho_\varepsilon = 0\}$. Indeed, none of the quantities
  $\mathbf{u}_\varepsilon$, $\nabla_x \mathbf{u}_\varepsilon$ and
  $1/\sqrt{\varrho_\varepsilon}$ are defined a.e. in $\Omega$.
  Consequently, the problem is best analyzed in terms of
  $\sqrt{\varrho_\varepsilon}$, $\sqrt{\varrho_\varepsilon}\mathbf{u}_\varepsilon$ and
  $\varrho_\varepsilon \mathbf{u}_\varepsilon = \sqrt{\varrho_\varepsilon}
  \Lambda$ (see, for example, \cite{AnSp}, \cite{AnMa}, \cite{AnSp-1}).
  However, for consistency with the
  literature concerning the Navier-Stokes equations with density
  dependent viscosity, weak solutions are defined in terms of
  $\sqrt{\varrho_\varepsilon}$ and $\sqrt{\varrho_\varepsilon}\mathbf{u}_\varepsilon$.
  For these reasons, the viscous stress tensor in the energy inequality
  \eqref{ee} is thought as
  \begin{equation*}
    \varrho_\varepsilon \mathbb{D}(\mathbf{u}_\varepsilon) = \sqrt{\varrho_\varepsilon}\mathcal{S}.
  \end{equation*}
  Indeed, it is not clear if weak solutions satisfy the energy
  inequality in the usual sense, namely
  \begin{equation*}
    \sup_{t\in(0,T)}\int_{\Omega}\frac{1}{2}\varrho_\varepsilon\left|\mathbf{u}_\varepsilon\right|^{2}
    +H(\varrho_\varepsilon)dx
    +2\varepsilon\int_{0}^T\int_{\Omega}\varrho_\varepsilon\left|\mathbb{D}(\mathbf{u}_\varepsilon)\right|^2dxdt
    +r_1\int_{0}^T\int_{\Omega}\varrho_\varepsilon|\mathbf{u}_\varepsilon|^3dxdt
  \end{equation*}
  \begin{equation} \label{ee-1} \leq
    \int_{\Omega}\frac{1}{2}\varrho_{0,\varepsilon}\left|\mathbf{u}_{0,\varepsilon}\right|^{2}+H(\varrho_{0,\varepsilon})dx.
  \end{equation}
\end{remark}

The following existence result has been proved by Bresch et al. \cite{BDGV}.

\begin{theorem} \label{Th-ws}
  Let the initial data for the compressible Navier-Stokes be given in such a way
\begin{gather*}
    \varrho_{0,\varepsilon} \in L^\gamma(\Omega), \,\,\, \varrho_{0,\varepsilon} \geq 0, \,\,\,
    \nabla_x \sqrt{\varrho_{0,\varepsilon}} \in L^2(\Omega),
\\
    \varrho_{0, \varepsilon}
  \mathbf{u}_{0, \varepsilon} \in L^1(\Omega), \,\,\, \varrho_{0, \varepsilon}
  \mathbf{u}_{0, \varepsilon} = 0 \,\,\,  \text{if} \,\,\, \varrho_{0,\varepsilon}=0, \,\,\, 
  \frac{|\varrho_{0, \varepsilon} \mathbf{u}_{0, \varepsilon}|^2}{\varrho_{0,\varepsilon}} \in L^1(\Omega).
\end{gather*}
Then, for fixed $\varepsilon>0$ and $r_1$, there exist at least a global weak solution to the
compressible Navier-Stokes \eqref{cont-}, \eqref{mom-} with boundary conditions \eqref{bc} in the sense
of Definition~\ref{def-ws}.
\end{theorem}

\begin{remark}
  Theorem \ref{Th-ws} extends to smooth enough bounded domains
  existence results about barotropic compressible Navier–Stokes
  systems with density dependent viscosity coefficients. The authors
  in \cite{BDGV} proved the existence of global weak solutions for
  Dirichlet and Navier boundary conditions on the velocity. An
  additional turbulent drag term in the momentum equation is used to
  handle the construction of approximate solutions.
\end{remark}

\subsection{Strong solution to the compressible Euler system}

We recall the local existence of strong solution for the comnpressible Euler system (see \cite{Ag}, \cite{Ve}, \cite{Eb1}, \cite{Eb2}, \cite{Sc})
\begin{theorem} \label{E-str}
  Let $(\varrho_0^E, \mathbf{u}_0^E) \in C^{1+\delta}$, $\delta > 0$ be some compatible initial data with
  $0 < \inf_\Omega \varrho_0^E$ and $\sup_\Omega \varrho_0^E < \infty$. Then, there exists $T>0$ and a unique solution
\begin{equation} \label{E-int}
    (\varrho^E, \mathbf{u}^E) \in C_w ([0,T]; C^{1+\delta}(\Omega)) \ \cap \ C^1([0,T]\times \overline{\Omega})
\end{equation}
of \eqref{cont-E}--\eqref{bc-E} such that
\begin{equation} \label{E-vac}
    0 < \inf_{(0,T)\times\Omega} \varrho^E   \,\,\,   \text{and}   \,\,\,  \sup_{(0,T)\times\Omega} \varrho^E < \infty.
\end{equation}
\end{theorem}

\begin{remark}
  As remarked by Sueur \cite{Su}, "compatible" refers to some
  conditions satisfied by the initial data on the boundary
  $\partial \Omega$ which are necessary for the existence of strong
  solution (see \cite{RaMa}, \cite{Sc} for more details).
\end{remark}

\subsection{``Augmented" version of the compressible Navier-Stokes system}
As observed in \cite{BDZ}, the system \eqref{cont-}--\eqref{mom-} can
be reformulated through an ``augmented" version. Indeed,
defining the velocity
\begin{equation} \label{v_eps} \mathbf{v}_\varepsilon =
  \mathbf{u}_\varepsilon + \varepsilon
  \nabla_{x}\log\varrho_\varepsilon
\end{equation}
with
\begin{equation} \label{w_eps} \mathbf{w}_\varepsilon =
  \varepsilon\nabla_{x} \log \varrho_\varepsilon,
\end{equation}
the ``augmented" version of the Navier-Stokes system reads as follows
\begin{gather}
  \partial_t \varrho_\varepsilon + \textrm{div}_{x}
  (\varrho_\varepsilon \mathbf{u}_\varepsilon) = 0, \label{cont-ag}
  \\[0.2 cm]
  \begin{aligned} \label{mom1-ag} \partial_t (&\varrho_\varepsilon
    \mathbf{v}_\varepsilon) + \textrm{div}_{x} (\varrho_\varepsilon
    \mathbf{v}_\varepsilon \otimes \mathbf{u}_\varepsilon) +
    \nabla_{x} p(\varrho_\varepsilon) +r_1 \varrho_\varepsilon
    |\mathbf{v}_\varepsilon -
    \mathbf{w}_\varepsilon|(\mathbf{v}_\varepsilon -
    \mathbf{w}_\varepsilon)
    \\[0.2 cm]
    &- \varepsilon \textrm{div}_x (\varrho_\varepsilon
    \mathbb{D}(\mathbf{v}_\varepsilon)) - \varepsilon \textrm{div}_{x}
    ( \varrho_\varepsilon \mathbb{A}(\mathbf{v}_\varepsilon)) +
    \varepsilon \textrm{div}_{x} \left(\varrho_\varepsilon \nabla_{x}
      \mathbf{w}_\varepsilon \right)=0,
  \end{aligned}
  \\[0.2 cm]
  \begin{aligned} \label{mom2-ag} \partial_t (\varrho_\varepsilon
    \mathbf{w}_\varepsilon) + \textrm{div}_{x} (\varrho_\varepsilon
    \mathbf{w}_\varepsilon \otimes \mathbf{u}_\varepsilon)
    +\varepsilon \textrm{div}_{x} \left(\varrho_\varepsilon
      \nabla_{x}^\top \mathbf{u}_\varepsilon\right)=0
  \end{aligned}
\end{gather}
supplemented with the following boundary conditions
\begin{equation} \label{bc-ag}
\left[ \varrho_\varepsilon (\mathbf{v}_\varepsilon -\mathbf{w}_\varepsilon) \right]
|_{\partial  \Omega}=0, \,\,\, 
    \left[  \varrho_\varepsilon  \mathbf{w}_\varepsilon \right] \times \mathbf{n} |_{\partial
    \Omega}= 0,
\end{equation}
meant in the sense of distributions on $\D\Omega$ (see the Appendix).

\begin{remark} \label{deriv}
  We discuss here how to derive the equations (\ref{mom1-ag}) and (\ref{mom2-ag}).
  We start with (\ref{mom2-ag}). From the continuity equation, we have
\begin{equation*}
    2\varepsilon \partial_t (\varrho_\varepsilon \nabla_{x} \log \varrho_\varepsilon)
    = -2\varepsilon\nabla_{x} \text{div}_{x}(\varrho_\varepsilon\mathbf{u}_\varepsilon).
\end{equation*}
Now, the following identity holds (see Antonelli and Spirito \cite{AnSp})
\begin{align*}
    2\varepsilon\nabla_{x} \text{div}_{x}(\varrho_\varepsilon\mathbf{u}_\varepsilon)
    &=2\varepsilon\text{div}_{x}(\varrho_\varepsilon\mathbf{u}_\varepsilon\otimes\nabla_{x}\log \varrho_\varepsilon
    +\varrho_\varepsilon\nabla_{x}\log\varrho_\varepsilon\otimes\mathbf{u}_\varepsilon)
\\
   & \quad -2\varepsilon\Delta(\varrho_\varepsilon\mathbf{u}_\varepsilon)
    +4\varepsilon\text{div}_{x}(\varrho_\varepsilon\mathbb{D}(\mathbf{u}_\varepsilon)).
\end{align*}
Thus
\begin{align*}
    2\varepsilon \partial_t (\varrho_\varepsilon \nabla_{x} \log \varrho_\varepsilon)
    &= -2\varepsilon\text{div}_{x}(\varrho_\varepsilon\mathbf{u}_\varepsilon\otimes\nabla_{x}\log \varrho_\varepsilon
    +\varrho_\varepsilon\nabla_{x}\log\varrho_\varepsilon\otimes\mathbf{u}_\varepsilon)
\\[0.1 cm]
   &\quad  +2\varepsilon\Delta(\varrho_\varepsilon\mathbf{u}_\varepsilon)
    -4\varepsilon\text{div}_{x}(\varrho_\varepsilon\mathbb{D}(\mathbf{u}_\varepsilon))
\\[0.1 cm]
  &  = -2\varepsilon\text{div}_{x}(\varrho_\varepsilon\mathbf{u}_\varepsilon\otimes\nabla_{x}\log \varrho_\varepsilon
    +\varrho_\varepsilon\nabla_{x}\log\varrho_\varepsilon\otimes\mathbf{u}_\varepsilon)
\\[0.1 cm]
  &\quad   +2\varepsilon\Delta(\varrho_\varepsilon\mathbf{u}_\varepsilon)
    -2\varepsilon\text{div}_{x}(\varrho_\varepsilon (\nabla_{x} \mathbf{u}_\varepsilon + \nabla_x^{\top} \mathbf{u}_\varepsilon)).
\\[0.1 cm]
 &   = -2\varepsilon\text{div}_{x}(\varrho_\varepsilon\mathbf{u}_\varepsilon\otimes\nabla_{x}\log \varrho_\varepsilon
    +\varrho_\varepsilon\nabla_{x}\log\varrho_\varepsilon\otimes\mathbf{u}_\varepsilon)
\\[0.1 cm]
 &\quad   +2\varepsilon
    \text{div}_{x}
    (\mathbf{u}_\varepsilon \otimes \nabla_{x} \varrho_\varepsilon
    +\varrho_\varepsilon \nabla_{x}\mathbf{u}_\varepsilon)
    -2\varepsilon\text{div}_{x}(\varrho_\varepsilon (\nabla_{x} \mathbf{u}_\varepsilon + \nabla_x^{\top} \mathbf{u}_\varepsilon))
\\[0.1 cm]
&    = -2\varepsilon\text{div}_{x}(\varrho_\varepsilon\mathbf{u}_\varepsilon\otimes\nabla_{x}\log \varrho_\varepsilon
    +\varrho_\varepsilon\nabla_{x}\log\varrho_\varepsilon\otimes\mathbf{u}_\varepsilon)
\\[0.1 cm]
&  \quad  +2\varepsilon \text{div}_{x} (\mathbf{u}_\varepsilon  \otimes\nabla_{x} \varrho_\varepsilon)
    -2\varepsilon\text{div}_{x}(\varrho_\varepsilon \nabla_x^{\top} \mathbf{u}_\varepsilon).
\end{align*}
Now, we can write
\begin{equation*}
    2\varepsilon \text{div}_{x} (\mathbf{u}_\varepsilon \otimes\nabla_{x} \varrho_\varepsilon)=
    2\varepsilon \text{div}_{x}
    (\varrho_\varepsilon\mathbf{u}_\varepsilon 
    \otimes\nabla_{x}\log \varrho_\varepsilon),
\end{equation*}
and it cancels with its counterpart.
Consequently
\begin{equation*}
    \varepsilon \partial_t (\varrho_\varepsilon \nabla_{x} \log \varrho_\varepsilon)
    +\varepsilon\text{div}_{x}(\varrho_\varepsilon\nabla_{x}\log\varrho_\varepsilon
    \otimes\mathbf{u}_\varepsilon)+\varepsilon\text{div}_{x}(\varrho_\varepsilon
    \nabla_x^{\top} \mathbf{u}_\varepsilon)=0.
\end{equation*}
From the definition of $\mathbf{w}_\varepsilon$ we obtain relation
\eqref{mom2-ag}. Equation \eqref{mom1-ag} follows from summing up
\eqref{mom-} with \eqref{mom2-ag}.
\end{remark}

Let us now introduce the notion of the weak solution to the system \eqref{cont-ag}--\eqref{mom2-ag}.
\begin{definition} \label{def-aug}
  We say that $(\varrho_\varepsilon, \mathbf{v}_\varepsilon,\mathbf{w}_\varepsilon)$ is
  a global weak solution of \eqref{cont-ag} and \eqref{mom2-ag} with boundary conditions \eqref{bc-ag}
  if it satisfies the following regularity properties
\begin{equation} \label{reg-prop-ag}
  \varrho_\varepsilon \in L^\infty(0,T;L^\gamma(\Omega)), \,\,\, \sqrt{\varrho_\varepsilon}\mathbf{v}_\varepsilon,
  \,\,\, \nabla_x \sqrt{\varrho_\varepsilon} \in L^\infty(0,T;L^2(\Omega)).
\end{equation}
The continuity equation is satisfied in the following sense
\begin{equation} \label{mass-ag}
  \begin{aligned}
- \int_{\Omega} \varrho_\varepsilon(T,\cdot) \varphi(T,\cdot) dx
&+ \int_{\Omega} \varrho_\varepsilon(0,\cdot) \varphi(0,\cdot) dx
\\
&\quad + \int_0^T \int_{\Omega} \varrho_\varepsilon \partial_t \varphi dxdt + \int_0^T \int_{\Omega}
\varrho_\varepsilon \mathbf{u}_\varepsilon \cdot \nabla_{x} \varphi dxdt =0,
\end{aligned}
\end{equation}
for all $\varphi \in C_c^\infty([0,T]\times\Omega;\mathbb{R})$.
The momentum equations are satisfied in the following sense
\allowdisplaybreaks[1]
\begin{equation} \label{mom-ag-v} 
\begin{gathered}
    -\int_{\Omega} 
    \varrho_\varepsilon \mathbf{v}_\varepsilon(T,\cdot) \cdot \boldsymbol{\varphi}(T,\cdot) dx
    +\int_{\Omega}
    \varrho_\varepsilon \mathbf{v}_\varepsilon(0,\cdot) \cdot \boldsymbol{\varphi}(0,\cdot) dx
  +\int_0^T \int_{\Omega} \varrho_\varepsilon \mathbf{v}_\varepsilon
  \cdot \partial_t \boldsymbol{\varphi} dxdt
  \\
  +\int_0^T \int_{\Omega} (\varrho_\varepsilon \mathbf{v}_\varepsilon \otimes \mathbf{u}_\varepsilon) : \nabla_{x} \boldsymbol{\varphi} dxdt 
  -\varepsilon\int_0^T \int_{\Omega} \varrho_\varepsilon \mathbb{D}(\mathbf{v}_\varepsilon) : \nabla_{x} \boldsymbol{\varphi} dxdt
\\
    -\varepsilon\int_0^T \int_{\Omega} \varrho_\varepsilon \mathbb{A}(\mathbf{v}_\varepsilon) : \nabla_{x} \boldsymbol{\varphi} dxdt
    + \varepsilon\int_0^T \int_{\Omega}
    \varrho_\varepsilon \nabla_x \mathbf{w}_\varepsilon: \nabla_{x} \boldsymbol{\varphi} dxdt 
\\
    -r_1\int_0^T \int_{\Omega}
    \varrho_\varepsilon |\mathbf{v}_\varepsilon - \mathbf{w}_\varepsilon|(\mathbf{v}_\varepsilon - \mathbf{w}_\varepsilon)
    \cdot \boldsymbol{\varphi}
    + \int_0^T \int_{\Omega}  p(\varrho_\varepsilon) \mbox{div}_{x} \boldsymbol{\varphi} dxdt  =0,
  \end{gathered}
\end{equation}
and
\begin{equation} \label{mom-ag-w}
  \begin{gathered}
    -\int_{\Omega} \varrho_\varepsilon \mathbf{w}_\varepsilon(T,\cdot)
    \cdot \boldsymbol{\varphi}(T,\cdot) dx +\int_{\Omega}
    \varrho_\varepsilon \mathbf{w}_\varepsilon(0,\cdot) \cdot
    \boldsymbol{\varphi}(0,\cdot) dx +\int_0^T \int_{\Omega}
    \varrho_\varepsilon \mathbf{w}_\varepsilon \cdot \partial_t
    \boldsymbol{\varphi} dxdt
    \\
    +\int_0^T \int_{\Omega} (\varrho_\varepsilon
    \mathbf{w}_\varepsilon \otimes \mathbf{u}_\varepsilon) :
    \nabla_{x} \boldsymbol{\varphi} dxdt -\varepsilon\int_0^T
    \int_{\Omega} \varrho_\varepsilon \nabla_x^{\top}
    \mathbf{u}_\varepsilon : \nabla_{x} \boldsymbol{\varphi} dxdt =0,
  \end{gathered}
\end{equation}
for all
$\boldsymbol{\varphi} \in
C_c^\infty([0,T]\times\Omega);\mathbb{R}^3)$, where, for $i,j=1,2,3$,
the viscous terms reads as follows 
%
\begin{equation} \label{visc-terms-D}
  \begin{aligned}
    \varepsilon\int_0^T &\int_{\Omega} \varrho_\varepsilon
    \mathbb{D}(\mathbf{v}_\varepsilon) : \nabla_{x}
    \boldsymbol{\varphi} dxdt
    \\
    & = -\varepsilon \int_0^T \int_{\Omega}
    \sqrt{\varrho_\varepsilon}\sqrt{\varrho_\varepsilon}(\mathbf{v}_\varepsilon)_j
    \partial_{ii} \varphi_j dxdt - 2\varepsilon \int_0^T \int_{\Omega}
    \sqrt{\varrho_\varepsilon}(\mathbf{v}_\varepsilon)_j \partial_i
    \sqrt{\varrho_\varepsilon} \partial_i \varphi_j dxdt
    \\
    & \quad - \varepsilon \int_0^T \int_{\Omega}
    \sqrt{\varrho_\varepsilon}\sqrt{\varrho_\varepsilon}(\mathbf{v}_\varepsilon)_i
    \partial_{ji} \varphi_j dxdt- 2\varepsilon \int_0^T \int_{\Omega}
    \sqrt{\varrho_\varepsilon}(\mathbf{v}_\varepsilon)_i \partial_j
    \sqrt{\varrho_\varepsilon} \partial_i \varphi_j dxdt,
  \end{aligned}
\end{equation}

\begin{equation} \label{visc-terms-A}
  \begin{aligned}
    \varepsilon\int_0^T & \int_{\Omega}  \varrho_\varepsilon
    \mathbb{A}(\mathbf{v}_\varepsilon) : \nabla_{x}
    \boldsymbol{\varphi} dxdt
    \\
    & = -\varepsilon \int_0^T \int_{\Omega}
    \sqrt{\varrho_\varepsilon}\sqrt{\varrho_\varepsilon}(\mathbf{v}_\varepsilon)_j
    \partial_{ii} \varphi_jdxdt - 2\varepsilon \int_0^T \int_{\Omega}
    \sqrt{\varrho_\varepsilon}(\mathbf{v}_\varepsilon)_j \partial_i
    \sqrt{\varrho_\varepsilon} \partial_i \varphi_jdxdt
    \\
    &\quad + \varepsilon \int_0^T \int_{\Omega}
    \sqrt{\varrho_\varepsilon}\sqrt{\varrho_\varepsilon}(\mathbf{v}_\varepsilon)_i
    \partial_{ji} \varphi_jdxdt + 2\varepsilon \int_0^T \int_{\Omega}
    \sqrt{\varrho_\varepsilon}(\mathbf{v}_\varepsilon)_i \partial_j
    \sqrt{\varrho_\varepsilon} \partial_i \varphi_jdxdt,
  \end{aligned}
\end{equation}

\begin{equation} \label{visc-terms-w}
  \begin{aligned}
    \varepsilon\int_0^T & \int_{\Omega} \varrho_\varepsilon \nabla_x
    \mathbf{w}_\varepsilon: \nabla_{x} \boldsymbol{\varphi} dxdt
    \\
    & = - \varepsilon \int_0^T \int_{\Omega}
    \sqrt{\varrho_\varepsilon}\sqrt{\varrho_\varepsilon}(\mathbf{w}_\varepsilon)_j
    \partial_{ii} \varphi_j dxdt- 2\varepsilon \int_0^T \int_{\Omega}
    \sqrt{\varrho_\varepsilon}(\mathbf{w}_\varepsilon)_j \partial_i
    \sqrt{\varrho_\varepsilon} \partial_i \varphi_jdxdt,
  \end{aligned}
\end{equation}
and
\begin{equation} \label{visc-terms-ut}
  \begin{aligned}
    \varepsilon \int_0^T & \int_{\Omega} \varrho_\varepsilon \nabla_x^{\top}
    \mathbf{u}_\varepsilon : \nabla_{x} \boldsymbol{\varphi} dxdt
    \\
    & = - \varepsilon \int_0^T \int_{\Omega}
    \sqrt{\varrho_\varepsilon}\sqrt{\varrho_\varepsilon}(\mathbf{u}_\varepsilon)_i
    \partial_{ji} \varphi_j dxdt- 2\varepsilon \int_0^T \int_{\Omega}
    \sqrt{\varrho_\varepsilon}(\mathbf{u}_\varepsilon)_i \partial_j
    \sqrt{\varrho_\varepsilon} \partial_i \varphi_jdxdt
    \\
    & \leq \int_{\Omega}\frac{1}{2}\left(\varrho_{0,\varepsilon}
      \left|\mathbf{u}_{0,\varepsilon}\right|^{2}+H(\varrho_{0,\varepsilon})\right)dx.
  \end{aligned}
\end{equation}
Moreover, there exists $\Lambda$ such that
$\varrho_\varepsilon \mathbf{u}_\varepsilon=\sqrt{\varrho_\varepsilon}\Lambda$, and there
exists $\mathcal{A}\in L^2((0,T)\times\Omega)$ such that
$\sqrt{\varrho_\varepsilon}\mathcal{A}=\mbox{Asymm}(\nabla(\varrho_\varepsilon
\mathbf{u}_\varepsilon)-2\nabla\sqrt{\varrho_\varepsilon}\otimes\sqrt{\varrho_\varepsilon}
\mathbf{u}_\varepsilon)$ in $\mathcal{D}'$, such that the following
Bresch-Desjardins entropy inequality is satisfied
\allowdisplaybreaks[1]  
\begin{equation} \label{k-entropy}
  \begin{aligned}
\sup_{t\in(0,T)}&\int_{\Omega} \left(\frac{1}{2}\left(\left|\Lambda
+2\varepsilon\nabla\sqrt{\varrho_\varepsilon}\right|^{2}
+|2\varepsilon\nabla\sqrt{\varrho_\varepsilon}|^2
\right) +H(\varrho_\varepsilon)\right)dx
\\
&\quad +\varepsilon\int_{0}^T\int_{\Omega}\left|\mathcal{S}\right|^2dxdt
+\varepsilon\int_{0}^T\int_{\Omega}\left|\mathcal{A}\right|^2dxdt
+\varepsilon\int_{0}^T\int_{\Omega}
\frac{p'(\varrho_\varepsilon)}{\varrho_\varepsilon}\left|\nabla\varrho\right|^2dxdt
\\
&\quad +r_1\int_{\Omega}\varrho_\varepsilon|\mathbf{u}_\varepsilon|^3dxdt
+\varepsilon r_1\int_{0}^T\int_{\Omega}|\mathbf{u}_\varepsilon|\mathbf{u}_\varepsilon 
\nabla_x \varrho_\varepsilon dxdt
\\
&\leq \int_{\Omega}\left(\frac{1}{2}\left(\left|\sqrt{\varrho_{0,\varepsilon}}\mathbf{u}_{0,\varepsilon}+
2\varepsilon\nabla\sqrt{\varrho_{0,\varepsilon}}\right|^{2}
+|2\varepsilon\nabla\sqrt{\varrho_{0,\varepsilon}}|^2 \right) +H(\varrho_{0,\varepsilon})\right)dx.
\end{aligned}
\end{equation}
\end{definition}

As the authors remarked in \cite{BGL}, a global weak solution
$(\varrho_\varepsilon, \mathbf{u}_\varepsilon)$ of the compressible
Navier-Stokes system is also a solution of the augmented version.
Consequently, Theorem~\ref{Th-ws} holds for weak solutions to the system
\eqref{cont-ag}--\eqref{mom2-ag} in the sense of Definition~\ref{def-aug}.

\begin{remark} \label{ei-deriv}
  Relation (\ref{k-entropy}) has been originally derived on a
  three-dimensional torus, $\mathbb{T}^3$, in \cite{BNV-2}. In the
  case of a bounded domain, it is possible to obtain the same relation
  thanks to the boundary conditions \eqref{bc}.  Indeed, equation
  \eqref{mom1-ag} can be rewritten as follows
  \begin{equation} \label{ei-c1} 
    \begin{gathered}
  \partial_t (\varrho_\varepsilon \mathbf{v}_\varepsilon)
  + \textrm{div}_{x} (\varrho_\varepsilon \mathbf{v}_\varepsilon \otimes \mathbf{u}_\varepsilon)
  + \nabla_{x} p(\varrho_\varepsilon) +r_1 \varrho_\varepsilon |\mathbf{v}_\varepsilon
  - \mathbf{w}_\varepsilon|(\mathbf{v}_\varepsilon - \mathbf{w}_\varepsilon)
\\
- \varepsilon \textrm{div}
  (\varrho_\varepsilon \mathbb{D}(\mathbf{u}_\varepsilon)) -
  \varepsilon \textrm{div}_{x} ( \varrho_\varepsilon
  \mathbb{A}(\mathbf{u}_\varepsilon)) =0.
\end{gathered}
\end{equation}

We multiply equation \eqref{ei-c1} by $\mathbf{v}_\varepsilon$.
Using the continuity equation and integration by parts, we obtain
\begin{equation} \label{ei-c2}
  \begin{aligned}
    \frac{d}{dt} \int_\Omega & \left(  \varrho_\varepsilon \frac{|\mathbf{v}_\varepsilon|^2}{2}
      + H(\varrho_\varepsilon)  \right)dx  + \int_\Omega \nabla p(\varrho_\varepsilon) \cdot \mathbf{w}_\varepsilon dx
        \\
   &\quad     +r_1 \int_\Omega \varrho_\varepsilon |\mathbf{v}_\varepsilon 
    - \mathbf{w}_\varepsilon|(\mathbf{v}_\varepsilon - \mathbf{w}_\varepsilon)\mathbf{v}_\varepsilon dx
 + \int_\Omega \varrho_\varepsilon \nabla \mathbf{v}_\varepsilon \left(
    \mathbb{D}(\mathbf{u}_\varepsilon)+\mathbb{A}(\mathbf{u}_\varepsilon)  \right) dx
    \\
   &\quad   +\int_{\partial \Omega} \frac{|\mathbf{v}_\varepsilon|^2}{2}
    \varrho_\varepsilon \mathbf{u}_\varepsilon \cdot \mathbf{n} ds
    +\int_{\partial \Omega}  H(\varrho_\varepsilon) \mathbf{u}_\varepsilon \cdot \mathbf{n} ds
    -\varepsilon \int_{\partial \Omega}
    \varrho_\varepsilon \mathbb{D}(\mathbf{u}_\varepsilon) \mathbf{v}_\varepsilon \mathbf{n} ds
        \\
   &\quad   -\varepsilon \int_{\partial \Omega}  \varrho_\varepsilon \mathbb{A}(\mathbf{u}_\varepsilon)
   \mathbf{v}_\varepsilon \mathbf{n}ds =0.
  \end{aligned}
\end{equation}
Thanks to the boundary condition \eqref{bc}, we have
\begin{equation*}
    \int_{\partial \Omega} \frac{|\mathbf{v}_\varepsilon|^2}{2}
    \varrho_\varepsilon \mathbf{u}_\varepsilon \cdot \mathbf{n}ds
    =\int_{\partial \Omega}
    H(\varrho_\varepsilon) \mathbf{u}_\varepsilon \cdot \mathbf{n}ds=0,
\end{equation*}
while, because $\varrho_\varepsilon\mathbf{v}_\varepsilon = \varrho_\varepsilon\mathbf{w}_\varepsilon$
on $\partial \Omega$, we have
\begin{equation*}
    -\int_{\partial \Omega}
    \varrho_\varepsilon \mathbb{A}(\mathbf{u}_\varepsilon) \mathbf{v}_\varepsilon \mathbf{n}ds
    = - \varepsilon \int_{\partial \Omega} \nabla_{x} \varrho_\varepsilon \cdot (\text{curl} \mathbf{u}_\varepsilon \times \mathbf{n})ds
    = \varepsilon \int_{\partial \Omega}
    (\nabla \varrho_\varepsilon \times \mathbf{n}) \cdot \text{curl} \mathbf{u}_\varepsilon ds
\end{equation*}
that is equal zero thanks to the boundary condition $\eqref{bc}_{2}$.
Consequently, we rewrite relation \eqref{ei-c2} as follows
\begin{equation} \label{ei-c3}
\hspace{-0.5 cm}  \begin{gathered}
    \frac{d}{dt}\int_\Omega \left( 
    \varrho_\varepsilon \frac{|\mathbf{v}_\varepsilon|^2}{2}
    + H(\varrho_\varepsilon)
    \right)dx
    + \int_\Omega \nabla p(\varrho_\varepsilon) \cdot \mathbf{w}_\varepsilon dx
    +r_1 \int_\Omega  \varrho_\varepsilon |\mathbf{v}_\varepsilon
    - \mathbf{w}_\varepsilon|(\mathbf{v}_\varepsilon - \mathbf{w}_\varepsilon)\mathbf{v}_\varepsilon dx
\\
    + \varepsilon \int_\Omega \varrho_\varepsilon \nabla \mathbf{v}_\varepsilon \left(
    \nabla_{x} \mathbf{v}_\varepsilon - \nabla_x^{\top} \mathbf{w}_\varepsilon
  \right) dx -\varepsilon \int_{\partial \Omega}  \varrho_\varepsilon
  \mathbb{D}(\mathbf{u}_\varepsilon) \mathbf{v}_\varepsilon \mathbf{n}  ds =0,
  \end{gathered}
\end{equation}
where we formally assumed $\nabla \mathbf{w}_\varepsilon = \nabla^{\top} \mathbf{w}_\varepsilon$.
Now, we multiply (\ref{mom2-ag}) by $\mathbf{w}_\varepsilon$. Similarly as before, we obtain
\begin{equation} \label{ei-c4}
    \frac{d}{dt}\int_\Omega
    \varrho_\varepsilon \frac{|\mathbf{w}_\varepsilon|^2}{2} dx
    + \varepsilon \int_\Omega \textrm{div}_{x} \left(\varrho_\varepsilon \nabla_x^{\top} \mathbf{u}_\varepsilon\right) 
    \cdot \mathbf{w}_\varepsilon dx =0.
\end{equation}
Integrating by parts the viscous term, we have
\begin{align*}
 \int_\Omega \textrm{div}_{x} &\left(\varrho_\varepsilon \nabla_x^{\top} \mathbf{u}_\varepsilon\right) 
  \cdot \mathbf{w}_\varepsilon dx 
\\
       &= - \int_\Omega \varrho_\varepsilon \nabla_x^{\top} \mathbf{u}_\varepsilon \nabla_{x} \mathbf{w}_\varepsilon dx
    + \int_{\partial \Omega}  \varrho_\varepsilon \nabla_x^{\top} \mathbf{u}_\varepsilon \mathbf{w}_\varepsilon
    \mathbf{n} ds
\\
   & = - \int_\Omega \varrho_\varepsilon \nabla_x^{\top} (\mathbf{v}_\varepsilon
    - \mathbf{w}_\varepsilon)  \nabla_{x} \mathbf{w}_\varepsilon dx  + \int_{\partial \Omega}  \varrho_\varepsilon  
    (\mathbb{D}(\mathbf{u}_\varepsilon)-\mathbb{A}(\mathbf{u}_\varepsilon))
    \mathbf{w}_\varepsilon  \mathbf{n} ds.
\end{align*}
Now, again formally, we consider $\nabla \mathbf{w}_\varepsilon = \nabla^{\top} \mathbf{w}_\varepsilon$,
and we rewrite the relation above as
\begin{equation*}
  \begin{aligned}
    - \int_\Omega \varrho_\varepsilon & (\nabla_x^{\top} \mathbf{v}_\varepsilon
    - \nabla_x^{\top} \mathbf{w}_\varepsilon)  \nabla_x^{\top} \mathbf{w}_\varepsilon dx
    + \int_{\partial \Omega}  \varrho_\varepsilon  
    (\mathbb{D}(\mathbf{u}_\varepsilon)-\mathbb{A}(\mathbf{u}_\varepsilon))
    \mathbf{w}_\varepsilon   \mathbf{n} ds
\\
  &  =  - \int_\Omega \varrho_\varepsilon (\nabla_x^{\top} \mathbf{v}_\varepsilon
    - \nabla_x^{\top} \mathbf{w}_\varepsilon)     \nabla_x^{\top} \mathbf{w}_\varepsilon dx
    + \int_{\partial \Omega}     \varrho_\varepsilon    \mathbb{D}(\mathbf{u}_\varepsilon)
    \mathbf{w}_\varepsilon \mathbf{n}   - \int_{\partial \Omega}   \varrho_\varepsilon  
    \mathbb{A}(\mathbf{u}_\varepsilon)  \mathbf{w}_\varepsilon  \mathbf{n} ds,
  \end{aligned}
\end{equation*}
where the last term is equal zero for the same arguments as above.
Consequently, relation \eqref{ei-c4} could be rewritten as follows
\begin{equation} \label{ei-c5}
    \frac{d}{dt}\int_\Omega  \varrho_\varepsilon \frac{|\mathbf{w}_\varepsilon|^2}{2} dx
    - \varepsilon \int_\Omega \varrho_\varepsilon (\nabla_x^{\top} \mathbf{v}_\varepsilon
    - \nabla_{x} \mathbf{w}_\varepsilon) \nabla \mathbf{w}_\varepsilon dx
    + \int_{\partial \Omega}  \varrho_\varepsilon    \mathbb{D}(\mathbf{u}_\varepsilon)
    \mathbf{w}_\varepsilon \mathbf{n} ds =0.
\end{equation}
Summing up \eqref{ei-c3} and \eqref{ei-c5}, we obtain
\begin{equation} \label{ei-c6}
  \begin{aligned}
    \frac{d}{dt}\int_\Omega & \left( \varrho_\varepsilon \left(
    \frac{|\mathbf{v}_\varepsilon|^2}{2}  + \frac{|\mathbf{w}_\varepsilon|^2}{2}
    \right)  + H(\varrho_\varepsilon)  \right) dx
    + \int_\Omega \nabla p(\varrho_\varepsilon) \cdot \mathbf{w}_\varepsilon
    \\
  &\quad   +r_1 \int_\Omega \varrho_\varepsilon |\mathbf{v}_\varepsilon
    - \mathbf{w}_\varepsilon|(\mathbf{v}_\varepsilon - \mathbf{w}_\varepsilon)\mathbf{v}_\varepsilon dx
    + \varepsilon \int_\Omega \varrho_\varepsilon \nabla \mathbf{v}_\varepsilon \left(
    \nabla_{x} \mathbf{v}_\varepsilon   - \nabla_x^{\top} \mathbf{w}_\varepsilon
  \right) dx
  \\
   &\quad  - \varepsilon \int_\Omega \varrho_\varepsilon (\nabla_x^{\top} \mathbf{v}_\varepsilon
    - \nabla_x^{\top} \mathbf{w}_\varepsilon)
    \nabla_x^{\top} \mathbf{w}_\varepsilon=0.
  \end{aligned}
\end{equation}
Now,
\begin{align*}
    \varepsilon \int_\Omega & \varrho_\varepsilon \nabla \mathbf{v}_\varepsilon \left(
    \nabla_{x} \mathbf{v}_\varepsilon  - \nabla_x^{\top} \mathbf{w}_\varepsilon
    \right)  dx - \varepsilon \int_\Omega \varrho_\varepsilon (\nabla_x^{\top} \mathbf{v}_\varepsilon
    - \nabla_x^{\top} \mathbf{w}_\varepsilon)  \nabla_x^{\top} \mathbf{w}_\varepsilon dx
\\
  &  = \varepsilon  \int_\Omega  \varrho_\varepsilon  \left(  |\mathbb{D}(\mathbf{v}_\varepsilon)+
    \mathbb{A}(\mathbf{v}_\varepsilon)|^2 +|\nabla_x^{\top} \mathbf{w}_\varepsilon|^2
    - \nabla_x^{\top} \mathbf{w}_\varepsilon \left( \nabla_{x} \mathbf{v}_\varepsilon
    + \nabla_x^{\top} \mathbf{v}_\varepsilon  \right)  \right) dx
\\
  &  = \varepsilon  \int_\Omega \varrho_\varepsilon  \left(
    |\mathbb{D}(\mathbf{v}_\varepsilon)|^2+  |\mathbb{A}(\mathbf{v}_\varepsilon)|^2
    +|\nabla_x^{\top} \mathbf{w}_\varepsilon|^2
    - 2 \nabla_x^{\top} \mathbf{w}_\varepsilon \mathbb{D}(\mathbf{v}_\varepsilon) \right) dx
\\
  &  = \varepsilon  \int_\Omega  \varrho_\varepsilon  \left(  |\mathbb{A}(\mathbf{v}_\varepsilon)|^2
    +  |\mathbb{D}(\mathbf{v}_\varepsilon)-\nabla_x^{\top} \mathbf{w}_\varepsilon|^2 \right) dx
\\
  &  = \varepsilon   \int_\Omega  \varrho_\varepsilon  \left(    |\mathbb{A}(\mathbf{u}_\varepsilon)|^2
    +   |\mathbb{D}(\mathbf{u}_\varepsilon)|^2  \right) dx.
\end{align*}
Moreover,
\begin{equation*}
    \int_\Omega \nabla p(\varrho_\varepsilon) \cdot \mathbf{w}_\varepsilon 
    = \varepsilon \int_\Omega 
    p'(\varrho_\varepsilon) \nabla_{x} \varrho_\varepsilon
    \cdot \frac{\nabla_{x} \varrho_\varepsilon}{\varrho_\varepsilon}.
\end{equation*}
We consider now the drag term. We can write
\begin{equation*}
  r_1 \int_\Omega \varrho_\varepsilon
  |\mathbf{v}_\varepsilon - \mathbf{w}_\varepsilon|(\mathbf{v}_\varepsilon - \mathbf{w}_\varepsilon)\mathbf{v}_\varepsilon =
  r_1 \int_\Omega \varrho_\varepsilon |\mathbf{v}_\varepsilon - \mathbf{w}_\varepsilon|^3
  + r_1 \int_\Omega  \varrho_\varepsilon |\mathbf{v}_\varepsilon - \mathbf{w}_\varepsilon|(\mathbf{v}_\varepsilon - \mathbf{w}_\varepsilon)\mathbf{w}_\varepsilon. 
\end{equation*}
The second term reads as follows:
\begin{equation*}
  \varepsilon r_1 \int_\Omega \varrho_\varepsilon |\mathbf{u}_\varepsilon| \mathbf{u}_\varepsilon \nabla_x \log \varrho_\varepsilon=
  \varepsilon r_1 \int_\Omega |\mathbf{u}_\varepsilon| \mathbf{u}_\varepsilon \nabla_x \varrho_\varepsilon.
\end{equation*}
By parts integration gives
\begin{equation*}
  \varepsilon r_1 \int_\Omega |\mathbf{u}_\varepsilon| \mathbf{u}_\varepsilon \nabla_x \varrho_\varepsilon
  =- \varepsilon r_1 \int_\Omega |\mathbf{u}_\varepsilon|
  \text{div}_x \mathbf{u}_\varepsilon  \varrho_\varepsilon
  - \varepsilon r_1 \int_\Omega 
  \varrho_\varepsilon 
  \frac{u_k}{|\mathbf{u}_\varepsilon|}
  u_j \partial_j u_k
  + \varepsilon r_1 \int_{\partial_\Omega} \varrho_\varepsilon\mathbf{u}_\varepsilon
  |\mathbf{u}_\varepsilon|,
\end{equation*}
where the boundary term is zero thanks to the condition \eqref{bc}. 
We have
\begin{align*}
\left|\varepsilon r_1 \int_\Omega |\mathbf{u}_\varepsilon| \mathbf{u}_\varepsilon \nabla_x \varrho_\varepsilon
\right| dx &\leq \varepsilon r_1 \int_\Omega \varrho_\varepsilon
|\mathbf{u}_\varepsilon||\mathbb{D}(\mathbf{u}_\varepsilon)| dx
\\
& \leq \varepsilon r_1 \| \sqrt{\varrho_\varepsilon}\mathbf{u}_\varepsilon
\|_{L^2(\Omega)} \| \sqrt{\varrho_\varepsilon}\mathbb{D}(\mathbf{u}_\varepsilon)\|_{L^2(\Omega)} 
\\
& \leq  r_1 \left\| \frac{\sqrt{\varrho_\varepsilon}}{\varrho_\varepsilon^{1/3}}\varrho_\varepsilon^{1/3}\mathbf{u}_\varepsilon
\right\|_{L^2(\Omega)} \varepsilon \|\sqrt{\varrho_\varepsilon}\mathbb{D}(\mathbf{u}_\varepsilon)\|_{L^2(\Omega)}
\\
  &  \leq \frac{r_1}{2} \left\|
    \frac{\sqrt{\varrho_\varepsilon}}{\varrho_\varepsilon^{1/3}}\varrho_\varepsilon^{1/3}\mathbf{u}_\varepsilon
\right\|_{L^2(\Omega)}^2 + \frac{\varepsilon}{2} \|\sqrt{\varrho_\varepsilon}\mathbb{D}(\mathbf{u}_\varepsilon)
\|_{L^2(\Omega)}^2,
\end{align*}
where the second term can be absorbed. For the first term, we have
\begin{align*}
\frac{r_1}{2} 
\left\|\frac{\sqrt{\varrho_\varepsilon}}{\varrho_\varepsilon^{1/3}}\varrho_\varepsilon^{1/3}\mathbf{u}_\varepsilon
  \right\|_{L^2(\Omega)}^2 &= \frac{r_1}{2} \int_\Omega \frac{\varrho_\varepsilon}{\varrho_\varepsilon^{2/3}}
   \varrho_\varepsilon^{2/3}\mathbf{u}_\varepsilon^{2} dx.
\\
\intertext{Namely}
\frac{r_1}{2} \left\| \frac{\sqrt{\varrho_\varepsilon}}{\varrho_\varepsilon^{1/3}}\varrho_\varepsilon^{1/3}\mathbf{u}_\varepsilon
\right\|_{L^2(\Omega)}^2 &\leq \frac{r_1}{6}\int_\Omega \varrho_\varepsilon
+ \frac{r_1}{3} \int_\Omega \varrho_\varepsilon |\mathbf{u}_\varepsilon|^3 dx.
\end{align*}
The second term can be absorbed. The first term is bounded by a constant
$c=c(r_1) > 0$ because $\varrho_\varepsilon$ is bounded in $L^\infty(0,T;L^\gamma (\Omega))$ with $\gamma>1$.
Consequently, we conclude the derivation of \eqref{k-entropy}.
\end{remark}

\subsection{Main result}
We are now in position to state the following theorem
\begin{theorem} \label{main}
  Let $T > 0$ be given and let
  $(\varrho^E,\mathbf{u}^E)$ be the strong solution for the
  compressible Euler system (\ref{cont-E}),(\ref{mom-E}) corresponding
  to the initial data $(\varrho^E_0,\mathbf{u}^E_0)$ as in Theorem
  \ref{E-str}. For any $\varepsilon \in (0,1)$, let $(\varrho_{0,\varepsilon},
  \mathbf{v}_{0,\varepsilon},\mathbf{w}_{0,\varepsilon})$ be an initial data such that
\begin{equation*}
\mathbf{v}_{0,\varepsilon}
= \mathbf{u}_{0,\varepsilon} + \mathbf{w}_{0,\varepsilon}
\end{equation*}
and
\begin{equation*}
    \varrho_{0,\varepsilon} \in L^\gamma(\Omega), \,\,\,
    \varrho_{0,\varepsilon} \geq 0, \,\,\,
    \nabla_x \sqrt{\varrho_{0,\varepsilon}} \in L^2(\Omega)
\end{equation*}
\begin{equation*}
    \varrho_{0, \varepsilon} \mathbf{v}_{0, \varepsilon} \in L^1(\Omega), \,\,\, 
  \varrho_{0, \varepsilon} \mathbf{v}_{0, \varepsilon} = 0
  \,\,\,  \text{if} \,\,\,  \varrho_{0,\varepsilon}=0, \,\,\, 
  \frac{|\varrho_{0, \varepsilon}
    \mathbf{v}_{0, \varepsilon}|^2}{\varrho_{0,\varepsilon}} \in L^1(\Omega).
\end{equation*}
Assume that
\begin{equation} \label{id-conv}
 \hspace{-0.3 cm}   \left[  \|  \varrho_{0,\varepsilon} - \varrho_0^E
    \|_{L^\gamma(\Omega)}  + \int_\Omega 
    \varrho_{0,\varepsilon}
    |\mathbf{v}_{0,\varepsilon}-\mathbf{v}_0^E|^2  dx
    + \int_\Omega \varrho_{0,\varepsilon}
    |\mathbf{w}_{0,\varepsilon}-\mathbf{w}_0^E|^2  dx \right]  \to 0 
    \mbox{ as } \varepsilon \to 0
\end{equation}
and
\begin{equation} \label{cond-conv}
\| \varrho_\varepsilon \|_{L^\gamma([0,T];L^\gamma(\Gamma_\varepsilon))}
= o(\varepsilon^{\frac{1}{\gamma}}),\ \ \ \ 
\varepsilon^{\frac{\gamma-1}{\gamma}} \int_0^T \int_{\Gamma_\varepsilon}
\frac{\varrho_\varepsilon|\mathbf{u}_\varepsilon|^2}{d_\Omega^2(x)}dxdt
\to 0 \,\, \textrm{ as } \,\, \varepsilon \to 0.
\end{equation}
Then
\begin{equation} \label{conv-inv-lim}
    \sup_{t \in [0,T]} \left[  \| \varrho_{\varepsilon}  - \varrho^E
    \|_{L^\gamma(\Omega)} + \int_\Omega  \varrho_{\varepsilon}
    |\mathbf{u}_{\varepsilon}-\mathbf{u}^E|^2
    dx  \right] \to 0 \mbox{ as } \varepsilon \to 0.
\end{equation}
\end{theorem}

\begin{remark} \label{th-rm}
  Differently from Bardos and Nguyen \cite{BaNg}, we require a rate,
  for the $L^\gamma$-norm of the density $\varrho_\varepsilon$, in
  terms of $\varepsilon$. Moreover, with respect to the requirements
  of Sueur \cite{Su}, we assume only the condition on the kinetic
  energy (\ref{cond-conv})$_2$. However, our assumption is stronger
  and implies the analogous introduced in \cite{Su}.  As remarked by
  Bardos and Nguyen \cite{BaNg} and Sueur \cite{Su}, the assumptions
  (\ref{cond-conv}) are implied (when the density is constant), thanks
  to the Hardy's inequality, by the condition
  \begin{equation*}
  \varepsilon \int_0^T \int_{\Gamma_\varepsilon} |\nabla_x
  \mathbf{u}|^2 dxdt \to 0 \ \text{as} \ \varepsilon \to 0
\end{equation*}
used by Kato \cite{Ka} in the incompressible case.
\end{remark}

\subsection{Preliminary lemma and a priori estimates}
The following Lemma holds  (see Bresch et al. \cite{BDL}; Lemma 2)
\begin{lemma} \label{l1}
  Let $(\varrho_{\varepsilon n}, \mathbf{u}_{\varepsilon n})$ be smooth
  solution of \eqref{cont}--\eqref{mom}. Then the following identity holds
  \begin{equation}
    \begin{aligned} \label{lemm-1} 
    \frac{1}{2} \frac{d}{dt} &\int \varrho_{\varepsilon n} |\nabla_x \log \varrho_{\varepsilon n}|^2 
    + \int \nabla_x \text{div}_x \mathbf{u}_{\varepsilon n} \cdot \nabla_x \varrho_{\varepsilon n}
\\
 & +\int \varrho_{\varepsilon n}
    \mathbb{D}(\mathbf{u}_{\varepsilon n}) : \nabla_x \log \varrho_{\varepsilon n} \otimes \nabla_x \log
    \varrho_{\varepsilon n} = 0.
  \end{aligned}
\end{equation}
\end{lemma}
\begin{proof}
See Bresch et al. \cite{BDL}.
\end{proof}

Now, from the continuity equation (\ref{cont}) and the energy inequality (\ref{ee}) we can deduce the following a priori estimates
\begin{equation*}
\left\Vert \varrho_{\varepsilon n}\right\Vert _{ L^{\infty}\left(0,T;L^{1}\left(\Omega\right) \cap L^{\gamma}\left(\Omega\right)\right)}
\leq C, \,\,\,  \left\Vert \sqrt{\varrho_{\varepsilon n}}\mathbf{u}_{\varepsilon n}\right\Vert _{L^{\infty}\left(0,T;L^{2}\left(\Omega\right)\right)}\leq C,
\end{equation*}
\begin{equation} \label{est_1}
  \left\Vert \sqrt{\varrho_{\varepsilon n }}\mathbb{D}\left(\mathbf{u}_{\varepsilon n}\right)\right\Vert _{L^{2}\left(0,T;L^{2}\left(\Omega\right)\right)}
  \leq C, \,\,\, \left\Vert \varrho_{\varepsilon n}^{1/3}\mathbf{u}_{\varepsilon n}\right\Vert _{L^{3}\left((0,T)\times \Omega, \right)}\leq C, 
\end{equation}
Moreover, Lemma \ref{l1} yields
\begin{equation} \label{est_2}
\left\Vert \nabla\sqrt{\varrho_{\varepsilon n}}\right\Vert_{L^{\infty}\left(0,T;L^{2}\left(\Omega\right)\right)}\leq C. 
\end{equation}
Finally, by (\ref{est_2}) and the first estimate in (\ref{est_1}), we can conclude that
\begin{equation} \label{est_4}
\left\Vert \sqrt{\varrho_{\varepsilon n}}\right\Vert _{L^{\infty}\left(0,T;W^{1,2}\left(\Omega\right)\right)}\leq C.
\end{equation}

\section{Convergence}
\subsection{Relative energy inequality}
Inspired by Bresch et al. \cite{BGL} (see also \cite{BNV-1}, \cite{BNV-2}),
we introduce an energy functional of the following type
\begin{equation} \label{entr-funct}
  \begin{aligned}
  \mathcal{E}(\varrho_\varepsilon, \mathbf{v}_\varepsilon,
  &\mathbf{w}_\varepsilon|\varrho^E, \overline{\mathbf{v}},\overline{\mathbf{w}}) (T,\cdot)
  \\
  =&   \int_{\Omega}  \left(\frac{1}{2}  \varrho_\varepsilon(|\mathbf{v}_\varepsilon
  -  \overline{\mathbf{v}}|^2 + |\mathbf{w}_\varepsilon -
  \overline{\mathbf{w}}|^2) \right)  (T,\cdot)  dx
   \\
 &  \quad
  + \int_{\Omega} (H(\varrho_\varepsilon) - H(\varrho^E)
  - H'(\varrho^E)(\varrho_\varepsilon - \varrho^E)) (T,\cdot)  dx
\\
&\quad +\varepsilon \int_0^T \int_\Omega \varrho_\varepsilon
\left( \left| \frac{\mathcal{S}(\mathbf{u}_\varepsilon)}{\sqrt{\varrho_\varepsilon}}
-\mathbb{D}(\overline{\mathbf{u}})
\right|^2 + \left| \frac{\mathcal{A}(\mathbf{u}_\varepsilon)}{\sqrt{\varrho_\varepsilon}}
-\mathbb{A}(\overline{\mathbf{u}})
\right|^2\right) dxdt
\end{aligned}
\end{equation}
where  $(\varrho_\varepsilon,\mathbf{v}_\varepsilon,\mathbf{w}_\varepsilon)$
is a weak solution
to the system (\ref{cont-ag}) - (\ref{mom2-ag}) and
$(\varrho^E,\overline{\mathbf{v}},\overline{\mathbf{w}})$ are such that
\begin{equation} \label{v_bar} \overline{\mathbf{v}} =
  \overline{\mathbf{u}} + 
  \widetilde{\delta}(\varepsilon) \nabla_{x}\log \varrho^E
\end{equation}
with $\overline{\mathbf{u}}$ smooth velocity field satisfying $\overline{\mathbf{u}}|_{\partial \Omega}=0$ and
$\overline{\mathbf{w}} = \widetilde{\delta}(\varepsilon)\nabla_{x} \log \varrho^E$
such that $\widetilde{\delta}(\varepsilon) \to 0$ as $\varepsilon \to 0$.

\begin{remark} \label{def-w}
  Relation \eqref{v_bar} can be written in a more general way, defining
  $\overline{\mathbf{w}}=\widetilde{\delta}(\varepsilon)\nabla_{x}
  \log r$, with $r$ arbitrary smooth function. However, for
  technical reasons due to the derivation of the relative energy
  inequality, and also because $r$ will play the role of the density of the
  compressible Euler system, we define already here
  $\overline{\mathbf{w}}=\widetilde{\delta}(\varepsilon)\nabla_{x}
  \log \varrho^E$. The expression for $\overline{\mathbf{u}}$ will be
  introduced later in Section~\ref{sec:inviscid-lim}. 
\end{remark}

In the following we will derive a relative energy inequality
satisfied by the weak solution of the ``augmented system" 
\eqref{cont-ag}--\eqref{mom2-ag}.

Now, thanks to the energy inequality \eqref{k-entropy}, we can write
\begin{equation} \label{step-1}
  \begin{aligned}
    \mathcal{E}(T,\cdot) - E(0,\cdot) \leq & \int_\Omega \left(
      \frac{1}{2} \varrho_\varepsilon |\overline{\mathbf{v}}|^2 -
      \varrho_\varepsilon \mathbf{v}_\varepsilon \cdot
      \overline{\mathbf{v}} + \frac{1}{2}\varrho_\varepsilon
      |\overline{\mathbf{w}}|^2 - \varrho_\varepsilon
      \mathbf{w}_\varepsilon \cdot\overline{\mathbf{w}} \right)
    (T,\cdot)dx
    \\
    &\quad - \int_\Omega \left( \frac{1}{2} \varrho_\varepsilon
      |\overline{\mathbf{v}}|^2 - \varrho_\varepsilon
      \mathbf{v}_\varepsilon \cdot \overline{\mathbf{v}} +
      \frac{1}{2}\varrho_\varepsilon |\overline{\mathbf{w}}|^2 -
      \varrho_\varepsilon \mathbf{w}_\varepsilon
      \cdot\overline{\mathbf{w}} \right) (0,\cdot)dx
    \\
    &\quad - \int_\Omega \left( H(\varrho^E) +
      H'(\varrho^E)(\varrho_\varepsilon-\varrho^E) \right) (T,\cdot)dx
    \\
    &\quad + \int_\Omega \left( H(\varrho^E) +
      H'(\varrho^E)(\varrho_\varepsilon-\varrho^E) \right) (0,\cdot)dx
    \\
    &\quad -2\varepsilon \int_0^T \int_\Omega \left(
      \sqrt{\varrho_\varepsilon} \mathcal{S}(\mathbf{u}_\varepsilon)
      \mathbb{D}(\overline{\mathbf{u}}) + \sqrt{\varrho_\varepsilon}
      \mathcal{A}(\mathbf{u}_\varepsilon)
      \mathbb{A}(\overline{\mathbf{u}}) \right) dxdt
    \\
    &\quad +\varepsilon \int_0^T \int_\Omega \left(
      \varrho_\varepsilon \left( |\mathbb{D}(\overline{\mathbf{u}})|^2
        + |\mathbb{A}(\overline{\mathbf{u}})|^2 \right) -
      \frac{p'(\varrho_\varepsilon)}{\varrho_\varepsilon} |\nabla_x
      \varrho_\varepsilon|^2 \right) dxdt,
  \end{aligned}
\end{equation}
where
\begin{equation} \label{E0}
  \begin{aligned}
    E(0,\cdot)   = &    \int_{\Omega} \left( \frac{1}{2}
    \varrho_{0,\varepsilon}(|\mathbf{v}_{0,\varepsilon} -
    \overline{\mathbf{v}}|^2 + |\mathbf{w}_{0,\varepsilon} -
    \overline{\mathbf{w}}|^2)  \right)  (0,\cdot)  dx
\\
   &\quad  + \int_{\Omega} (H(\varrho_{0,\varepsilon})
    - H(\varrho^E) - H'(\varrho^E)(\varrho_{0,\varepsilon} - \varrho^E))
    (0,\cdot)  dx
  \end{aligned}
\end{equation}

Now, we recall the weak formulation to the system \eqref{cont-ag}--\eqref{mom2-ag}
\begin{equation} \label{weak-cont}
  \begin{aligned}
- \int_{\Omega} \varrho_\varepsilon(T,\cdot) \varphi(T,\cdot) dx
+ & \int_{\Omega} \varrho_\varepsilon(0,\cdot) \varphi(0,\cdot) dx
\\
&\quad + \int_0^T \int_{\Omega}
\varrho_\varepsilon \partial_t \varphi dxdt + \int_0^T \int_{\Omega}
\varrho_\varepsilon \mathbf{u}_\varepsilon \cdot \nabla_{x} \varphi dxdt
=0,
\end{aligned}
\end{equation}

\begin{equation} \label{weak-w} 
  \begin{aligned}
    -\int_{\Omega}  \varrho_\varepsilon &\mathbf{v}_\varepsilon(T,\cdot) \cdot \boldsymbol{\varphi}(T,\cdot) dx
    +\int_{\Omega}
    \varrho_\varepsilon \mathbf{v}_\varepsilon(0,\cdot) \cdot \boldsymbol{\varphi}(0,\cdot) dx
  +\int_0^T \int_{\Omega} \varrho_\varepsilon \mathbf{v}_\varepsilon
  \cdot \partial_t \boldsymbol{\varphi} dxdt
  \\
  &\quad +\int_0^T \int_{\Omega} (\varrho_\varepsilon \mathbf{v}_\varepsilon \otimes \mathbf{u}_\varepsilon) : \nabla_{x} \boldsymbol{\varphi} dxdt 
  -\varepsilon\int_0^T \int_{\Omega} \varrho_\varepsilon \mathbb{D}(\mathbf{v}_\varepsilon) : \nabla_{x} \boldsymbol{\varphi} dxdt
\\
  &\quad  -\varepsilon\int_0^T \int_{\Omega} \varrho_\varepsilon \mathbb{A}(\mathbf{v}_\varepsilon) : \nabla_{x} \boldsymbol{\varphi} dxdt
    + \varepsilon\int_0^T \int_{\Omega}
    \varrho_\varepsilon \nabla_x \mathbf{w}_\varepsilon: \nabla_{x} \boldsymbol{\varphi} dxdt
\\ 
    &\quad -r_1\int_0^T \int_{\Omega}
    \varrho_\varepsilon |\mathbf{v}_\varepsilon - \mathbf{w}_\varepsilon|(\mathbf{v}_\varepsilon - \mathbf{w}_\varepsilon)
    \cdot \boldsymbol{\varphi}   + \int_0^T \int_{\Omega}   p(\varrho_\varepsilon) \mbox{div}_{x}  \boldsymbol{\varphi} dxdt  = 0,
  \end{aligned}
\end{equation}

\begin{equation} \label{weak-v-bar}
  \begin{aligned}
    -\int_{\Omega}  \varrho_\varepsilon &\mathbf{w}_\varepsilon  (T,\cdot) \cdot \boldsymbol{\varphi}(T,\cdot) dx
    +\int_{\Omega}  \varrho_\varepsilon \mathbf{w}_\varepsilon(0,\cdot) \cdot \boldsymbol{\varphi}(0,\cdot) dx
    +\int_0^T \int_{\Omega} \varrho_\varepsilon \mathbf{w}_\varepsilon  \cdot \partial_t \boldsymbol{\varphi} dxdt
\\
   &\quad  +\int_0^T \int_{\Omega} (\varrho_\varepsilon
    \mathbf{w}_\varepsilon \otimes \mathbf{u}_\varepsilon) : \nabla_{x} \boldsymbol{\varphi} dxdt 
    -\varepsilon\int_0^T \int_{\Omega} \varrho_\varepsilon \nabla_x^{\top}
    \mathbf{u}_\varepsilon : \nabla_{x} \boldsymbol{\varphi} dxdt = 0
\end{aligned}
\end{equation}
where the viscous terms read as in \eqref{visc-terms-D}--\eqref{visc-terms-ut}.

First, we test the continuity equation \eqref{weak-cont} by
$\frac{1}{2}|\overline{\mathbf{v}}|^2$ and
$\frac{1}{2}|\overline{\mathbf{w}}|^2$, respectively. We have,
\begin{equation} \label{test-w2}
  \begin{aligned}
  - \int_\Omega \frac{1}{2}   (\varrho_\varepsilon \cdot |\overline{\mathbf{v}}|^2)(T,\cdot)dx + \int_\Omega
 & \frac{1}{2} (\varrho_\varepsilon  \cdot |\overline{\mathbf{v}}|^2)(0,\cdot)dx
  \\
 &\quad + \int_0^T \int_\Omega \frac{1}{2} \left( \varrho_\varepsilon \partial_t
    |\overline{\mathbf{v}}|^2 + \varrho_\varepsilon
    \mathbf{u}_\varepsilon \cdot \nabla |\overline{\mathbf{v}}|^2
  \right)dxdt =0,
\end{aligned}
\end{equation}
\begin{equation} \label{test-v2}
  \begin{aligned}
  - \int_\Omega \frac{1}{2}
  (\varrho\varepsilon \cdot |\overline{\mathbf{w}}|^2)(T,\cdot)dx + \int_\Omega&
  \frac{1}{2}(\varrho_\varepsilon \cdot |\overline{\mathbf{w}}|^2)(0,\cdot)dx
   \\
 &\quad
  + \int_0^T \int_\Omega \frac{1}{2} \left( \varrho_\varepsilon \partial_t
    |\overline{\mathbf{w}}|^2 + \varrho_\varepsilon
    \mathbf{u}_\varepsilon \cdot \nabla |\overline{\mathbf{w}}|^2
  \right) dxdt =0.
\end{aligned}
\end{equation}
Now, we test the equation \eqref{weak-w} by $\overline{\mathbf{v}}$. We have
\begin{equation} \label{test-w}
  \begin{aligned}
    -\int_{\Omega} &
    \varrho_\varepsilon \mathbf{v}_\varepsilon(T,\cdot) \cdot \overline{\mathbf{v}}(T,\cdot) dx
    +\int_{\Omega}
    \varrho_\varepsilon \mathbf{v}_\varepsilon(0,\cdot) \cdot \overline{\mathbf{v}}(0,\cdot) dx
  +\int_0^T \int_{\Omega} \varrho_\varepsilon \mathbf{v}_\varepsilon
  \cdot \partial_t \overline{\mathbf{v}} dxdt
  \\
 & +\int_0^T \int_{\Omega} (\varrho_\varepsilon \mathbf{v}_\varepsilon \otimes \mathbf{u}_\varepsilon) : \nabla_{x} \overline{\mathbf{v}} dxdt 
  -\varepsilon\int_0^T \int_{\Omega} \varrho_\varepsilon \mathbb{D}(\mathbf{v}_\varepsilon) : \nabla_{x} \overline{\mathbf{v}}dxdt
\\
 & -\varepsilon\int_0^T \int_{\Omega} \varrho_\varepsilon \mathbb{A}(\mathbf{v}_\varepsilon) : \nabla_{x} \overline{\mathbf{v}} dxdt
    + \varepsilon\int_0^T \int_{\Omega}
    \varrho_\varepsilon \nabla_x \mathbf{w}_\varepsilon: \nabla_{x} \overline{\mathbf{v}} dxdt
\\
  &  -r_1\int_0^T \int_{\Omega}
    \varrho_\varepsilon |\mathbf{v}_\varepsilon - \mathbf{w}_\varepsilon|(\mathbf{v}_\varepsilon - \mathbf{w}_\varepsilon)
    \cdot \overline{\mathbf{v}} dxdt  + \int_0^T \int_{\Omega}  p(\varrho_\varepsilon) \mbox{div}_{x}  \overline{\mathbf{v}} dxdt  =0,
  \end{aligned}
\end{equation}
Next, we test the equation \eqref{weak-v-bar} by
$\overline{\mathbf{w}}$. We have
\begin{equation} \label{test-v-bar}
  \begin{aligned}
    -\int_{\Omega} &    \varrho_\varepsilon \mathbf{w}_\varepsilon(T,\cdot) \cdot \overline{\mathbf{w}}(T,\cdot) dx
    +\int_{\Omega}   \varrho_\varepsilon \mathbf{w}_\varepsilon(0,\cdot) \cdot \overline{\mathbf{w}}(0,\cdot) dx
    +\int_0^T \int_{\Omega} \varrho_\varepsilon \mathbf{w}_\varepsilon  \cdot \partial_t \overline{\mathbf{w}} dxdt
    \\
 & +\int_0^T \int_{\Omega} (\varrho_\varepsilon \mathbf{w}_\varepsilon \otimes \mathbf{u}_\varepsilon) : \nabla_{x} \overline{\mathbf{w}} dxdt 
 -\varepsilon\int_0^T \int_{\Omega} \varrho_\varepsilon \nabla_x^{\top} \mathbf{u}_\varepsilon : \nabla_{x} \overline{\mathbf{w}} dxdt =0.
\end{aligned}
\end{equation}
Finally, using the continuity equation \eqref{cont-ag}, we have
\begin{equation}  \label{test-H}
 \hspace{-0.3 cm} \begin{aligned}
  & \int_0^T\!\! \int_\Omega \left[ \partial_t \left( H(\varrho^E)
 + H'(\varrho^E)(\varrho_\varepsilon - \varrho^E) \right) \right] dxdt
\\
& \,\,= \int_0^T \!\!\int_\Omega \left[ H'(\varrho^E)\partial_t \varrho^E + \partial_t (H'(\varrho^E))(\varrho_\varepsilon - \varrho^E)
  + H'(\varrho^E)\partial_t \varrho_\varepsilon - H'(\varrho^E) \partial_t \varrho^E
  \right] dxdt
\\
 &\,\, = \int_0^T\!\! \int_\Omega \left[ \partial_t (H'(\varrho^E))(\varrho_\varepsilon - \varrho^E)
  - H'(\varrho^E)\textrm{div}_x(\varrho_\varepsilon \mathbf{u}_\varepsilon)  \right] dxdt
\\
& \,\, = \int_0^T\!\! \int_\Omega  \left[  \partial_t  (H'(\varrho^E))(\varrho_\varepsilon - \varrho^E) + \varrho_\varepsilon
    \mathbf{u}_\varepsilon \nabla_x (H'(\varrho^E))  \right]  dxdt.
\end{aligned}
\end{equation}
Plugging \eqref{test-w2}--\eqref{test-H} in \eqref{step-1}, we obtain
\begin{equation} \label{step-2}
\begin{aligned}
  \mathcal{E}(T,\cdot) - E(0,\cdot)
  & \leq  \int_0^T \int_\Omega \frac{1}{2} \left( \varrho_\varepsilon \partial_t
    |\overline{\mathbf{v}}|^2
    + \varrho_\varepsilon \mathbf{u}_\varepsilon \cdot \nabla |\overline{\mathbf{v}}|^2 \right) dxdt
  \\
 &\quad +\int_0^T \int_\Omega \frac{1}{2} \left( \varrho_\varepsilon \partial_t
    |\overline{\mathbf{w}}|^2
    + \varrho_\varepsilon \mathbf{u}_\varepsilon \cdot \nabla |\overline{\mathbf{w}}|^2 \right)
    dxdt
\\
   &\quad - \int_0^T \int_\Omega  \left(  \varrho_\varepsilon \mathbf{v}_\varepsilon \cdot \partial_t \overline{\mathbf{v}}
    + \varrho_\varepsilon \mathbf{v}_\varepsilon \otimes \mathbf{u}_\varepsilon : \nabla_x \overline{\mathbf{v}}
    + p(\varrho_\varepsilon) \textrm{div}_x \overline{\mathbf{v}}  \right)  dxdt
\\
   &\quad  +\varepsilon\int_0^T \int_{\Omega} \varrho_\varepsilon \mathbb{D}(\mathbf{v}_\varepsilon) : \nabla_{x} \overline{\mathbf{v}}dxdt
   +\varepsilon\int_0^T \int_{\Omega} \varrho_\varepsilon \mathbb{A}(\mathbf{v}_\varepsilon) : \nabla_{x} \overline{\mathbf{v}} dxdt
   \\
   &\quad  - \varepsilon\int_0^T \int_{\Omega}  \varrho_\varepsilon \nabla_x \mathbf{w}_\varepsilon: \nabla_{x} \overline{\mathbf{v}} dxdt
   +\varepsilon\int_0^T \int_{\Omega} \varrho_\varepsilon \nabla_x^{\top} \mathbf{u}_\varepsilon : \nabla_{x} \overline{\mathbf{w}} dxdt
   \\
    &\quad +r_1\int_0^T \int_{\Omega}  \varrho_\varepsilon |\mathbf{v}_\varepsilon - \mathbf{w}_\varepsilon|(\mathbf{v}_\varepsilon - \mathbf{w}_\varepsilon)
    \cdot \overline{\mathbf{v}}  dxdt
\\
   &\quad - \int_0^T \int_\Omega \left(  \varrho_\varepsilon \mathbf{w}_\varepsilon \cdot\partial_t\overline{\mathbf{w}} 
    +  \varrho_\varepsilon \mathbf{w}_\varepsilon \otimes \mathbf{u}_\varepsilon : \nabla_x \overline{\mathbf{w}} 
  \right)   dxdt  
\\
 &\quad  - \int_0^T \int_\Omega  \left[  \partial_t (H'(\varrho^E))(\varrho_\varepsilon - \varrho^E)  + \varrho_\varepsilon \mathbf{u}_\varepsilon \nabla_x (H'(\varrho^E))
  \right]   dxdt
\\
  &\quad   -2\varepsilon   \int_0^T \int_\Omega  \left(  \sqrt{\varrho_\varepsilon}  \mathcal{S}(\mathbf{u}_\varepsilon)
    \mathbb{D}(\overline{\mathbf{u}})   +   \sqrt{\varrho_\varepsilon}  \mathcal{A}(\mathbf{u}_\varepsilon)   \mathbb{A}(\overline{\mathbf{u}})     \right)   dxdt
  \\
   &\quad+\varepsilon \int_0^T \int_\Omega \left(  \varrho_\varepsilon \left( |\mathbb{D}(\overline{\mathbf{u}})|^2 + |\mathbb{A}(\overline{\mathbf{u}})|^2 \right) -
    \frac{p'(\varrho_\varepsilon)}{\varrho_\varepsilon} |\nabla_x   \varrho_\varepsilon|^2 \right) dxdt.
\end{aligned}
\end{equation}
Rearranging \eqref{step-2}, we obtain
\begingroup
\allowdisplaybreaks[1] 
\begin{equation} \label{step-3}
  \begin{aligned}
  \mathcal{E}(T,\cdot) - E(0,\cdot)
  &\leq \int_0^T \int_\Omega \varrho_\varepsilon \left( \partial_t \overline{\mathbf{v}} \cdot \left( \overline{\mathbf{v}} - \mathbf{v}_\varepsilon \right)
    +   \left( \nabla_x \overline{\mathbf{v}} \cdot \mathbf{u}_\varepsilon \right) \cdot \left( \overline{\mathbf{v}} - \mathbf{v}_\varepsilon \right) \right) dxdt
\\
  &\quad + \int_0^T \int_\Omega \varrho_\varepsilon \left( \partial_t \overline{\mathbf{w}} \cdot \left( \overline{\mathbf{w}} - \mathbf{w}_\varepsilon \right)
    +     \left( \nabla_x \overline{\mathbf{w}} \cdot \mathbf{u}_\varepsilon \right) \cdot \left( \overline{\mathbf{w}} - \mathbf{w}_\varepsilon \right) \right) dxdt
\\
   &\quad +\varepsilon\int_0^T \int_{\Omega} \varrho_\varepsilon \mathbb{D}(\mathbf{v}_\varepsilon) : \nabla_{x} \overline{\mathbf{v}}dxdt
   +\varepsilon\int_0^T \int_{\Omega} \varrho_\varepsilon \mathbb{A}(\mathbf{v}_\varepsilon) : \nabla_{x} \overline{\mathbf{v}} dxdt
   \\
   &\quad - \varepsilon\int_0^T \int_{\Omega}   \varrho_\varepsilon \nabla_x \mathbf{w}_\varepsilon: \nabla_{x} \overline{\mathbf{v}} dxdt
   +\varepsilon\int_0^T \int_{\Omega} \varrho_\varepsilon \nabla_x^{\top} \mathbf{u}_\varepsilon : \nabla_{x} \overline{\mathbf{w}} dxdt 
   \\
&\quad  +r_1\int_0^T \int_{\Omega}  \varrho_\varepsilon |\mathbf{v}_\varepsilon - \mathbf{w}_\varepsilon|(\mathbf{v}_\varepsilon - \mathbf{w}_\varepsilon)
    \cdot \overline{\mathbf{v}}  dxdt
    \end{aligned}
  \end{equation}
\begin{equation*}
  \hspace{3 cm}\begin{aligned}
    &\quad - \int_0^T \int_\Omega \left[\partial_t (H'(\varrho^E))(\varrho_\varepsilon - \varrho^E)
      + \varrho_\varepsilon \mathbf{u}_\varepsilon \nabla_x (H'(\varrho^E)) + p(\varrho_\varepsilon) \textrm{div}_x \overline{\mathbf{v}}
    \right]   dxdt
\\
&\quad
-2\varepsilon \int_0^T \int_\Omega \left( \sqrt{\varrho_\varepsilon} \mathcal{S}(\mathbf{u}_\varepsilon)
  \mathbb{D}(\overline{\mathbf{u}})  +  \sqrt{\varrho_\varepsilon}  \mathcal{A}(\mathbf{u}_\varepsilon)  \mathbb{A}(\overline{\mathbf{u}})  \right)  dxdt
\\
&\quad +\varepsilon \int_0^T \int_\Omega \left( \varrho_\varepsilon \left( |\mathbb{D}(\overline{\mathbf{u}})|^2
    + |\mathbb{A}(\overline{\mathbf{u}})|^2 \right) - \frac{p'(\varrho_\varepsilon)}{\varrho_\varepsilon} |\nabla_x \varrho_\varepsilon|^2 \right) dxdt
  \end{aligned}
\end{equation*}
\endgroup
We multiply \eqref{cont-E} by $H'(\varrho^E)$
\begin{equation*}
  \int_0^T \int_\Omega
  \left[
  H'(\varrho^E)\partial_t \varrho^E + H'(\varrho^E) \textrm{div}_x (\varrho^E \mathbf{u}^E)
  \right]
  dxdt 
  = 0.
\end{equation*}
Performing integration by parts, we obtain
\begin{equation*}
  \int_0^T \int_\Omega
  \left[
  H'(\varrho^E)\partial_t \varrho^E - \varrho^E \nabla_x (H'(\varrho^E)) \cdot \mathbf{u}^E 
  \right]
  dxdt
  = 0.
\end{equation*}
From $p(\varrho^E) = H'(\varrho^E)\varrho^E - H(\varrho^E)$, we have $\varrho^E \nabla_x (H'(\varrho^E)) = \nabla_x
p(\varrho^E)$. 
Consequently, from the above relation, we have
\begin{equation*}
  \int_0^T \int_\Omega
  \left[
  H'(\varrho^E)\partial_t \varrho^E - \nabla_x p(\varrho^E) \cdot \mathbf{u}^E
  \right]
  dxdt
  = 0.
\end{equation*}
Performing again integration by parts, we obtain
\begin{equation} \label{obs} 
\int_0^T \int_\Omega 
H'(\varrho^E)\partial_t \varrho^E dxdt= - \int_0^T \int_\Omega
  p(\varrho^E) \textrm{div}_x \mathbf{u}^E dxdt.
\end{equation}
From \eqref{step-3}, we have
\begin{equation} \label{step-4}
  \begin{aligned}
  \mathcal{E}(T,\cdot) - E(0,\cdot)
  &\leq
  \int_0^T \int_\Omega \varrho_\varepsilon \left( \partial_t \overline{\mathbf{v}} \cdot \left( \overline{\mathbf{v}} - \mathbf{v}_\varepsilon \right)
    +
    \left( \nabla_x \overline{\mathbf{v}} \cdot \mathbf{u}_\varepsilon \right) \cdot \left( \overline{\mathbf{v}} - \mathbf{v}_\varepsilon \right) \right)
    dxdt
\\
 &\quad  + \int_0^T \int_\Omega \varrho_\varepsilon \left( \partial_t \overline{\mathbf{w}} \cdot \left( \overline{\mathbf{w}} - \mathbf{w}_\varepsilon \right)
    +    \left( \nabla_x \overline{\mathbf{w}} \cdot \mathbf{u}_\varepsilon \right) \cdot \left( \overline{\mathbf{w}} - \mathbf{w}_\varepsilon \right) \right)  dxdt
\\
  &\quad  +\varepsilon\int_0^T \int_{\Omega} \varrho_\varepsilon \mathbb{D}(\mathbf{v}_\varepsilon) : \nabla_{x} \overline{\mathbf{v}}dxdt
    +\varepsilon\int_0^T \int_{\Omega} \varrho_\varepsilon  \mathbb{A}(\mathbf{v}_\varepsilon) : \nabla_{x}
    \overline{\mathbf{v}} dxdt
\\
&\quad    +r_1\int_0^T \int_{\Omega}  \varrho_\varepsilon |\mathbf{v}_\varepsilon - \mathbf{w}_\varepsilon|(\mathbf{v}_\varepsilon - \mathbf{w}_\varepsilon)
    \cdot \overline{\mathbf{v}}  dxdt
    \\
    &\quad  - \varepsilon\int_0^T \int_{\Omega}  \varrho_\varepsilon \nabla_x \mathbf{w}_\varepsilon: \nabla_{x} \overline{\mathbf{v}} dxdt
    +\varepsilon\int_0^T \int_{\Omega} \varrho_\varepsilon \nabla_x^{\top} \mathbf{u}_\varepsilon : \nabla_{x} \overline{\mathbf{w}} dxdt
  \end{aligned}
\end{equation}
\begin{equation*}
 \hspace{3 cm} \begin{aligned}
    &\quad - \int_0^T \int_\Omega \left[ \partial_t (H'(\varrho^E))(\varrho_\varepsilon - \varrho^E) + \varrho_\varepsilon \mathbf{u}_\varepsilon
      \nabla_x (H'(\varrho^E)) + p(\varrho_\varepsilon) \textrm{div}_x \overline{\mathbf{v}} \right] dxdt
    \\
    &\quad    -\int_0^T \int_\Omega \left[-p(\varrho^E)\textrm{div}_x \mathbf{u}^E - H'(\varrho^E)\partial_t \varrho^E  \right]  dxdt
    \\
    &\quad   -2\varepsilon  \int_0^T \int_\Omega \left( \sqrt{\varrho_\varepsilon} \mathcal{S}(\mathbf{u}_\varepsilon) \mathbb{D}(\overline{\mathbf{u}})
      + \sqrt{\varrho_\varepsilon} \mathcal{A}(\mathbf{u}_\varepsilon) \mathbb{A}(\overline{\mathbf{u}})  \right)  dxdt
    \\
    &\quad    +\varepsilon \int_0^T \int_\Omega \left( \varrho_\varepsilon \left( |\mathbb{D}(\overline{\mathbf{u}})|^2 
        + |\mathbb{A}(\overline{\mathbf{u}})|^2 \right) -  \frac{p'(\varrho_\varepsilon)}{\varrho_\varepsilon} |\nabla_x \varrho_\varepsilon|^2 \right) dxdt
  \end{aligned}
\end{equation*}
Now, from \eqref{mom-E} and using the continuity equation \eqref{cont-E}, we deduce
\begin{equation} \label{mom-E-1}
    \varrho^E(\partial_t \mathbf{u}^E + \mathbf{u}^E \cdot \nabla_x \mathbf{u}^E) + \nabla_x p(\varrho^E)=0.
\end{equation}
We multiply \eqref{mom-E-1} by $\frac{\varrho_\varepsilon}{\varrho^E}(\overline{\mathbf{u}}-\mathbf{u}_\varepsilon)$. We
obtain
\begin{equation} \label{mom-U-test}
  \varrho_\varepsilon \partial_t
  \mathbf{u}^E \cdot
  (\overline{\mathbf{u}}-\mathbf{u}_\varepsilon) + \varrho_\varepsilon
  \nabla_x \mathbf{u}^E \cdot \mathbf{u}^E \cdot
  (\overline{\mathbf{u}}-\mathbf{u}_\varepsilon) +
  \frac{\varrho_\varepsilon}{\varrho^E}\nabla_x p(\varrho^E) \cdot
  (\overline{\mathbf{u}}-\mathbf{u}_\varepsilon) =0.
\end{equation}
Now, we consider the first term on the right-hand-side of \eqref{step-4}. We have
\begin{equation} \label{v-terms}
 \hspace{-0.4 cm} \begin{aligned}
&  \int_0^T \int_\Omega \varrho_\varepsilon \left( \partial_t \overline{\mathbf{v}} \cdot \left( \overline{\mathbf{v}} - \mathbf{v}_\varepsilon \right)
    +
    \left( \nabla_x \overline{\mathbf{v}} \cdot \mathbf{u}_\varepsilon \right) \cdot \left( \overline{\mathbf{v}} - \mathbf{v}_\varepsilon \right) \right)   dxdt
\\
 & \,\,\, = \int_0^T \int_\Omega \varrho_\varepsilon \left( \partial_t \overline{\mathbf{u}} + \varepsilon \partial_t \nabla_x \log \varrho^E)
    \cdot \left( \overline{\mathbf{u}} + \varepsilon \nabla_x \log \varrho^E -
      \mathbf{u}_\varepsilon  - \varepsilon \nabla_x \log \varrho_\varepsilon \right)\right)  dxdt
\\
&\quad  +\int_0^T \int_\Omega \left( \nabla_x \overline{\mathbf{u}} + \varepsilon \nabla_x \nabla_x \log \varrho^E
  \right) \cdot \mathbf{u}_\varepsilon \cdot \left( \overline{\mathbf{u}} + \varepsilon \nabla_x \log \varrho^E
    - \mathbf{u}_\varepsilon  - \varepsilon \nabla_x \log \varrho_\varepsilon \right)  dxdt
\\
&\,\,\,  =   \int_0^T \int_\Omega \varrho_\varepsilon  \partial_t \overline{\mathbf{u}} \cdot (\overline{\mathbf{u}} - \mathbf{u}_\varepsilon) dxdt
  +  \varepsilon \int_0^T \int_\Omega \varrho_\varepsilon \partial_t \overline{\mathbf{u}}  \cdot 
  (\nabla_x \log \varrho^E - \nabla_x \log \varrho_\varepsilon)  dxdt
  \\
  &\quad   + \varepsilon \int_0^T \int_\Omega \varrho_\varepsilon
  \partial_t \nabla_x \log \varrho^E \cdot (\overline{\mathbf{u}} - \mathbf{u}_\varepsilon)
  dxdt + \int_0^T \int_\Omega   \varrho_\varepsilon \nabla_x \overline{\mathbf{u}}
  \cdot  \mathbf{u}_\varepsilon  \cdot  (\overline{\mathbf{u}} - \mathbf{u}_\varepsilon)   dxdt
   \\
  &\quad  +  \varepsilon^2 \int_0^T \int_\Omega \varrho_\varepsilon \partial_t \nabla_x \log \varrho^E  \cdot
  (\nabla_x \log \varrho^E - \nabla_x \log \varrho_\varepsilon)  dxdt
  \\
  &\quad +  \varepsilon \int_0^T \int_\Omega \varrho_\varepsilon \nabla_x \overline{\mathbf{u}}
  \cdot \mathbf{u}_\varepsilon  \cdot  (\nabla_x \log \varrho^E - \nabla_x \log \varrho_\varepsilon)  dxdt
  \\
  &  \quad + \varepsilon \int_0^T \int_\Omega \varrho_\varepsilon \nabla_x \nabla_x \log \varrho^E \cdot
  \mathbf{u}_\varepsilon \cdot (\overline{\mathbf{u}} -  \mathbf{u}_\varepsilon)   dxdt
  \\
  &\quad + \varepsilon^2   \int_0^T \int_\Omega   \varrho_\varepsilon \nabla_x \nabla_x \log \varrho^E \cdot   \mathbf{u}_\varepsilon
  \cdot (\nabla_x \log \varrho^E - \nabla_x \log  \varrho_\varepsilon)  dxdt,
\end{aligned}
\end{equation}
while, the second term on the right-hand-side of \eqref{step-4} reads as
\begin{equation} \label{w-terms}
  \begin{aligned}
  \int_0^T \int_\Omega &\varrho_\varepsilon \left( \partial_t \overline{\mathbf{w}} \cdot \left( \overline{\mathbf{w}} - \mathbf{w}_\varepsilon \right)
    +  \left( \nabla_x \overline{\mathbf{w}} \cdot \mathbf{u}_\varepsilon \right) \cdot \left( \overline{\mathbf{w}} - \mathbf{w}_\varepsilon \right) \right) dxdt
\\
&= \varepsilon^2 \int_0^T \int_\Omega \Big(\varrho_\varepsilon \partial_t \nabla_x \log \varrho^E \cdot (\nabla_x \log \varrho^E - \nabla_x \log \varrho_\varepsilon)
\\
& \hspace{1 cm}+ \varrho_\varepsilon
  \nabla_x \nabla_x \log \varrho^E \cdot \mathbf{u}_\varepsilon \cdot  \big(\nabla_x \log \varrho^E - \nabla_x \log \varrho_\varepsilon\big) \Big) dxdt.
\end{aligned}
\end{equation}
Consequently, using \eqref{mom-U-test}, we have
\begin{equation} \label{vw-terms}
  \begin{aligned}
    \int_0^T \int_\Omega & \varrho_\varepsilon \left( \partial_t \overline{\mathbf{v}} \cdot \left( \overline{\mathbf{v}} - \mathbf{v}_\varepsilon \right)
      +     \left( \nabla_x \overline{\mathbf{v}} \cdot \mathbf{u}_\varepsilon \right)  \cdot \left( \overline{\mathbf{v}} - \mathbf{v}_\varepsilon \right) \right)  dxdt
    \\
    & \quad +\int_0^T \int_\Omega \varrho_\varepsilon \left( \partial_t \overline{\mathbf{w}} \cdot \left( \overline{\mathbf{w}} - \mathbf{w}_\varepsilon \right)
    +    \left( \nabla_x \overline{\mathbf{w}} \cdot \mathbf{u}_\varepsilon  \right) \cdot \left( \overline{\mathbf{w}} -  \mathbf{w}_\varepsilon \right) \right)dxdt
  \\
  = &  \int_0^T \int_\Omega \varrho_\varepsilon \partial_t \overline{\mathbf{u}} \cdot (\overline{\mathbf{u}} - \mathbf{u}_\varepsilon)
  +  \varepsilon \int \int \varrho_\varepsilon \partial_t \overline{\mathbf{u}}   \cdot  (\nabla_x \log \varrho^E - \nabla_x \log \varrho_\varepsilon)dxdt
  \\
   & \quad + \varepsilon \int_0^T \int_\Omega \varrho_\varepsilon  \partial_t \nabla_x \log \varrho^E \cdot (\overline{\mathbf{u}} - \mathbf{u}_\varepsilon)dxdt
\end{aligned}
\end{equation}
\begin{equation*}
\begin{aligned}
&\quad  +   2\varepsilon^2 \int_0^T \int_\Omega \varrho_\varepsilon \partial_t \nabla_x \log \varrho^E   \cdot  (\nabla_x \log \varrho^E - \nabla_x \log \varrho_\varepsilon)dxdt
 \\
& \quad  + \int_0^T \int_\Omega   \varrho_\varepsilon \nabla_x \overline{\mathbf{u}}
\cdot   \mathbf{u}_\varepsilon   \cdot   (\overline{\mathbf{u}} - \mathbf{u}_\varepsilon)   dxdt
 \\
& \quad  +  \varepsilon \int_0^T \int_\Omega \varrho_\varepsilon \nabla_x \overline{\mathbf{u}}    \cdot \mathbf{u}_\varepsilon  \cdot
  (\nabla_x \log \varrho^E - \nabla_x \log \varrho_\varepsilon)dxdt
  \\
& \quad + \varepsilon \int_0^T \int_\Omega  \varrho_\varepsilon \nabla_x \nabla_x \log \varrho^E \cdot  \mathbf{u}_\varepsilon \cdot (\overline{\mathbf{u}} -
\mathbf{u}_\varepsilon)dxdt
\\
&\quad + 2\varepsilon^2 \int_0^T \int_\Omega  \varrho_\varepsilon \nabla_x \nabla_x \log \varrho^E \cdot
\mathbf{u}_\varepsilon \cdot (\nabla_x \log \varrho^E - \nabla_x \log  \varrho_\varepsilon)dxdt\\
& \quad - \int_0^T \int_\Omega \Big(\varrho_\varepsilon \partial_t  \mathbf{u}^E \cdot  (\overline{\mathbf{u}}-\mathbf{u}_\varepsilon) + \varrho_\varepsilon
\nabla_x \mathbf{u}^E \cdot \mathbf{u}^E \cdot   (\overline{\mathbf{u}}-\mathbf{u}_\varepsilon)
 \\
& \hspace{6 cm} +  \frac{\varrho_\varepsilon}{\varrho^E}\nabla_x p(\varrho^E) \cdot
  (\overline{\mathbf{u}}-\mathbf{u}_\varepsilon)\Big)dxdt.
\end{aligned}
\end{equation*}
Plugging \eqref{vw-terms} in \eqref{step-4}, we have
\begin{equation}  \label{step-5}
\hspace{-0.5 cm}\begin{aligned}
  \mathcal{E}(T,\cdot) - E(0,\cdot) & \leq \int_0^T \int_\Omega \varrho_\varepsilon
  \partial_t \overline{\mathbf{u}} \cdot (\overline{\mathbf{u}} - \mathbf{u}_\varepsilon)
  \\
  &\quad +  \varepsilon \int_0^T \int_\Omega \varrho_\varepsilon \partial_t \overline{\mathbf{u}} 
  \cdot  (\nabla_x \log \varrho^E - \nabla_x \log \varrho_\varepsilon)dxdt
  \\
 & \quad + \varepsilon \int_0^T \int_\Omega \varrho_\varepsilon
 \partial_t \nabla_x \log \varrho^E \cdot (\overline{\mathbf{u}} - \mathbf{u}_\varepsilon)dxdt
 \\
  &\quad +  2\varepsilon^2 \int_0^T \int_\Omega \varrho_\varepsilon \partial_t \nabla_x \log \varrho^E
  \cdot  (\nabla_x \log \varrho^E - \nabla_x \log \varrho_\varepsilon)dxdt
  \\
& \quad  + \int_0^T \int_\Omega  \varrho_\varepsilon \nabla_x \overline{\mathbf{u}}
  \cdot  \mathbf{u}_\varepsilon  \cdot  (\overline{\mathbf{u}} - \mathbf{u}_\varepsilon) dxdt
\\
  &\quad  +  \varepsilon \int_0^T \int_\Omega \varrho_\varepsilon \nabla_x \overline{\mathbf{u}}
  \cdot \mathbf{u}_\varepsilon
  \cdot   (\nabla_x \log \varrho^E - \nabla_x \log \varrho_\varepsilon)dxdt
  \\
& \quad + \varepsilon \int_0^T \int_\Omega
  \varrho_\varepsilon \nabla_x \nabla_x \log \varrho^E \cdot
  \mathbf{u}_\varepsilon \cdot (\overline{\mathbf{u}} -
  \mathbf{u}_\varepsilon)dxdt
  \\
  &\quad + 2\varepsilon^2 \int_0^T \int_\Omega
  \varrho_\varepsilon \nabla_x \nabla_x \log \varrho^E \cdot
  \mathbf{u}_\varepsilon \cdot (\nabla_x \log \varrho^E - \nabla_x \log
  \varrho_\varepsilon)dxdt
  \\
&  \quad - \int_0^T \int_\Omega\Big( \varrho_\varepsilon \partial_t
  \mathbf{u}^E \cdot  (\overline{\mathbf{u}}-\mathbf{u}_\varepsilon) + \varrho_\varepsilon
  \nabla_x \mathbf{u}^E \cdot \mathbf{u}^E \cdot
  (\overline{\mathbf{u}}-\mathbf{u}_\varepsilon)
\\
&  \hspace{6 cm} +   \frac{\varrho_\varepsilon}{\varrho^E}\nabla_x p(\varrho^E) \cdot
  (\overline{\mathbf{u}}-\mathbf{u}_\varepsilon)\Big)dxdt
      \\
&  \quad 
+\varepsilon\int_0^T \int_{\Omega} \varrho_\varepsilon \mathbb{D}(\mathbf{v}_\varepsilon) : \nabla_{x} \overline{\mathbf{v}}dxdt
    +\varepsilon\int_0^T \int_{\Omega} \varrho_\varepsilon \mathbb{A}(\mathbf{v}_\varepsilon) : \nabla_{x} \overline{\mathbf{v}} dxdt
    \\
    &\quad
    +\varepsilon\int_0^T \int_{\Omega} \varrho_\varepsilon \nabla_x^{\top} \mathbf{u}_\varepsilon : \nabla_{x} \overline{\mathbf{w}} dxdt
\end{aligned}
\end{equation}
\begin{equation*}
  \begin{aligned}
          \\
&  \quad 
    - \varepsilon\int_0^T \int_{\Omega}
    \varrho_\varepsilon \nabla_x \mathbf{w}_\varepsilon: \nabla_{x} \overline{\mathbf{v}} dxdt
    +r_1\int_0^T \int_{\Omega}
    \varrho_\varepsilon |\mathbf{v}_\varepsilon - \mathbf{w}_\varepsilon|(\mathbf{v}_\varepsilon - \mathbf{w}_\varepsilon)
    \cdot \overline{\mathbf{v}}   dxdt
    \\
&  \quad 
  - \int_0^T \int_\Omega  \left[  \partial_t (H'(\varrho^E))(\varrho_\varepsilon - \varrho^E)
  + \varrho_\varepsilon \mathbf{u}_\varepsilon \nabla_x (H'(\varrho^E)) + p(\varrho_\varepsilon) \textrm{div}_x \overline{\mathbf{v}}
  \right]  dxdt
  \\
&  \quad    -\int_0^T \int_\Omega     \left[   -p(\varrho^E)\textrm{div}_x \mathbf{u}^E
    - H'(\varrho^E)\partial_t \varrho^E  \right]   dxdt
  \\
&  \quad -2\varepsilon  \int_0^T \int_\Omega  \left( \sqrt{\varrho_\varepsilon}  \mathcal{S}(\mathbf{u}_\varepsilon)
    \mathbb{D}(\overline{\mathbf{u}})  + \sqrt{\varrho_\varepsilon} \mathcal{A}(\mathbf{u}_\varepsilon)   \mathbb{A}(\overline{\mathbf{u}})  \right)  dxdt
  \\
&  \quad +\varepsilon \int_0^T \int_\Omega \left(  \varrho_\varepsilon \left( |\mathbb{D}(\overline{\mathbf{u}})|^2
      + |\mathbb{A}(\overline{\mathbf{u}})|^2 \right) -  \frac{p'(\varrho_\varepsilon)}{\varrho_\varepsilon} |\nabla_x   \varrho_\varepsilon|^2 \right) dxdt
\end{aligned}
\end{equation*}
Now, we compute
\begin{equation} \label{H-prime}
  \begin{aligned}
  \partial_t (H'(\varrho^E)) &= - p'(\varrho^E)\textrm{div} \mathbf{u}^E
  - H''(\varrho^E)\nabla_x \varrho^E \cdot \mathbf{u}^E 
  \\[0.2 cm]
  & = - p'(\varrho^E)\textrm{div}_xn \mathbf{u}^E - \nabla_x (H'(\varrho^E)) \cdot \mathbf{u}^E.
\end{aligned}
\end{equation}
Consequently,
\begin{equation} \label{rew} 
  \begin{aligned}
    -\int_0^T& \int_\Omega \frac{\varrho_\varepsilon}{\varrho^E}\nabla_x p(\varrho^E)
    \cdot (\overline{\mathbf{u}} - \mathbf{u}_\varepsilon) dxdt  - \int_0^T \int_\Omega \partial_t
    (H'(\varrho^E))(\varrho_\varepsilon - \varrho^E)   dxdt
    \\
    &  - \int_0^T \int_\Omega  \varrho_\varepsilon \mathbf{u}_\varepsilon \nabla_x (H'(\varrho^E))
    dxdt   +  \int_0^T \int_\Omega  H'(\varrho^E)\partial_t \varrho^E   dxdt
    \\
    = & - \int_0^T \int_\Omega \varrho_\varepsilon \nabla_x (H'(\varrho^E))
    (\overline{\mathbf{u}} - \mathbf{u}_\varepsilon) dxdt
    + \int_0^T \int_\Omega p'(\varrho^E)(\varrho_\varepsilon - \varrho^E) \textrm{div} \mathbf{u}^E  dxdt
    \\
    & + \int_0^T \int_\Omega \nabla(H'(\varrho^E)) (\varrho_\varepsilon - \varrho^E) \cdot \mathbf{u}^E  dxdt
    - \int_0^T \int_\Omega \varrho_\varepsilon  \mathbf{u}_\varepsilon \nabla_x (H'(\varrho^E))  dxdt
    \\
    &  + \int_0^T \int_\Omega \varrho^E \nabla_x (H'(\varrho^E)) \cdot \mathbf{u}^E dxdt
    \\
    =  & \int_0^T \int_\Omega p'(\varrho^E)(\varrho_\varepsilon -  \varrho^E)\textrm{div} \mathbf{u}^E dxdt.
\end{aligned}
\end{equation}
Plugging \eqref{rew} in \eqref{step-5}, we obtain
\begin{equation*}
\begin{aligned}
  \mathcal{E}(T,\cdot) - E(0,\cdot)
& \leq  \int_0^T \int_\Omega \varrho_\varepsilon
  \partial_t \overline{\mathbf{u}} \cdot (\overline{\mathbf{u}} - \mathbf{u}_\varepsilon)
  +  \varepsilon \int \int \varrho_\varepsilon \partial_t \overline{\mathbf{u}} 
  \cdot  (\nabla_x \log \varrho^E - \nabla_x \log \varrho_\varepsilon)dxdt
  \\
 & \quad + \varepsilon \int_0^T \int_\Omega \varrho_\varepsilon
 \partial_t \nabla_x \log \varrho^E \cdot (\overline{\mathbf{u}} - \mathbf{u}_\varepsilon)dxdt
 \\
& \quad  +  2\varepsilon^2 \int_0^T \int_\Omega \varrho_\varepsilon \partial_t \nabla_x \log \varrho^E
  \cdot   (\nabla_x \log \varrho^E - \nabla_x \log \varrho_\varepsilon)dxdt\\
&  \quad + \int_0^T \int_\Omega 
  \varrho_\varepsilon \nabla_x \overline{\mathbf{u}}
  \cdot   \mathbf{u}_\varepsilon   \cdot  (\overline{\mathbf{u}} - \mathbf{u}_\varepsilon)   dxdt
  \\
& \quad   +  \varepsilon \int_0^T \int_\Omega \varrho_\varepsilon \nabla_x \overline{\mathbf{u}}
\cdot \mathbf{u}_\varepsilon   \cdot  (\nabla_x \log \varrho^E - \nabla_x \log \varrho_\varepsilon)dxdt\\
\end{aligned}
\end{equation*}
\begin{equation}  \label{step-6}
\hspace{1 cm} \begin{aligned}
    &\quad + \varepsilon \int_0^T \int_\Omega
    \varrho_\varepsilon \nabla_x \nabla_x \log \varrho^E \cdot
    \mathbf{u}_\varepsilon \cdot (\overline{\mathbf{u}} -
    \mathbf{u}_\varepsilon)dxdt
    \\
    &\quad + 2\varepsilon^2 \int_0^T \int_\Omega
    \varrho_\varepsilon \nabla_x \nabla_x \log \varrho^E \cdot
    \mathbf{u}_\varepsilon \cdot (\nabla_x \log \varrho^E - \nabla_x \log
    \varrho_\varepsilon)dxdt
    \\
    &  \quad - \int_0^T \int_\Omega \varrho_\varepsilon \partial_t
    \mathbf{u}^E \cdot  (\overline{\mathbf{u}}-\mathbf{u}_\varepsilon)
    + \varrho_\varepsilon \nabla_x \mathbf{u}^E \cdot \mathbf{u}^E \cdot
    (\overline{\mathbf{u}}-\mathbf{u}_\varepsilon) dxdt
    \\
    &\quad     +\varepsilon\int_0^T \int_{\Omega} \varrho_\varepsilon
    \mathbb{D}(\mathbf{v}_\varepsilon) : \nabla_{x} \overline{\mathbf{v}}dxdt
    +\varepsilon\int_0^T \int_{\Omega} \varrho_\varepsilon
    \mathbb{A}(\mathbf{v}_\varepsilon) : \nabla_{x} \overline{\mathbf{v}} dxdt
    \\
    &   \quad +r_1\int_0^T \int_{\Omega}  \varrho_\varepsilon |\mathbf{v}_\varepsilon
    - \mathbf{w}_\varepsilon|(\mathbf{v}_\varepsilon - \mathbf{w}_\varepsilon)
    \cdot \overline{\mathbf{v}}  dxdt
    \\
    &\quad  +\varepsilon\int_0^T \int_{\Omega} \varrho_\varepsilon
    \nabla_x^{\top} \mathbf{u}_\varepsilon : \nabla_{x} \overline{\mathbf{w}} dxdt
    - \varepsilon\int_0^T \int_{\Omega}  \varrho_\varepsilon
    \nabla_x \mathbf{w}_\varepsilon: \nabla_{x} \overline{\mathbf{v}} dxdt
\\
  &\quad   -\int_0^T \int_\Omega   \left[  -p(\varrho^E)\textrm{div}_x \mathbf{u}^E
    + p(\varrho_\varepsilon) \textrm{div}_x \overline{\mathbf{v}}
    - p'(\varrho^E)(\varrho_\varepsilon-\varrho^E)\text{div}_x \mathbf{u}^E  \right] dxdt
\\
   &\quad  -2\varepsilon  \int_0^T \int_\Omega \left( \sqrt{\varrho_\varepsilon} 
    \mathcal{S}(\mathbf{u}_\varepsilon) \mathbb{D}(\overline{\mathbf{u}})  +
    \sqrt{\varrho_\varepsilon}   \mathcal{A}(\mathbf{u}_\varepsilon)
    \mathbb{A}(\overline{\mathbf{u}})  \right)  dxdt
  \\
  &\quad +\varepsilon \int_0^T \int_\Omega \left(
    \varrho_\varepsilon \left( |\mathbb{D}(\overline{\mathbf{u}})|^2
      + |\mathbb{A}(\overline{\mathbf{u}})|^2 \right) -
    \frac{p'(\varrho_\varepsilon)}{\varrho_\varepsilon} |\nabla_x
    \varrho_\varepsilon|^2 \right) dxdt
  \end{aligned}
\end{equation}
After some calculations, we can rewrite \eqref{step-6} as follows
\begin{equation*}
\begin{aligned}
  \mathcal{E}(T,\cdot) - E(0,\cdot) & \leq
  - \int_0^T \int_\Omega \varrho_\varepsilon
  \partial_t \mathbf{v}_{bl} \cdot (\overline{\mathbf{u}} -
  \mathbf{u}_\varepsilon)
 +  \varepsilon \int_0^T \int_\Omega \varrho_\varepsilon \partial_t \overline{\mathbf{u}} 
  \cdot  (\nabla_x \log \varrho^E - \nabla_x \log \varrho_\varepsilon)dxdt
  \\
 & \quad + \varepsilon \int_0^T \int_\Omega \varrho_\varepsilon
  \partial_t \nabla_x \log \varrho^E \cdot (\overline{\mathbf{u}} -
  \mathbf{u}_\varepsilon)dxdt
  \\
  & \quad +  2\varepsilon^2 \int_0^T \int_\Omega \varrho_\varepsilon \partial_t \nabla_x \log \varrho^E
  \cdot  (\nabla_x \log \varrho^E - \nabla_x \log
  \varrho_\varepsilon)dxdt
  \\
&  \quad + \int_0^T \int_\Omega  \varrho_\varepsilon \nabla_x
\overline{\mathbf{u}}  \cdot  \mathbf{u}_\varepsilon  \cdot  (\overline{\mathbf{u}} - \mathbf{u}_\varepsilon)
  dxdt  - \int_0^T \int_\Omega \varrho_\varepsilon  \nabla_x \mathbf{u}^E \cdot \mathbf{u}^E \cdot
  (\overline{\mathbf{u}}-\mathbf{u}_\varepsilon) de xdt
\\
  & \quad +
  \varepsilon \int_0^T \int_\Omega \varrho_\varepsilon \nabla_x \overline{\mathbf{u}}
  \cdot \mathbf{u}_\varepsilon   \cdot  (\nabla_x \log \varrho^E - \nabla_x \log \varrho_\varepsilon)dxdt
  \\
  &\quad + \varepsilon \int_0^T \int_\Omega  \varrho_\varepsilon \nabla_x \nabla_x \log \varrho^E \cdot
  \mathbf{u}_\varepsilon \cdot (\overline{\mathbf{u}}  -  \mathbf{u}_\varepsilon)dxdt
    \\
  & \quad + 2\varepsilon^2 \int_0^T \int_\Omega  \varrho_\varepsilon \nabla_x \nabla_x \log \varrho^E \cdot
  \mathbf{u}_\varepsilon \cdot (\nabla_x \log \varrho^E - \nabla_x   \log   \varrho_\varepsilon)dxdt
  \\
   &\quad  +\varepsilon\int_0^T \int_{\Omega} \varrho_\varepsilon \mathbb{D}(\mathbf{v}_\varepsilon) : \nabla_{x} \overline{\mathbf{v}}dxdt
   +\varepsilon\int_0^T \int_{\Omega} \varrho_\varepsilon \mathbb{A}(\mathbf{v}_\varepsilon) : \nabla_{x} \overline{\mathbf{v}} dxdt
   \\
   &\quad - \varepsilon\int_0^T \int_{\Omega}
   \varrho_\varepsilon \nabla_x \mathbf{w}_\varepsilon: \nabla_{x} \overline{\mathbf{v}} dxdt
     +\varepsilon\int_0^T \int_{\Omega} \varrho_\varepsilon \nabla_x^{\top} \mathbf{u}_\varepsilon : \nabla_{x} \overline{\mathbf{w}} dxdt
\\
&\quad    +r_1\int_0^T \int_{\Omega}  \varrho_\varepsilon |\mathbf{v}_\varepsilon - \mathbf{w}_\varepsilon|(\mathbf{v}_\varepsilon
- \mathbf{w}_\varepsilon)  \cdot \overline{\mathbf{v}}   dxdt
\end{aligned}
\end{equation*}
\begin{equation} \label{step-7}
  \begin{aligned}
 &\quad   -\int_0^T \int_\Omega \left[ -p(\varrho^E)\textrm{div}_x \mathbf{u}^E
    + p(\varrho_\varepsilon) \textrm{div}_x \overline{\mathbf{v}}
    - p'(\varrho^E)(\varrho_\varepsilon-\varrho^E)\text{div}_x \mathbf{u}^E  \right]  dxdt
\\
  &\quad  -2\varepsilon \int_0^T \int_\Omega  \left( \sqrt{\varrho_\varepsilon}
    \mathcal{S}(\mathbf{u}_\varepsilon) \mathbb{D}(\overline{\mathbf{u}})
    +  \sqrt{\varrho_\varepsilon} \mathcal{A}(\mathbf{u}_\varepsilon)
    \mathbb{A}(\overline{\mathbf{u}}) \right)  dxdt
\\
&\quad  +\varepsilon \int_0^T \int_\Omega \left(
    \varrho_\varepsilon \left( |\mathbb{D}(\overline{\mathbf{u}})|^2
      + |\mathbb{A}(\overline{\mathbf{u}})|^2 \right) -
    \frac{p'(\varrho_\varepsilon)}{\varrho_\varepsilon} |\nabla_x
    \varrho_\varepsilon|^2 \right) dxdt
\end{aligned}
\end{equation}
Now, we observe that
\begin{equation*}
  \begin{aligned}
-\varepsilon \int_0^T &\int_\Omega \frac{p'(\varrho_\varepsilon)}{\varrho_\varepsilon} |\nabla_x
\varrho_\varepsilon|^2dxdt
\\
& = -\varepsilon \int_0^T \int_\Omega \varrho_\varepsilon (p'(\varrho_\varepsilon)\nabla_x \log
\varrho_\varepsilon-p'(\varrho^E)\nabla_x \log \varrho^E)(\nabla_x \log \varrho_\varepsilon - \nabla_x \log \varrho^E) dxdt
\\
&\quad + \varepsilon \int_0^T \int_\Omega \frac{\varrho_\varepsilon}{\varrho^E}p'(\varrho^E)\nabla_x \varrho^E \left(
\frac{\nabla_x \varrho^E}{\varrho^E} - \frac{\nabla_x \varrho_\varepsilon}{\varrho_\varepsilon} \right) dxdt
\\
&\quad - \varepsilon \int_0^T \int_\Omega p'(\varrho_\varepsilon)\nabla_x \varrho_\varepsilon \frac{\nabla_x \varrho^E}{\varrho^E} dxdt.
\end{aligned}
\end{equation*}
Consequently, rewriting \eqref{step-7}, we derive the relative energy in its final version as follows
\begin{equation} \label{step-8}
  \begin{aligned}
    \mathcal{E}(T,\cdot)  - E(0,\cdot)
   +\varepsilon \int_0^T \int_\Omega
  \varrho_\varepsilon  \big(p'(\varrho_\varepsilon)&\nabla_x \log \varrho_\varepsilon-p'(\varrho^E)
  \nabla_x \log \varrho^E\big)
   \\
   & \quad \times(\nabla_x \log \varrho_\varepsilon - \nabla_x \log \varrho^E)
  dxdt  \leq \sum_{i=1}^{11}R_i,
\end{aligned}
\end{equation}
where
\begin{equation*}
\begin{aligned}
& R_1 = \int_0^T \int_\Omega \varrho_\varepsilon
  \partial_t \mathbf{v}_{bl} \cdot (\overline{\mathbf{u}} - \mathbf{u}_\varepsilon)dxdt,
\\
& R_2 = \varepsilon \int_0^T \int_\Omega \varrho_\varepsilon \partial_t \overline{\mathbf{u}} 
\cdot (\nabla_x \log \varrho^E - \nabla_x \log \varrho_\varepsilon)dxdt,
\\
& R_3 = \varepsilon \int_0^T \int_\Omega \varrho_\varepsilon
\partial_t \nabla_x \log \varrho^E \cdot (\overline{\mathbf{u}} - \mathbf{u}_\varepsilon)dxdt,
\\
& R_4 = 2\varepsilon^2 \int_0^T \int_\Omega \varrho_\varepsilon \partial_t \nabla_x \log \varrho^E
\cdot (\nabla_x \log \varrho^E - \nabla_x \log \varrho_\varepsilon)dxdt,
\\
& R_5 = \int_0^T \int_\Omega \varrho_\varepsilon \nabla_x \overline{\mathbf{u}}
  \cdot \mathbf{u}_\varepsilon \cdot  (\overline{\mathbf{u}} - \mathbf{u}_\varepsilon)  dxdt
  - \int_0^T \int_\Omega \varrho_\varepsilon \nabla_x \mathbf{u}^E \cdot \mathbf{u}^E \cdot
  (\overline{\mathbf{u}}-\mathbf{u}_\varepsilon) dxdt,
\\
& R_6 = \varepsilon \int_0^T \int_\Omega \varrho_\varepsilon \nabla_x \overline{\mathbf{u}}
\cdot \mathbf{u}_\varepsilon \cdot (\nabla_x \log \varrho^E - \nabla_x \log \varrho_\varepsilon)dxdt,
\\
& R_7 = \varepsilon \int_0^T \int_\Omega \varrho_\varepsilon \nabla_x \nabla_x \log \varrho^E \cdot
\mathbf{u}_\varepsilon \cdot (\overline{\mathbf{u}} - \mathbf{u}_\varepsilon)dxdt,
\\
&R_8 = 2\varepsilon^2 \int_0^T \int_\Omega \varrho_\varepsilon \nabla_x \nabla_x \log \varrho^E \cdot
\mathbf{u}_\varepsilon \cdot (\nabla_x \log \varrho^E - \nabla_x \log \varrho_\varepsilon)dxdt,
\end{aligned}
\end{equation*}
\begin{align*}
  & \begin{aligned}
    R_9 & = \varepsilon\int_0^T \int_{\Omega} \varrho_\varepsilon \mathbb{D}(\mathbf{v}_\varepsilon) : \nabla_{x} \overline{\mathbf{v}}dxdt
+\varepsilon\int_0^T \int_{\Omega} \varrho_\varepsilon \mathbb{A}(\mathbf{v}_\varepsilon) : \nabla_{x} \overline{\mathbf{v}} dxdt
- \varepsilon\int_0^T \int_{\Omega} \varrho_\varepsilon \nabla_x \mathbf{w}_\varepsilon: \nabla_{x} \overline{\mathbf{v}} dxdt,
\\
& \quad +\varepsilon\int_0^T \int_{\Omega} \varrho_\varepsilon \nabla_x^{\top} \mathbf{u}_\varepsilon : \nabla_{x} \overline{\mathbf{w}} dxdt
-2\varepsilon \int_0^T \int_\Omega \left( \sqrt{\varrho_\varepsilon} \mathcal{S}(\mathbf{u}_\varepsilon) \mathbb{D}(\overline{\mathbf{u}})
+ \sqrt{\varrho_\varepsilon} \mathcal{A}(\mathbf{u}_\varepsilon) \mathbb{A}(\overline{\mathbf{u}}) \right) dxdt
\\
& \quad +\varepsilon \int_0^T \int_\Omega \left( \varrho_\varepsilon \left( |\mathbb{D}(\overline{\mathbf{u}})|^2
    + |\mathbb{A}(\overline{\mathbf{u}})|^2 \right) \right) dxdt,
\end{aligned}
  \\
&\begin{aligned}
R_{10} & = -\int_0^T \int_\Omega  \left[ -p(\varrho^E)\textrm{div}_x \mathbf{u}^E
+ p(\varrho_\varepsilon) \textrm{div}_x \overline{\mathbf{v}}
- p'(\varrho^E)(\varrho_\varepsilon-\varrho^E)\text{div}_x \mathbf{u}^E \right] dxdt
\\
&\quad + \varepsilon \int_0^T \int_\Omega
\frac{\varrho_\varepsilon}{\varrho^E}p'(\varrho^E)\nabla_x \varrho^E \left(
\frac{\nabla_x \varrho^E}{\varrho^E} - \frac{\nabla_x \varrho_\varepsilon}{\varrho_\varepsilon}
\right)
dxdt
\\
&\quad - \varepsilon \int_0^T \int_\Omega
p'(\varrho_\varepsilon)\nabla_x \varrho_\varepsilon \frac{\nabla_x \varrho^E}{\varrho^E}
dxdt,
\end{aligned}
\\
& R_{11} = r_1\int_0^T \int_{\Omega} \varrho_\varepsilon |\mathbf{v}_\varepsilon
- \mathbf{w}_\varepsilon|(\mathbf{v}_\varepsilon - \mathbf{w}_\varepsilon)
\cdot \overline{\mathbf{v}} dxdt.
\end{align*}
%
%
\subsection{Kato type ``fake" boundary layer}
\label{subsec:fake-boundary-layer}
In the spirit of \cite{Ka} (see also \cite{BaNg}, \cite{Su} and
\cite{WZ}), we introduce a Kato type "fake" boundary layer. We
consider $(\varrho^E, \mathbf{u}^E)$ the strong solution of
\eqref{cont-E}--\eqref{bc-E}, and we define
\begin{equation} \label{fl}
  \mathbf{v}_{bl} \triangleq \xi \Big(
  \frac{d_{\Omega} (x)}{\delta} \Big) \mathbf{u}^E,
\end{equation}
with $\delta=\delta(\varepsilon)\to 0$ as $\varepsilon \to 0$ with the
rate $\delta(\varepsilon)$ given below, and $\xi : [0,+\infty)
\rightarrow [0,+\infty)$  is a smooth cut-off function, such that
\begin{equation*}
  \xi(0)=1, \, \,  \textrm{ supp } \xi\subseteq [0,1),
  \,\,\, \|\xi\|_{L^\infty}<\infty, \,\,\, \|\xi'\|_{L^\infty}<\infty.
\end{equation*}
It follows that $\mathbf{v}_{bl} = \mathbf{u}^E$ on the boundary
$[0,T] \times \partial \Omega$ and it has support in the domain
$[0,T] \times \Gamma_\varepsilon$ and recall that
$\Gamma_{c\varepsilon} = \{ x\in \Omega \ : \ d_\Omega(x) \leq c
\varepsilon \}$.  Moreover, the quantity $\mathbf{v}_{bl}$ satisfies
the following properties (see \cite{BaNg}, \cite{Ka}, \cite{Su})
with $\delta =c\varepsilon$, $c>0$: for
$1\leq p < + \infty$, and for $\varepsilon\to 0$, we have
\begin{equation} \label{v-prop}
  \begin{aligned}
    &\|\mathbf{v}_{bl}\|_{C([0,T]\times \Omega)} \leq C,
    \,\,\,
    \|\bv_{bl}\|_{C([0,T]\times L^p(\Omega))} \leq C
    \varepsilon^{\frac{1}{p}},
    \,\,\,
    \|\D_t \bv_{bl}\|_{C([0,T]\times L^p(\Omega))} \leq C
    \varepsilon^{\frac{1}{p}},
    \\[0.2 cm]
    &\|\D_t \bv_{bl}\|_{L^\infty([0,T]\times \Omega)} \leq C,
    \,\,\, \|\nabla_x \bv_{bl}\|_{L^\infty([0,T]\times
      \Omega)} = O( \varepsilon^{-1}),
    \\[0.2 cm]
    &\|\text{div}_x \bv_{bl}\|_{L^\infty([0,T]\times \Omega)} \leq C,
    \,\,\, 
    \|\text{div}_x \bv_{bl}\|_{C([0,T]\times L^p(\Omega))} \leq C
    \varepsilon^{\frac{1}{p}},
    \\[0.2 cm]
    &\|d_\Omega(x) \nabla_x \bv_{bl}\|_{L^\infty([0,T]\times
    \Omega)} \leq C,
    \,\,\, 
    \|d^2_\Omega(x) \nabla_x
    \bv_{bl}\|_{L^\infty([0,T]\times \Omega)} \leq C\varepsilon,
        \\[0.2 cm]
        &\|d_\Omega(x)\nabla_x \bv_{bl}\|_{L^\infty([0,T]; L^2(\Omega))}= O(\varepsilon^{\frac{1}{2}}),
          \,\,\,
       \|d_\Omega(x)\nabla_x \bv_{bl}\|_{L^\infty([0,T]; L^p(\Omega))}= O(\varepsilon^{\frac{1}{p}}). 
  \end{aligned}
\end{equation}

\begin{remark}\label{Su}
To prove the  estimates
\eqref{v-prop}, the following quantities have been introduced
\begin{equation} \label{xiz}
  \begin{aligned}
    & z(x)= \xi \left(\frac{d_\Omega(x)}{\delta}\right), \ \ \ \
    \widetilde{\xi}(r) = r \xi'(r), \ \ \ \ \widetilde{z}(x) =
    \widetilde{\xi} \left(\frac{d_\Omega(x)}{\delta}\right),
    \\[0.2 cm]
    & \widehat{\xi}(r) = r^2 \xi'(r), \ \ \ \ \widehat{z}(x) =
    \widehat{\xi} \left(\frac{d_\Omega(x)}{\delta}\right), \ \ \ \
    \mathbf{v}_{bl} = z \mathbf{u}^E
  \end{aligned}
\end{equation}
and recall that $\delta =\delta(\varepsilon)=c\varepsilon, c>0$. 
Relations \eqref{xiz} will be used in the subsequent analysis of the inviscid
limit in order to handle some of the viscous terms
(see Subsection~\ref{subsec:viscous} below).
Moreover, with the above notation, we have (see \cite[p. 170]{Su})
\begin{equation} \label{norm-vbl}
  \begin{aligned}
    \mathbf{n}\cdot \nabla_x\mathbf{v}_{bl} &=
     z \mathbf{n}\cdot \nabla_x\mathbf{u}^E + \frac{1}{c\varepsilon}\xi'\big(\frac{d_\Omega(x)}{c\varepsilon}\big)\mathbf{u}^E\\
 &=       z \mathbf{n}\cdot \nabla_x\mathbf{u}^E + \frac{1}{d_\Omega}\widetilde{z}\, \mathbf{n}\cdot \mathbf{n}\otimes\mathbf{u}^E
  \end{aligned}
\end{equation}
and
\begin{equation} \label{div-vbl}
  \begin{aligned}
    \div \mathbf{v}_{bl} &= z \div \mathbf{u}^E + \mathbf{u}^E \cdot \nabla_xz\\
     &= z \div \mathbf{u}^E + \frac{\mathbf{u}^E \cdot \mathbf{n}}{d_\Omega}\widetilde{z}.\\
  \end{aligned}
\end{equation}
\end{remark}

\subsection{Inviscid limit} \label{sec:inviscid-lim}
We now consider
$\overline{\mathbf{u}} = \mathbf{u}^E - \mathbf{v}_{bl}$ in
\eqref{step-8}. In what follows we provide suitable bounds, and we
pass to the limit as $\varepsilon \to 0$, for the terms $R_i$,
$i=1,\ldots, 11$ involved in \eqref{step-8}.
For simplicity, and without loss of generality, we set $\widetilde{\delta}(\varepsilon)=\varepsilon$.
\subsubsection{Time dependent term}
For $R_1$, we have
\begingroup
\allowdisplaybreaks
\begin{align*}
  \int_0^T \int_\Omega \varrho_\varepsilon &
  \partial_t \mathbf{v}_{bl} \cdot (\overline{\mathbf{u}} - \mathbf{u}_\varepsilon)dxdt
  \\
  &=
    \int_0^T \int_\Omega \varrho_\varepsilon
    \partial_t \mathbf{v}_{bl} \cdot [\overline{\mathbf{v}} - \mathbf{v}_\varepsilon
    -(\overline{\mathbf{w}} - \mathbf{w}_\varepsilon)]dxdt
  \\
  & = \int_0^T \int_\Omega \varrho_\varepsilon
    \partial_t \mathbf{v}_{bl} \cdot (\overline{\mathbf{v}} - \mathbf{v}_\varepsilon)dxdt
    -\int_0^T \int_\Omega \varrho_\varepsilon
    \partial_t \mathbf{v}_{bl} \cdot (\overline{\mathbf{w}} - \mathbf{w}_\varepsilon)dxdt
  \\
  & \leq \int_0^T \left( \int_\Omega \varrho_\varepsilon  |\partial_t \mathbf{v}_{bl}|^2dx\right)^{1/2}
    \left(  \int_\Omega  \varrho_\varepsilon |\overline{\mathbf{v}}-\mathbf{v}_\varepsilon|^2  dx \right)^{1/2}dt
  \\
& \quad     +  \int_0^T    \left(  \int_\Omega  \varrho_\varepsilon |\partial_t \mathbf{v}_{bl}|^2  dx  \right)^{1/2}
    \left(  \int_\Omega \varrho_\varepsilon |\overline{\mathbf{w}}-\mathbf{w}_\varepsilon|^2   dx  \right)^{1/2}dt
\\
&    \leq    C \varepsilon^{\frac{1}{p}}   + C\int_0^T  \mathcal{E}(t,\cdot)dt.
\end{align*}
\endgroup

\subsubsection{Convective terms}
For $R_2$, we have
\begin{align*}
  \varepsilon \int_0^T \int_\Omega \varrho_\varepsilon &\partial_t \overline{\mathbf{u}} 
  \cdot (\nabla_x  \log \varrho^E  - \nabla_x \log \varrho_\varepsilon)dxdt
  \\
  & = \varepsilon \int_0^T \int_\Omega \varrho_\varepsilon \partial_t (\mathbf{u}^E - \mathbf{v}_{bl}) 
    \cdot  (\nabla_x \log \varrho^E - \nabla_x \log \varrho_\varepsilon)  dxdt
  \\
  & = \varepsilon \int_0^T \int_\Omega \left(\varrho_\varepsilon \partial_t \mathbf{u}^E 
    \cdot \nabla_x \log \varrho^E  - \varrho_\varepsilon \partial_t \mathbf{u}^E
    \cdot \nabla_x \log \varrho_\varepsilon \right)dxdt
  \\
  &\quad  - \varepsilon \int_0^T \int_\Omega \left(\varrho_\varepsilon \partial_t \mathbf{v}_{bl} 
    \cdot \nabla_x \log \varrho^E - \varrho_\varepsilon \partial_t \mathbf{v}_{bl} 
    \cdot \nabla_x \log \varrho_\varepsilon \right)dxdt
  \\
  & = \bigg[ \varepsilon \int_0^T \int_\Omega  \left( \varrho_\varepsilon \partial_t \mathbf{u}^E 
    \cdot \nabla_x \log \varrho^E - \partial_t \mathbf{u}^E  \cdot \nabla_x \sqrt{\varrho_\varepsilon}
    \sqrt{\varrho_\varepsilon} \right) dxdt
  \\
  & \quad  - \varepsilon \int_0^T \int_\Omega  \left( \varrho_\varepsilon \partial_t \mathbf{v}_{bl} 
    \cdot \nabla_x \log \varrho^E  - \partial_t \mathbf{v}_{bl} \cdot \nabla_x \sqrt{\varrho_\varepsilon}
    \sqrt{\varrho_\varepsilon}\right)dxdt \bigg] \longrightarrow 0  \, \, \textrm{ as } \, \, \varepsilon \to 0.
\end{align*}

For the term $R_3$, we have
\begin{align*}
  \varepsilon \int_0^T \int_\Omega \varrho_\varepsilon 
  \partial_t \nabla_x & \log \varrho^E  \cdot (\overline{\mathbf{u}}  - \mathbf{u}_\varepsilon)dxdt
  \\
& =
  \varepsilon \int_0^T \int_\Omega \varrho_\varepsilon
  \partial_t \nabla_x \log \varrho^E \cdot (\mathbf{u}^E - \mathbf{v}_{bl} - \mathbf{u}_\varepsilon)
  dxdt
\\
 &  = \varepsilon \int_0^T \int_\Omega \varrho_\varepsilon \mathbf{u}^E \partial_t \nabla_x \log \varrho^E dxdt
 - \varepsilon \int_0^T \int_\Omega \varrho_\varepsilon \mathbf{v}_{bl} \partial_t \nabla_x \log \varrho^E dxdt
  \\
 &\quad - \varepsilon \int_0^T \int_\Omega \varrho_\varepsilon \mathbf{u}_\varepsilon
 \partial_t \nabla_x \log \varrho^E dxdt
\\
 & = \bigg[ \varepsilon \int_0^T \int_\Omega \varrho_\varepsilon \mathbf{u}^E  \partial_t \nabla_x \log \varrho^E
    dxdt  - \varepsilon \int_0^T \int_\Omega \varrho_\varepsilon  \mathbf{v}_{bl} \partial_t \nabla_x \log \varrho^E
    dxdt
 \\
&  \quad  - \varepsilon \int_0^T \int_\Omega \sqrt{\varrho_\varepsilon} \sqrt{\varrho_\varepsilon}  \mathbf{u}_\varepsilon
    \partial_t \nabla_x \log \varrho^E   dxdt  \bigg]  \longrightarrow 0   \, \, \textrm{ as } \, \, \varepsilon \to 0.
\end{align*}

Again, for $R_4$, we have
\begin{align*}
  2\varepsilon^2 \int_0^T \int_\Omega \varrho_\varepsilon \partial_t \nabla_x \log \varrho^E&
\cdot (\nabla_x \log \varrho^E - \nabla_x \log \varrho_\varepsilon)dxdt
\\
 & =
  \bigg[ 2\varepsilon^2 \int_0^T \int_\Omega \varrho_\varepsilon \partial_t \nabla_x \log \varrho^E
    \cdot  \nabla_x \log \varrho^E  dxdt
  \\
  &  \quad - 2\varepsilon^2 \int_0^T \int_\Omega \partial_t \nabla_x \log \varrho^E
    \cdot \nabla_x  \sqrt{\varrho_\varepsilon} \sqrt{\varrho_\varepsilon} dxdt \bigg]
  \longrightarrow 0 \, \, \textrm{ as } \, \,\varepsilon \to 0.
\end{align*}

Let us now consider the term $R_5$, to get
\begin{equation} \label{r4}
\begin{aligned}
  \int_0^T \int_\Omega  &\varrho_\varepsilon \nabla_x \overline{\mathbf{u}}
  \cdot \mathbf{u}_\varepsilon \cdot (\overline{\mathbf{u}} - \mathbf{u}_\varepsilon)
  dxdt - \int_0^T \int_\Omega \varrho_\varepsilon \nabla_x \mathbf{u}^E \cdot \mathbf{u}^E \cdot
  (\overline{\mathbf{u}}-\mathbf{u}_\varepsilon) dxdt
\\
&    = \int_0^T \int_\Omega
     \varrho_\varepsilon (\overline{\mathbf{u}}-\mathbf{u}_\varepsilon)(\nabla_x \overline{\mathbf{u}}
     \ \mathbf{u}_\varepsilon - \nabla_x \mathbf{u}^E \mathbf{u}^E)     dxdt
\\
&    = \int_0^T \int_\Omega
     \varrho_\varepsilon (\overline{\mathbf{u}}-\mathbf{u}_\varepsilon)(\nabla_x \overline{\mathbf{u}}
     \ \mathbf{u}_\varepsilon + \nabla_x \overline{\mathbf{u}} \ \mathbf{u}^E -\nabla_x \overline{\mathbf{u}}
     \  \mathbf{u}^E - \nabla_x \mathbf{u}^E \mathbf{u}^E)  dxdt
 \\
&    = \int_0^T \int_\Omega \varrho_\varepsilon (\overline{\mathbf{u}}-\mathbf{u}_\varepsilon)
    \left[ \nabla_x \overline{\mathbf{u}} (\mathbf{u}_\varepsilon - \mathbf{u}^E)
    -\mathbf{u}^E \nabla_x \mathbf{v}_{bl} \right]  dxdt
\\
&   = \int_0^T \int_\Omega \varrho_\varepsilon (\overline{\mathbf{u}}-\mathbf{u}_\varepsilon)
\left[  \nabla_x \overline{\mathbf{u}}  (\mathbf{u}_\varepsilon - \overline{\mathbf{u}} - \mathbf{v}_{bl})
  -\mathbf{u}^E \nabla_x \mathbf{v}_{bl}  \right]  dxdt
\\
& = \int_0^T \int_\Omega \varrho_\varepsilon \nabla_x \overline{\mathbf{u}} \ (\overline{\mathbf{u}}
-\mathbf{u}_\varepsilon) (\mathbf{u}_\varepsilon-\overline{\mathbf{u}})dxdt
    - \int_0^T \int_\Omega \varrho_\varepsilon  (\overline{\mathbf{u}}-\mathbf{u}_\varepsilon)
    \nabla_x \overline{\mathbf{u}} \ \mathbf{v}_{bl}  dxdt
    \\
&\qquad    - \int_0^T \int_\Omega \varrho_\varepsilon  (\overline{\mathbf{u}}-\mathbf{u}_\varepsilon)
    \mathbf{u}^E \nabla_x \mathbf{v}_{bl}dxdt
\\
&    = \widetilde{R_5}  - \int_0^T \int_\Omega  \varrho_\varepsilon (\overline{\mathbf{u}}-\mathbf{u}_\varepsilon)
    \nabla_x \overline{\mathbf{u}} \ \mathbf{v}_{bl}  dxdt - \int_0^T \int_\Omega \varrho_\varepsilon
    (\overline{\mathbf{u}}-\mathbf{u}_\varepsilon) \mathbf{u}^E \nabla_x \mathbf{v}_{bl}dxdt.
\end{aligned}
\end{equation}
%
%

%
%
For $\widetilde{R_5}$, we have
\begin{align*}
  \widetilde{R_5} & = \int_0^T \int_\Omega \varrho_\varepsilon \nabla_x \overline{\mathbf{u}}
  \cdot  (\mathbf{u}_\varepsilon - \overline{\mathbf{u}}) \cdot (\overline{\mathbf{u}}
  - \mathbf{u}_\varepsilon) dxdt
\\
 & = \int_0^T \int_\Omega \varrho_\varepsilon (\nabla_x \mathbf{u}^E - \nabla_x \mathbf{v}_{bl})
 (\mathbf{u}_\varepsilon - \mathbf{u}^E + \mathbf{v}_{bl})\cdot (\mathbf{u}^E - \mathbf{v}_{bl}
 - \mathbf{u}_\varepsilon) dxdt
\\
&  = \int_0^T \int_\Omega \varrho_\varepsilon \nabla_x \mathbf{u}^E \mathbf{v}_{bl}\cdot
  (\mathbf{u}^E - \mathbf{v}_{bl} - \mathbf{u}_\varepsilon) dxdt
+ \int_0^T \int_\Omega \varrho_\varepsilon \nabla_x \mathbf{u}^E (\mathbf{u}_\varepsilon - \mathbf{u}^E)
     \cdot (\mathbf{u}^E - \mathbf{v}_{bl}  - \mathbf{u}_\varepsilon)  dxdt
     \\
&\quad  - \int_0^T \int_\Omega \varrho_\varepsilon \mathbf{u}_\varepsilon \nabla_x \mathbf{v}_{bl}\cdot
  (\mathbf{u}^E - \mathbf{v}_{bl} - \mathbf{u}_\varepsilon) dxdt
- \int_0^T \int_\Omega \varrho_\varepsilon \nabla_x \mathbf{v}_{bl}   (\mathbf{v}_{bl} - \mathbf{u}^E)\cdot
  (\mathbf{u}^E - \mathbf{v}_{bl} - \mathbf{u}_\varepsilon)  dxdt
\\
 & = \widetilde{R_{5,1}} + \widetilde{R_{5,2}} + \widetilde{R_{5,3}} + \widetilde{R_{5,4}}. 
\end{align*}

For $\widetilde{R_{5,1}}$, integrating by parts, we have 
\begin{align*} 
  \int_0^T  \int_\Omega& \varrho_\varepsilon \nabla_x \mathbf{u}^E \mathbf{v}_{bl}
  (\overline{\mathbf{u}} - \mathbf{u}_\varepsilon) dxdt
  \\
&  =  \int_0^T \int_\Omega \varrho_\varepsilon \nabla_x \mathbf{u}^E \mathbf{v}_{bl}
    \overline{\mathbf{u}}  dxdt
  - \int_0^T \int_\Omega \varrho_\varepsilon \nabla_x \mathbf{u}^E \mathbf{v}_{bl}
  \mathbf{u}_\varepsilon  dxdt
\\
&  = - \int_0^T \int_\Omega \nabla_x \varrho_\varepsilon  \mathbf{u}^E (z \mathbf{u}^E)
    \overline{\mathbf{u}} dxdt
     - \int_0^T \int_\Omega \varrho_\varepsilon \mathbf{u}^E \text{div}_x \mathbf{v}_{bl} \overline{\mathbf{u}} dxdt
\\
&\quad    - \int_0^T \int_\Omega  \varrho_\varepsilon \mathbf{u}^E (z \mathbf{u}^E) \text{div}_x \overline{\mathbf{u}}  dxdt
- \int_0^T \int_\Omega  \varrho_\varepsilon \nabla_x \mathbf{u}^E (z \mathbf{u}^E) \mathbf{u}_\varepsilon  dxdt,
\end{align*}
and hence
  \begin{equation} \label{passaggio-al-limite-R41}
 \begin{aligned} 
      \int_0^T &\int_\Omega  \varrho_\varepsilon \nabla_x \mathbf{u}^E \mathbf{v}_{bl}
      (\overline{\mathbf{u}} - \mathbf{u}_\varepsilon) dxdt
      \\
      &  =  - 2\int_0^T \int_\Omega \sqrt{\varrho_\varepsilon}\nabla_x \sqrt{\varrho_\varepsilon} 
      \mathbf{u}^E \mathbf{v}_{bl} \overline{\mathbf{u}} dxdt
      - \int_0^T \int_\Omega \sqrt{\varrho_\varepsilon} \sqrt{\varrho_\varepsilon}\mathbf{u}^E \text{div}_x
      \mathbf{v}_{bl} \,\overline{\mathbf{u}} dxdt
      \\
      &\quad  - \int_0^T \int_\Omega \sqrt{\varrho_\varepsilon} \sqrt{\varrho_\varepsilon} \mathbf{u}^E \mathbf{v}_{bl}
      \text{div}_x \overline{\mathbf{u}} dxdt
      - \int_0^T \int_\Omega  \sqrt{\varrho_\varepsilon} \sqrt{\varrho_\varepsilon}\nabla_x \mathbf{u}^E \mathbf{v}_{bl}
      \mathbf{u}_\varepsilon  dxdt
      \\
      &\leq \left[ C \|\mathbf{v}_{bl}\|_{L^\infty(0,T;L^{\frac{2\gamma}{\gamma-1}}(\Omega))} 
        + C \|\text{div}_x \mathbf{v}_{bl}\|_{L^\infty(0,T;L^{\frac{2\gamma}{\gamma-1}}(\Omega))}\right]
\longrightarrow 0 \, \, \textrm{ as } \, \, \varepsilon \to 0.
\end{aligned}
\end{equation}
Indeed, the first term on the right-hand side can be bounded as follows
\begin{align*}
  \left| 2\int_0^T \int_\Omega \sqrt{\varrho_\varepsilon}\nabla_x \sqrt{\varrho_\varepsilon} 
    \mathbf{u}^E \mathbf{v}_{bl} \overline{\mathbf{u}} dxdt \right|\leq 2\int_0^T \Big(\|\sqrt{\varrho_\varepsilon}\|_{L^{2\gamma}(\Omega)}
  &\|\nabla_x\sqrt{\varrho_\varepsilon}\, \mathbf{u}^E \|_{L^2(\Omega)}
  \\
  &\times \| \mathbf{v}_{bl}\|_{L^{\frac{2\gamma}{\gamma-1}(\Omega))}}
  \|\overline{\mathbf{u}}\|_{L^\infty(\Omega)} \Big)dt,
\end{align*}
and we pass to the limit, thanks to \eqref{v-prop}, recalling that $
\|\mathbf{v}_{bl}\|_{L^\infty(0,T;L^{\frac{2\gamma}{\gamma-1}}(\Omega))}
= O(\varepsilon^{\frac{\gamma-1}{2\gamma}})$, as $\varepsilon\to
0$. The second term can be bounded still using H\"older's inequality
along with \eqref{v-prop}, i.e.
  \begin{align*}
  \left| \int_0^T \int_\Omega \sqrt{\varrho_\varepsilon} \sqrt{\varrho_\varepsilon}\mathbf{u}^E \text{div}_x
    \mathbf{v}_{bl} \,\overline{\mathbf{u}} dxdt\right| \leq
  \int_0^T\Big(\|\sqrt{\varrho_\varepsilon}\|_{L^{2\gamma}(\Omega)} & \|\sqrt{\varrho_\varepsilon}\mathbf{u}^E \|_{L^2(\Omega)}
   \\
  &\times\|\text{div}_x \mathbf{v}_{bl}\|_{L^{\frac{2\gamma}{\gamma-1}(\Omega))}}
  \|\overline{\mathbf{u}}\|_{L^\infty(\Omega)} \Big)dt,
  \end{align*}
  and to pass to the limit we exploit the fact that $\|\text{div}_x \mathbf{v}_{bl}\|_{L^\infty(0,T;L^{\frac{2\gamma}{\gamma-1}}(\Omega))}
= O(\varepsilon^{\frac{\gamma-1}{2\gamma}})$, as $\varepsilon\to
0$. The other terms can be treated in a similar way. This fully justifies the passage to the limit performed 
in the last line of \eqref{passaggio-al-limite-R41}. 
%
%

%
%
Let us now consider  $\widetilde{R_{5,2}}$. We have
\begin{align*}
  \int_0^T \int_\Omega \varrho_\varepsilon \nabla_x \mathbf{u}^E
  (\mathbf{u}_\varepsilon - \mathbf{u}^E) (\mathbf{u}^E - \mathbf{v}_{bl}
  - \mathbf{u}_\varepsilon) dxdt & = - \int_0^T \int_\Omega \varrho_\varepsilon \nabla_x
 \mathbf{u}^E (\mathbf{u}_\varepsilon - \mathbf{u}^E)^2 dxdt
  \\
  &\qquad  - \int_0^T \int_\Omega  \varrho_\varepsilon \nabla_x \mathbf{u}^E (\mathbf{u}_\varepsilon
    - \mathbf{u}^E) \mathbf{v}_{bl} dxdt,
\end{align*}
where the second term can be treated similarly to $\widetilde{R_{5,1}}$ and it goes to zero as $\varepsilon \to 0$. 

For the first term on the right-hand side, we get
\begin{equation*}
  \int_0^T \int_\Omega \varrho_\varepsilon \nabla_x \mathbf{u}^E (\mathbf{u}_\varepsilon - \mathbf{u}^E)^2
  dxdt = \int_0^T \int_\Omega  \left( \varrho_\varepsilon (\mathbf{u}_\varepsilon - \overline{\mathbf{u}})^2
    + \varrho_\varepsilon |\mathbf{v}_{bl}|^2-2\varrho_\varepsilon(\mathbf{u}_\varepsilon
    - \overline{\mathbf{u}}) \mathbf{v}_{bl}\right) dxdt,
\end{equation*}
and the last two terms go to zero as $\varepsilon \to 0$.

Moreover, we have
\begin{equation} \label{R422-again}
  \begin{aligned}
    \int_0^T \int_\Omega \varrho_\varepsilon (\mathbf{u}_\varepsilon - \overline{\mathbf{u}})^2
    dxdt & = \int_0^T \int_\Omega \varrho_\varepsilon (\mathbf{v}_\varepsilon - \overline{\mathbf{v}})^2
    dx dt + \int_0^T \int_\Omega \varrho_\varepsilon (\mathbf{w}_\varepsilon - \overline{\mathbf{w}})^2dxdt
    \\
    &\qquad  + 2 \int_0^T \int_\Omega  \varrho_\varepsilon (\mathbf{v}_\varepsilon - \overline{\mathbf{v}})
    (\overline{\mathbf{w}} - \mathbf{w}_\varepsilon) dxdt,
  \end{aligned}
\end{equation}
and for the first two addends in the right-hand side, we observe that
\begin{equation*}
  \int_0^T \int_\Omega \varrho_\varepsilon (\mathbf{v}_\varepsilon - \overline{\mathbf{v}})^2dxdt
  + \int_0^T \int_\Omega \varrho_\varepsilon (\mathbf{w}_\varepsilon - \overline{\mathbf{w}})^2dxdt
  \leq C \int_0^T \mathcal{E}(t,\cdot) dt,
\end{equation*}
while, for the last term on the right-hand side of \eqref{R422-again}, we have
\begingroup
\allowdisplaybreaks
\begin{align*}
  2\int_0^T& \int_\Omega  \varrho_\varepsilon  (\mathbf{v}_\varepsilon - \overline{\mathbf{v}})
  (\overline{\mathbf{w}} - \mathbf{w}_\varepsilon) dxdt
  = 2 \int_0^T \int_\Omega \big( \varrho_\varepsilon \mathbf{v}_\varepsilon \overline{\mathbf{w}}
  - \varrho_\varepsilon \mathbf{v}_\varepsilon \mathbf{w}_\varepsilon - \varrho_\varepsilon \overline{\mathbf{v}} \, 
  \overline{\mathbf{w}} + \varrho_\varepsilon \overline{\mathbf{v}} \, \mathbf{w}_\varepsilon \big) dxdt
\\
&  = 2 \int_0^T \int_\Omega \big[\varrho_\varepsilon (\mathbf{u}_\varepsilon + \varepsilon \nabla_x \log \varrho_\varepsilon)
     \varepsilon\nabla_x \log \varrho^E - \varrho_\varepsilon (\mathbf{u}_\varepsilon+\varepsilon\nabla_x
     \log \varrho_\varepsilon) \varepsilon \nabla \log \varrho_\varepsilon \big]dxdt
\\
&\quad  - 2 \int_0^T \int_\Omega \big[\varrho_\varepsilon (\overline{\mathbf{u}}+\varepsilon\nabla_x \log \varrho^E)
  \varepsilon \nabla_x \log \varrho^E - \varrho_\varepsilon (\overline{\mathbf{u}}+\varepsilon\nabla_x \log \varrho^E)
  \varepsilon \nabla_x \log \varrho_\varepsilon \big]dxdt
\\
&  = 2 \varepsilon \int_0^T \int_\Omega \varrho_\varepsilon \mathbf{u}_\varepsilon \nabla_x \log \varrho^E  dxdt
  + 2 \varepsilon^2 \int_0^T \int_\Omega \varrho_\varepsilon \nabla_x \log \varrho_\varepsilon \nabla_x \log \varrho^E  dxdt
\\
&\quad  - 2 \varepsilon \int_0^T \int_\Omega \varrho_\varepsilon \mathbf{u}_\varepsilon \nabla_x \log \varrho_\varepsilon  dxdt
  - 2 \varepsilon^2 \int_0^T \int_\Omega  \varrho_\varepsilon |\nabla_x \log \varrho_\varepsilon|^2 dxdt
  \\
 & \quad - 2 \varepsilon \int_0^T \int_\Omega \varrho_\varepsilon \overline{\mathbf{u}} \nabla_x \log \varrho^E dxdt
   -2 \varepsilon^2 \int_0^T \int_\Omega \varrho_\varepsilon |\nabla_x \log \varrho^E|^2 dxdt
  \\
& \quad  + 2 \varepsilon \int_0^T \int_\Omega \varrho_\varepsilon \overline{\mathbf{u}} \nabla_x \log \varrho_\varepsilon dxdt
       + 2 \varepsilon^2 \int_0^T \int_\Omega \varrho_\varepsilon \nabla_x \log \varrho^E \nabla_x \log \varrho_\varepsilon  dxdt,
\end{align*}
\endgroup
and hence
\begingroup
\allowdisplaybreaks
\begin{align*}
    2 &\int_0^T\int_\Omega  \varrho_\varepsilon  (\mathbf{v}_\varepsilon - \overline{\mathbf{v}})
  (\overline{\mathbf{w}} - \mathbf{w}_\varepsilon) dxdt\\
&  = \bigg[ 2\varepsilon \int_0^T \int_\Omega \sqrt{\varrho_\varepsilon} \sqrt{\varrho_\varepsilon} \mathbf{u}_\varepsilon 
  \nabla_x \log \varrho^E  dxdt + 2 \varepsilon^2 \int_0^T \int_\Omega \nabla_x \sqrt{\varrho_\varepsilon} \sqrt{\varrho_\varepsilon}
\nabla_x \log \varrho^E dxdt
\\
 &\quad    - 2 \varepsilon \int_0^T \int_\Omega \sqrt{\varrho_\varepsilon} \mathbf{u}_\varepsilon \nabla_x \sqrt{\varrho_\varepsilon} dxdt
   - 8 \varepsilon^2 \int_0^T \int_\Omega |\nabla_x \sqrt{\varrho_\varepsilon}|^2 dxdt
\\
&\quad    - 2 \varepsilon \int_0^T \int_\Omega \varrho_\varepsilon \overline{\mathbf{u}} \nabla_x \log \varrho^E
    -2 \varepsilon^2 dxdt \int_0^T \int_\Omega \varrho_\varepsilon |\nabla_x \log \varrho^E|^2 dxdt
\\
&\quad    + 2 \varepsilon \int_0^T \int_\Omega \sqrt{\varrho_\varepsilon} \overline{\mathbf{u}} \nabla_x \sqrt{\varrho_\varepsilon}
    + 2 \varepsilon^2 dxdt \int_0^T \int_\Omega \nabla_x \log \varrho^E \nabla_x  \sqrt{\varrho_\varepsilon} \sqrt{\varrho_\varepsilon} dxdt
    \bigg]
  \longrightarrow 0 \,\, \text{ as } \,\, \varepsilon \to 0.
\end{align*}
\endgroup
%
%

%
%
Next we turn to $\widetilde{R_{5,3}}$, to get
\begin{equation} \label{R43}
  \begin{aligned}
\int_0^T \int_\Omega \varrho_\varepsilon \mathbf{u}_\varepsilon \nabla_x \mathbf{v}_{bl}
  (\mathbf{u}^E  - \mathbf{v}_{bl} -  \mathbf{u}_\varepsilon) dxdt
  & = \int_0^T \int_\Omega \varrho_\varepsilon \mathbf{u}_\varepsilon \nabla_x \mathbf{v}_{bl}
   \overline{\mathbf{u}} dxdt
   \\
 &\quad   - \int_0^T \int_\Omega \varrho_\varepsilon \mathbf{u}_\varepsilon
  \nabla_x \mathbf{v}_{bl} \mathbf{u}_\varepsilon dxdt.
\end{aligned}
\end{equation}
For the first term on the right-hand side, we have
\begin{equation} \label{ancora-pezzi-intermedi}
\hspace{-0.3 cm}\begin{aligned} 
  \int_0^T\!\! \int_\Omega \varrho_\varepsilon \mathbf{u}_\varepsilon
  \nabla_x \mathbf{v}_{bl} \overline{\mathbf{u}} dxdt
  & =
  \int_0^T\!\! \int_\Omega \varrho_\varepsilon \mathbf{u}_\varepsilon
  \nabla_x (z \mathbf{u}^E) \overline{\mathbf{u}} dxdt
\\
&  = \int_0^T\!\! \int_\Omega \varrho_\varepsilon \mathbf{u}_\varepsilon
  z \nabla_x \mathbf{u}^E \overline{\mathbf{u}} dxdt
  +
  \int_0^T\!\! \int_\Omega \varrho_\varepsilon \mathbf{u}_\varepsilon
  \cdot \nabla_x z\otimes \mathbf{u}^E \overline{\mathbf{u}}\, dxdt
\\
 & = \bigg[\int_0^T\!\! \int_\Omega \sqrt{\varrho_\varepsilon}
  \sqrt{\varrho_\varepsilon} \mathbf{u}_\varepsilon z \nabla_x \mathbf{u}^E
   \overline{\mathbf{u}} dxdt
   \\
 & \quad +
   \int_0^T\!\! \int_\Omega \varrho_\varepsilon
   \mathbf{u}_\varepsilon  \cdot \frac{1}{d_\Omega(x)}\widetilde{z}  n\otimes
  \mathbf{u}^E \overline{\mathbf{u}} 
  dxdt \bigg] \longrightarrow 0  \,\, \text{ as } \,\, \varepsilon \to 0.
\end{aligned}
\end{equation}
  Here, the first term on the right-hand side of \eqref{ancora-pezzi-intermedi} can be easily
  bounded as done \eqref{passaggio-al-limite-R41}, i.e.
  \begin{align*}
 \left| \int_0^T \int_\Omega z \sqrt{\varrho_\varepsilon} \sqrt{\varrho_\varepsilon} \mathbf{u}_\varepsilon  \nabla_x \mathbf{u}^E
   \overline{\mathbf{u}} dxdt\right|
 \leq & \int_0^T \Big(\|\sqrt{\varrho_\varepsilon}\|_{L^{2\gamma}(\Omega)} \|\sqrt{\varrho_\varepsilon}\,
 \mathbf{u}_\varepsilon\|_{L^2(\Omega)} \| z \|_{L^{\frac{2\gamma}{\gamma-1}(\Omega))}}\\
 &\qquad \times \|\overline{\mathbf{u}}\|_{L^\infty(\Omega)}\|\nabla_x \mathbf{u}^E\|_{L^\infty(\Omega)} \Big)dt,
\end{align*}
where we used \eqref{reg-prop} and, in particular, relations \eqref{v-prop}  for passing to the limit. The second term, instead,
can be controlled as follows
\begin{align*}
  \bigg|\int_0^T \int_\Omega \varrho_\varepsilon \mathbf{u}_\varepsilon \cdot \frac{1}{d_\Omega(x)}\widetilde{z}  n\otimes
  \mathbf{u}^E \overline{\mathbf{u}}dxdt\bigg|
  \leq \hat C\int_0^T\|\sqrt{\rho_\varepsilon}\|_{L^{2\gamma}(\Omega)}
  \left\|\frac{\sqrt{\varrho_\varepsilon}\, \mathbf{u}_\varepsilon}{d_\Omega}\right\|_{L^2(\Omega)} \| \tilde z \|_{L^\infty(\Omega)}
  \|1\|_{L^{\frac{2\gamma}{\gamma-1}(\Gamma_\varepsilon))}} dt 
\end{align*}
with $\hat C = \hat C (\|\mathbf{u}^E\|_{L^\infty([0, T]\times \Omega)}, \|\bar{\mathbf{u}}\|_{L^\infty([0, T]\times \Omega)})$,
under the assumption (\ref{cond-conv})$_2$
and exploiting that $\textrm{supp}\,  \mathbf{v}_{bl}\subseteq [0,T] \times \Gamma_\varepsilon$.

For the second term on the right-hand side of \eqref{R43}, we have
\begin{align*}
  \int_0^T \int_\Omega \varrho_\varepsilon \mathbf{u}_\varepsilon \nabla_x \mathbf{v}_{bl} \mathbf{u}_\varepsilon dxdt
 & = \int_0^T \int_\Omega \frac{\sqrt{\varrho_\varepsilon} \mathbf{u}_\varepsilon}{d_\Omega(x)} \frac{\sqrt{\varrho_\varepsilon}
  \mathbf{u}_\varepsilon}{d_\Omega(x)} d^2_\Omega(x) \nabla_x \mathbf{v}_{bl}
  \\
  & \leq \| d^2_\Omega(x) \nabla_x \mathbf{v}_{bl}\|_{L^\infty([0, T]\times \Omega)}
    \int_0^T \int_{\Gamma_\varepsilon} \frac{\varrho_\varepsilon|\mathbf{u}_\varepsilon|^2}{d_\Omega^2(x)}dxdt
    \longrightarrow  0 \,\, \textrm{ as } \,\, \varepsilon \to 0,
\end{align*}
still using that the support of $\mathbf{v}_{bl}$ in contained in $[0,T] \times \Gamma_\varepsilon$, and thanks to the fact that
\begin{equation} \label{cond-conv-cons}
  \varepsilon \int_0^T \int_{\Gamma_\varepsilon} \frac{\varrho_\varepsilon|\mathbf{u}_\varepsilon|^2}{d_\Omega^2(x)}dxdt
  \to 0 \,\, \textrm{ as } \,\, \varepsilon \to 0,
\end{equation}
which is a direct consequence of (\ref{cond-conv})$_2$.

Finally, for $\widetilde{R_{5,4}}$, we have
\begin{equation*} 
\begin{aligned}
  \int_0^T \int_\Omega   \varrho_\varepsilon \nabla_x \mathbf{v}_{bl} & (\mathbf{v}_{bl} - \mathbf{u}^E)
  (\mathbf{u}^E - \mathbf{v}_{bl} - \mathbf{u}_\varepsilon) dxdt
  \\
  &= - \int_0^T \int_\Omega \varrho_\varepsilon \nabla_x \mathbf{v}_{bl}
    \overline{\mathbf{u}} (\overline{\mathbf{u}} - \mathbf{u}_\varepsilon)dxdt
  \\
  &  = - \int_0^T \int_\Omega \varrho_\varepsilon \nabla_x \mathbf{v}_{bl} \overline{\mathbf{u}}\, \overline{\mathbf{u}} dxdt
  \\
  &\quad  + \int_0^T \int_\Omega \varrho_\varepsilon \nabla_x \mathbf{v}_{bl}
    \overline{\mathbf{u}} \mathbf{u}_\varepsilon  dxdt,
  \end{aligned}
\end{equation*}
where for the first term we have
\begin{align*}
    \int_0^T \int_\Omega \varrho_\varepsilon \nabla_x \mathbf{v}_{bl}
    \overline{\mathbf{u}} \ \overline{\mathbf{u}} dxdt
    &\leq   \|  \overline{\mathbf{u}} \|_{L^\infty([0,T];L^\infty(\Omega))}^2
    \int_0^T \| \varrho_\varepsilon \|_{L^\gamma(\Omega)} \| \nabla_x
    \mathbf{v}_{bl}  \|_{L^\infty(\Omega)} \| 1 \|_{L^{\frac{\gamma}{\gamma-1}}(\Gamma_\varepsilon)}   dt
\\
   & \leq  C \varepsilon^{-\frac{1}{\gamma}} \| \varrho_\varepsilon \|_{L^\gamma([0,T];L^\gamma(\Omega))}
    \to 0 \, \, \textrm{ as } \, \,  \varepsilon \to 0 
\end{align*}
under the assumptions (\ref{cond-conv})$_1$, while the second term can be handle similarly
to $\widetilde{R_{5,3}}$ and goes to zero as $\varepsilon \to 0$.

Finally, we consider the second and third term in \eqref{r4}. We have,
\begin{align*}
\int_0^T \int_\Omega &\varrho_\varepsilon (\overline{\mathbf{u}}-\mathbf{u}_\varepsilon)
\nabla_x \overline{\mathbf{u}} \ \mathbf{v}_{bl} dxdt - \int_0^T \int_\Omega \varrho_\varepsilon
(\overline{\mathbf{u}}-\mathbf{u}_\varepsilon) \mathbf{u}^E \nabla_x \mathbf{v}_{bl}dxdt
\\
&=
\int_0^T \int_\Omega \varrho_\varepsilon (\overline{\mathbf{u}}-\mathbf{u}_\varepsilon) 
\nabla_x \mathbf{u}^E \mathbf{v}_{bl} dxdt
- \int_0^T \int_\Omega \varrho_\varepsilon (\overline{\mathbf{u}}-\mathbf{u}_\varepsilon)
\nabla_x \mathbf{v}_{bl} \mathbf{v}_{bl} dxdt
     \\
&\quad -\int_0^T \int_\Omega \varrho_\varepsilon \overline{\mathbf{u}}
\ \mathbf{u}^E \nabla_x \mathbf{v}_{bl} dxdt +\int_0^T \int_\Omega \varrho_\varepsilon
\mathbf{u}_\varepsilon \mathbf{u}^E \nabla_x \mathbf{v}_{bl} dxdt,
\end{align*}
where the first term is the same as in $\widetilde{R_{5,1}}$
while the other terms could be handled with similar argument as in
$\widetilde{R_{5,4}}$.
In conclusions, we have $R_5 \to 0$ as $\varepsilon \to 0$.

For $R_6$, we have
\begin{align*}
   \varepsilon \int_0^T \int_\Omega & \varrho_\varepsilon \nabla_x \overline{\mathbf{u}}
   \cdot \mathbf{u}_\varepsilon \cdot  (\nabla_x \log \varrho^E
  - \nabla_x \log \varrho_\varepsilon)dxdt
  \\
  &=  \varepsilon \int_0^T \int_\Omega \varrho_\varepsilon \nabla_x 
  (\mathbf{u}^E - \mathbf{v}_{bl}) \cdot \mathbf{u}_\varepsilon
  \cdot (\nabla_x \log \varrho^E - \nabla_x \log \varrho_\varepsilon) dxdt
\\
&  = \varepsilon \int_0^T \int_\Omega \varrho_{\varepsilon} \nabla_x 
  \mathbf{u}^E \cdot \mathbf{u}_\varepsilon
     \cdot   (\nabla_x \log \varrho^E - \nabla_x \log \varrho_\varepsilon) dxdt
  \\
&\quad   - \varepsilon \int_0^T \int_\Omega \varrho_\varepsilon \nabla_x 
 \mathbf{v}_{bl} \cdot \mathbf{u}_\varepsilon \cdot (\nabla_x \log \varrho^E
 - \nabla_x \log \varrho_\varepsilon)  dxdt
\\
&  = \varepsilon \int_0^T \int_\Omega \varrho_\varepsilon \nabla_x 
  \mathbf{u}^E  \cdot \mathbf{u}_\varepsilon \cdot \nabla_x \log \varrho^E
  dxdt  -  \varepsilon \int_0^T \int_\Omega \varrho_\varepsilon \nabla_x 
 \mathbf{u}^E  \cdot \mathbf{u}_\varepsilon  \cdot \nabla_x \log \varrho_\varepsilon  dxdt
\\
&\quad   - \varepsilon \int_0^T \int_\Omega \varrho_\varepsilon \nabla_x 
  \mathbf{v}_{bl} \cdot \mathbf{u}_\varepsilon  \cdot \nabla_x \log \varrho^E 
  dxdt  + \varepsilon \int_0^T \int_\Omega \varrho_\varepsilon \nabla_x 
  \mathbf{v}_{bl} \cdot \mathbf{u}_\varepsilon \cdot \nabla_x \log \varrho_\varepsilon  dxdt
\end{align*}
and so
\begin{equation} \label{conto-utile}
\hspace{-0.35 cm}\begin{aligned}
   & \varepsilon  \! \int_0^T\!\! \int_\Omega\! \varrho_\varepsilon \nabla_x
    \overline{\mathbf{u}} \cdot \mathbf{u}_\varepsilon \cdot (\nabla_x
    \log \varrho^E - \nabla_x \log \varrho_\varepsilon)dxdt
    \\
    & = \varepsilon \int_0^T \!\!\int_\Omega\! \sqrt{\varrho_\varepsilon}
    \sqrt{\varrho_\varepsilon} \mathbf{u}_\varepsilon \cdot \nabla_x
    \mathbf{u}^E \cdot \nabla_x \log \varrho^E dxdt - \varepsilon
    \int_0^T\!\!\int_\Omega\! \nabla_x \sqrt{\varrho_\varepsilon } \cdot
    \sqrt{\varrho_\varepsilon } \mathbf{u}_\varepsilon \cdot \nabla_x
    \mathbf{u}^E dxdt
    \\
    &\quad + \varepsilon \int_0^T\!\! \int_\Omega\! \nabla_x
    \varrho_\varepsilon \cdot \mathbf{v}_{bl} \cdot
    \mathbf{u}_\varepsilon \cdot \nabla_x \log \varrho^E dxdt +
    \varepsilon \int_0^T \int_\Omega \varrho_\varepsilon
    \mathbf{v}_{bl} \cdot \nabla_x \mathbf{u}_\varepsilon \cdot
    \nabla_x \log \varrho^E dxdt
    \\
    &\quad + \varepsilon \int_0^T \int_\Omega \varrho_\varepsilon
    \mathbf{v}_{bl} \cdot \mathbf{u}_\varepsilon \cdot \nabla_x
    \nabla_x \log \varrho^E dxdt + \varepsilon \int_0^T \int_\Omega
    \varrho_\varepsilon \nabla_x \mathbf{v}_{bl} \cdot
    \mathbf{u}_\varepsilon \cdot \nabla_x \log \varrho_\varepsilon dxdt
      \end{aligned}
    \end{equation}
    \begin{equation*}
     = \left[ \varepsilon \int_0^T \int_\Omega
      \sqrt{\varrho_\varepsilon} \sqrt{\varrho_\varepsilon}
      \mathbf{u}_\varepsilon \cdot \nabla_x \mathbf{u}^E \cdot
      \nabla_x \log \varrho^E dxdt - \varepsilon \int_0^T \int_\Omega
      \nabla_x \sqrt{\varrho_\varepsilon } \cdot
      \sqrt{\varrho_\varepsilon } \mathbf{u}_\varepsilon \cdot
      \nabla_x \mathbf{u}^E dxdt \right.
  \end{equation*}
\begin{equation*}
  \begin{aligned}
    &\qquad\,\, + \varepsilon \int_0^T \int_\Omega \nabla_x
    \sqrt{\varrho_\varepsilon} \sqrt{\varrho_\varepsilon}
    \mathbf{u}_\varepsilon \cdot \mathbf{v}_{bl} \cdot \nabla_x \log
    \varrho^E dxdt + \varepsilon \int_0^T \int_\Omega
    \sqrt{\varrho_\varepsilon} \mathbf{v}_{bl} \cdot
    \sqrt{\varrho_\varepsilon} \nabla_x
      \mathbf{u}_\varepsilon  \cdot \nabla_x \log \varrho^E dxdt
    \\
    & \qquad\,\, + \varepsilon \int_0^T \int_\Omega
    \sqrt{\varrho_\varepsilon} \sqrt{\varrho_\varepsilon}
    \mathbf{u}_\varepsilon \cdot \mathbf{v}_{bl} \cdot \nabla_x
    \nabla_x \log \varrho^E dxdt
 \left.  + \varepsilon \int_0^T
     \int_\Omega \varrho_\varepsilon \nabla_x \mathbf{v}_{bl} \cdot
     \mathbf{u}_\varepsilon \cdot \nabla_x \log \varrho_\varepsilon
      dxdt \right] \underset{\!\!\!\! \varepsilon \to 0}{-\hspace{-0.2 cm}\longrightarrow 0}.
  \end{aligned}
\end{equation*}
Observe that, the term with integrating function
$f(x, t)=\sqrt{\varrho_\varepsilon} \mathbf{v}_{bl} \cdot
\sqrt{\varrho_\varepsilon} \nabla_x
\mathbf{u}_\varepsilon  \cdot \nabla_x \log \varrho^E$
 can be rewritten in terms of those already present in
 brackets. Indeed, integrating by parts, we have
 \begin{equation*}
   \begin{aligned}
     \varepsilon \int_0^T &\int_\Omega \sqrt{\varrho_\varepsilon}
     \mathbf{v}_{bl} \cdot \sqrt{\varrho_\varepsilon} \nabla_x
     \mathbf{u}_\varepsilon \cdot \nabla_x \log \varrho^E dxdt
     \\
     & = -2 \varepsilon \int_0^T \int_\Omega \nabla_x
     \sqrt{\varrho_\varepsilon} \mathbf{v}_{bl} \cdot
     \sqrt{\varrho_\varepsilon} \mathbf{u}_\varepsilon \cdot \nabla_x
     \log \varrho^E dxdt - \varepsilon \int_0^T \int_\Omega
     \sqrt{\varrho_\varepsilon} \nabla_x\mathbf{v}_{bl} \cdot
     \sqrt{\varrho_\varepsilon} \mathbf{u}_\varepsilon \cdot \nabla_x
     \log \varrho^E dxdt
     \\
     &\qquad - \varepsilon \int_0^T \int_\Omega
     \sqrt{\varrho_\varepsilon} \mathbf{v}_{bl} \cdot
     \sqrt{\varrho_\varepsilon} \mathbf{u}_\varepsilon \cdot
     \nabla_x\nabla_x \log \varrho^E dxdt.
   \end{aligned}
 \end{equation*}
 
 In particular, for the last term on the right-hand side of
 \eqref{conto-utile}, we have that
 \begin{equation*}
   \varepsilon \int_0^T \int_\Omega \varrho_\varepsilon \nabla_x 
   \mathbf{v}_{bl} \cdot \mathbf{u}_\varepsilon \cdot \nabla_x \log \varrho_\varepsilon
   dxdt \leq \varepsilon \int_0^T \int_\Omega \varepsilon \nabla_x \mathbf{v}_{bl}
   \cdot \frac{\sqrt{\varrho_\varepsilon} \mathbf{u}_\varepsilon}{d_\Omega(x)}
   \cdot \nabla_x \sqrt{\varrho_\varepsilon} dxdt \to 0 \,\, \textrm{ as } \,\, \varepsilon \to 0
 \end{equation*}
thanks to (\ref{cond-conv-cons}).

For $R_7$, we have
\begin{equation*}
  \varepsilon \int_0^T \int_\Omega \varrho_\varepsilon \nabla_x \nabla_x \log \varrho^E \cdot
\mathbf{u}_\varepsilon \cdot (\overline{\mathbf{u}} -
\mathbf{u}_\varepsilon)dxdt
\end{equation*}
\begin{equation*}
  = \varepsilon \int_0^T \int_\Omega
  \varrho_\varepsilon \nabla_x \nabla_x \log \varrho^E \cdot \mathbf{u}_\varepsilon
  \cdot
  (\mathbf{u}^E - \mathbf{v}_{bl} - \mathbf{u}_\varepsilon)
  dxdt
\end{equation*}
\begin{equation*}
  = \left[ \varepsilon \int_0^T    \int_\Omega
    \sqrt{\varrho_\varepsilon}
    \sqrt{\varrho_\varepsilon}
    \mathbf{u}_\varepsilon
    \cdot
    \nabla_x \nabla_x \log \varrho^E 
    \cdot
    \mathbf{u}^E
    dxdt
    - \varepsilon \int_0^T \int_\Omega
    \sqrt{\varrho_\varepsilon}
    \sqrt{\varrho_\varepsilon}
    \mathbf{u}_\varepsilon
    \cdot
    \nabla_x \nabla_x \log \varrho^E 
    \cdot
    \mathbf{v}_{bl}
    dxdt
  \right.
\end{equation*}
\begin{equation*}
  \left. - \varepsilon \int_0^T \int_\Omega
    \sqrt{\varrho_\varepsilon}
    \mathbf{u}_\varepsilon
    \cdot
    \sqrt{\varrho_\varepsilon}
    \mathbf{u}_\varepsilon
    \cdot
    \nabla_x \nabla_x \log \varrho^E
    dxdt
  \right]
  \longrightarrow 0 
  \ \ \ \ \text{as} 
  \ \ \ \ \varepsilon \to 0.
\end{equation*}

For $R_8$, we have
\begin{equation*}
  2\varepsilon^2 \int_0^T \int_\Omega
  \varrho_\varepsilon \nabla_x \nabla_x \log \varrho^E \cdot
  \mathbf{u}_\varepsilon \cdot (\nabla_x \log \varrho^E - \nabla_x \log
  \varrho_\varepsilon)dxdt
\end{equation*}
\begin{equation*}
  = \left[ 2\varepsilon^2 \int_0^T \int_\Omega
    \sqrt{\varrho_\varepsilon}
    \sqrt{\varrho_\varepsilon }
    \mathbf{u}_\varepsilon 
    \cdot
    \nabla_x \nabla_x \log \varrho^E \cdot
    \nabla_x \log \varrho^E
    dxdt
    \right.
\end{equation*}    
\begin{equation*}
    \left.
    -
    2\varepsilon^2 \int_0^T \int_\Omega
    \nabla_x \sqrt{\varrho_\varepsilon}
    \sqrt{\varrho_\varepsilon }
    \mathbf{u}_\varepsilon 
    \cdot
    \nabla_x \nabla_x \log \varrho^E
    dxdt
  \right]
  \longrightarrow 0 
  \ \ \ \ \text{as} 
  \ \ \ \ \varepsilon \to 0.
\end{equation*}

\subsubsection{Viscous and pressure terms} \label{subsec:viscous}
For $R_9$, we have
\begin{align*}
\varepsilon\int_0^T &\int_{\Omega} \varrho_\varepsilon \mathbb{D}(\mathbf{v}_\varepsilon) : \nabla_{x} \overline{\mathbf{v}}dxdt
+\varepsilon\int_0^T \int_{\Omega} \varrho_\varepsilon \mathbb{A}(\mathbf{v}_\varepsilon) : \nabla_{x} \overline{\mathbf{v}} dxdt
- \varepsilon\int_0^T \int_{\Omega}
\varrho_\varepsilon \nabla_x \mathbf{w}_\varepsilon: \nabla_{x} \overline{\mathbf{v}} dxdt
\\
&\qquad +\varepsilon\int_0^T \int_{\Omega} \varrho_\varepsilon \nabla_x^{\top} \mathbf{u}_\varepsilon : \nabla_{x} \overline{\mathbf{w}} dxdt
-2\varepsilon \int_0^T \int_\Omega \left( \sqrt{\varrho_\varepsilon} \mathcal{S}(\mathbf{u}_\varepsilon) \mathbb{D}(\overline{\mathbf{u}})
+ \sqrt{\varrho_\varepsilon} \mathcal{A}(\mathbf{u}_\varepsilon) \mathbb{A}(\overline{\mathbf{u}}) \right) dxdt
\\
&\qquad  +\varepsilon \int_0^T \int_\Omega \left(\varrho_\varepsilon \left( |\mathbb{D}(\overline{\mathbf{u}})|^2 + |\mathbb{A}(\overline{\mathbf{u}})|^2
  \right) \right) dxdt.
\\
& 
= \left[
    -\varepsilon \int_0^T \int_{\Omega}
    \sqrt{\varrho_\varepsilon}\sqrt{\varrho_\varepsilon}(\mathbf{v}_\varepsilon)_j \partial_{ii} (\overline{\mathbf{v}})_j
    - 2\varepsilon \int_0^T \int_{\Omega}
    \sqrt{\varrho_\varepsilon}(\mathbf{v}_\varepsilon)_j \partial_i \sqrt{\varrho_\varepsilon} \partial_i (\overline{\mathbf{v}})_j
\right]
\\
&\qquad
\left\{
    - \varepsilon \int_0^T \int_{\Omega}
    \sqrt{\varrho_\varepsilon}\sqrt{\varrho_\varepsilon}(\mathbf{v}_\varepsilon)_i \partial_{ji} (\overline{\mathbf{v}})_j
    - 2\varepsilon \int_0^T \int_{\Omega}
    \sqrt{\varrho_\varepsilon}(\mathbf{v}_\varepsilon)_i \partial_j \sqrt{\varrho_\varepsilon} \partial_i (\overline{\mathbf{v}})_j
\right\}
  \\
  &\qquad 
\left[
    -\varepsilon \int_0^T \int_{\Omega}
    \sqrt{\varrho_\varepsilon}\sqrt{\varrho_\varepsilon}(\mathbf{v}_\varepsilon)_j \partial_{ii} (\overline{\mathbf{v}})_j
    - 2\varepsilon \int_0^T \int_{\Omega}
    \sqrt{\varrho_\varepsilon}(\mathbf{v}_\varepsilon)_j \partial_i \sqrt{\varrho_\varepsilon} \partial_i (\overline{\mathbf{v}})_j
\right]
\\
&\qquad 
\left\{
    + \varepsilon \int_0^T \int_{\Omega}
    \sqrt{\varrho_\varepsilon}\sqrt{\varrho_\varepsilon}(\mathbf{v}_\varepsilon)_i \partial_{ji} (\overline{\mathbf{v}})_j
    + 2\varepsilon \int_0^T \int_{\Omega}
    \sqrt{\varrho_\varepsilon}(\mathbf{v}_\varepsilon)_i \partial_j \sqrt{\varrho_\varepsilon} \partial_i (\overline{\mathbf{v}})_j
\right\}
\\
&\qquad    - \varepsilon \int_0^T \int_{\Omega}
    \sqrt{\varrho_\varepsilon}\sqrt{\varrho_\varepsilon}(\mathbf{w}_\varepsilon)_j \partial_{ii} (\overline{\mathbf{v}})_j
    - 2\varepsilon \int_0^T \int_{\Omega}
    \sqrt{\varrho_\varepsilon}(\mathbf{w}_\varepsilon)_j \partial_i \sqrt{\varrho_\varepsilon} \partial_i (\overline{\mathbf{v}})_j
\\
&\qquad    - \varepsilon \int_0^T \int_{\Omega}
    \sqrt{\varrho_\varepsilon}\sqrt{\varrho_\varepsilon}(\mathbf{u}_\varepsilon)_i \partial_{ji} (\overline{\mathbf{w}})_j
    - 2\varepsilon \int_0^T \int_{\Omega}
     \sqrt{\varrho_\varepsilon}(\mathbf{u}_\varepsilon)_i \partial_j \sqrt{\varrho_\varepsilon} \partial_i (\overline{\mathbf{w}})_j
\\
& \qquad -2\varepsilon \int_0^T \int_\Omega \left( \sqrt{\varrho_\varepsilon} \mathcal{S}(\mathbf{u}_\varepsilon)
\mathbb{D}(\overline{\mathbf{u}}) + \sqrt{\varrho_\varepsilon} \mathcal{A}(\mathbf{u}_\varepsilon) \mathbb{A}(\overline{\mathbf{u}})
\right) dxdt
\\
&\qquad +\varepsilon \int_0^T \int_\Omega \left( \varrho_\varepsilon \left( |\mathbb{D}(\overline{\mathbf{u}})|^2
+ |\mathbb{A}(\overline{\mathbf{u}})|^2 \right) \right) dxdt
\\
&\qquad \qquad
\end{align*}
where the sum of the terms in $\{...\}$ is zero, and we add the term in $[...]$. Hence,
\begingroup
\allowdisplaybreaks
\begin{align*}
\varepsilon\int_0^T &\int_{\Omega} \varrho_\varepsilon \mathbb{D}(\mathbf{v}_\varepsilon) : \nabla_{x} \overline{\mathbf{v}}dxdt
+\varepsilon\int_0^T \int_{\Omega} \varrho_\varepsilon \mathbb{A}(\mathbf{v}_\varepsilon) : \nabla_{x} \overline{\mathbf{v}} dxdt
- \varepsilon\int_0^T \int_{\Omega} \varrho_\varepsilon \nabla_x \mathbf{w}_\varepsilon: \nabla_{x} \overline{\mathbf{v}} dxdt
\\
&\qquad +\varepsilon\int_0^T \int_{\Omega} \varrho_\varepsilon \nabla_x^{\top} \mathbf{u}_\varepsilon : \nabla_{x} \overline{\mathbf{w}} dxdt
-2\varepsilon \int_0^T \int_\Omega \left( \sqrt{\varrho_\varepsilon} \mathcal{S}(\mathbf{u}_\varepsilon) \mathbb{D}(\overline{\mathbf{u}})
+ \sqrt{\varrho_\varepsilon} \mathcal{A}(\mathbf{u}_\varepsilon) \mathbb{A}(\overline{\mathbf{u}}) \right) dxdt
\\
 &\qquad  +\varepsilon \int_0^T \int_\Omega \left(\varrho_\varepsilon \left( |\mathbb{D}(\overline{\mathbf{u}})|^2
 + |\mathbb{A}(\overline{\mathbf{u}})|^2 \right) \right) dxdt
\\
&  = -2\varepsilon \int_0^T \int_{\Omega} \sqrt{\varrho_\varepsilon}\sqrt{\varrho_\varepsilon}(\mathbf{v}_\varepsilon)_j \partial_{ii} (\mathbf{u}^E-\mathbf{v}_{bl}
     +\varepsilon\nabla_x \log \varrho^E)_jdxdt
     \\
&\qquad - 4\varepsilon \int_0^T \int_{\Omega} \sqrt{\varrho_\varepsilon}(\mathbf{v}_\varepsilon)_j \partial_i
     \sqrt{\varrho_\varepsilon} \partial_i (\mathbf{u}^E-\mathbf{v}_{bl}+\varepsilon\nabla_x \log \varrho^E)_jdxdt
\\     
& \qquad  - \varepsilon \int_0^T \int_{\Omega} \sqrt{\varrho_\varepsilon}\sqrt{\varrho_\varepsilon}(\mathbf{w}_\varepsilon)_j
  \partial_{ii} (\mathbf{u}^E-\mathbf{v}_{bl}+\varepsilon\nabla_x \log \varrho^E)_jdxdt
  \\
  &\qquad - 2\varepsilon \int_0^T \int_{\Omega} \sqrt{\varrho_\varepsilon}(\mathbf{w}_\varepsilon)_j \partial_i \sqrt{\varrho_\varepsilon}
    \partial_i (\mathbf{u}^E-\mathbf{v}_{bl}+\varepsilon\nabla_x \log \varrho^E)_jdxdt
\\
  &\qquad  - \varepsilon \int_0^T \int_{\Omega} \sqrt{\varrho_\varepsilon}\sqrt{\varrho_\varepsilon}(\mathbf{u}_\varepsilon)_i
    \partial_{ji} (\varepsilon \nabla_x \log \varrho^E)_jdxdt
    \\
 &\qquad  - 2\varepsilon \int_0^T \int_{\Omega} \sqrt{\varrho_\varepsilon}(\mathbf{u}_\varepsilon)_i \partial_j \sqrt{\varrho_\varepsilon}
     \partial_i (\varepsilon \nabla_x \log \varrho^E)_jdxdt
\\
& \qquad -2\varepsilon \int_0^T \int_\Omega \left( \sqrt{\varrho_\varepsilon} \mathcal{S}(\mathbf{u}_\varepsilon) 
\mathbb{D}(\mathbf{u}^E - \mathbf{v}_{bl})dxdt
+ \sqrt{\varrho_\varepsilon} \mathcal{A}(\mathbf{u}_\varepsilon) \mathbb{A}(\mathbf{u}^E - \mathbf{v}_{bl}) \right) dxdt
\\
&\qquad +\varepsilon \int_0^T \int_\Omega \left( \varrho_\varepsilon \left( |\mathbb{D}(\mathbf{u}^E - \mathbf{v}_{bl})|^2
+ |\mathbb{A}(\mathbf{u}^E - \mathbf{v}_{bl})|^2 \right) \right) dxdt \,\, \triangleq \sum_{i=1}^8 I_i.
\end{align*}
\endgroup
The terms $I_1...I_4$ that do not contain the boundary layer velocity $\mathbf{v}_{bl}$ go easily to zero as 
$\varepsilon \to 0$. Consequently, we concentrate on the terms containing $\mathbf{v}_{bl}$. We consider $I_1$ and $I_2$, namely
\begin{equation} \label{first}
  \begin{aligned}
& I_1 =  2\varepsilon \int_0^T \int_{\Omega}
    \sqrt{\varrho_\varepsilon}\sqrt{\varrho_\varepsilon}(\mathbf{v}_\varepsilon)_j \partial_{ii}(\mathbf{v}_{bl})_j
    dxdt
\\
=  2\varepsilon \int_0^T \int_{\Omega} \sqrt{\varrho_\varepsilon}\sqrt{\varrho_\varepsilon}  (\mathbf{u}_\varepsilon)_j
&\partial_{ii}(\mathbf{v}_{bl})_j  dxdt  +  2\varepsilon \int_0^T \int_{\Omega}
    \sqrt{\varrho_\varepsilon}\sqrt{\varrho_\varepsilon}(\mathbf{w}_\varepsilon)_j \partial_{ii}(\mathbf{v}_{bl})_j
    dxdt
  \end{aligned}
\end{equation}
and
\begin{equation*}
I_2 =
    4\varepsilon \int_0^T \int_{\Omega}
    \sqrt{\varrho_\varepsilon}(\mathbf{v}_\varepsilon)_j \partial_i \sqrt{\varrho_\varepsilon} \partial_i (\mathbf{v}_{bl})_j
    dxdt
\end{equation*}
\begin{equation} \label{second}
    =
    4\varepsilon \int_0^T \int_{\Omega}
    \sqrt{\varrho_\varepsilon}(\mathbf{u}_\varepsilon)_j \partial_i \sqrt{\varrho_\varepsilon} \partial_i (\mathbf{v}_{bl})_j
    dxdt
    +
    4\varepsilon \int_0^T \int_{\Omega}
    \sqrt{\varrho_\varepsilon}(\mathbf{w}_\varepsilon)_j \partial_i \sqrt{\varrho_\varepsilon} \partial_i (\mathbf{v}_{bl})_j
    dxdt.
\end{equation}
Now, setting $\check{\xi} (r) = r^2\xi''(r) $ and $\check{z}(x)=\check{\xi}\big(\frac{d_{\Omega}}{\varepsilon}\big)$, we have
\begin{equation} \label{djdi-vbl}
  \begin{aligned}
    \D_{ii}\mathbf{v}_{bl, j} = \frac{1}{d_\Omega}\widetilde{z}n_i\D_i\mathbf{u}^E_j
    + z\D_i^2\mathbf{u}^E_j +\frac{1}{d_\Omega}\widetilde{z}\D_i(n_i\mathbf{u}^E_j)+\frac{1}{d^2_{\Omega}}\check z n_i^2\mathbf{u}^E_j
  \end{aligned}
\end{equation}
with $z$ and $\widetilde z$ defined in \eqref{xiz}. Consequently, noticing that
\begin{equation*}
  \frac{\widetilde{z}}{d_\Omega(r)}
  = \frac{\xi'(r)}{c\varepsilon}, \ \ \ \ \check z = \frac{\xi''(r)}{(c\varepsilon)^2},
\end{equation*}
we can rewrite (\ref{djdi-vbl}) as follows
\begin{equation*}
  \D_{ii}\mathbf{v}_{bl, j} = \frac{\xi'(r)}{c\varepsilon}
  n_i\D_i\mathbf{u}^E_j
  + z\D_i^2\mathbf{u}^E_j +\frac{\xi'(r)}{c\varepsilon}\D_i(n_i\mathbf{u}^E_j)+\frac{\xi''(r)}{(c\varepsilon)^2} n_i^2\mathbf{u}^E_j.
\end{equation*}
%
%

%
%
The worst term in this relation is the last one which is of leading order with respect to $\varepsilon \to 0$.
Now, for $I_1$ and $I_2$, we focus and deal with such a term, then the remaining ones can be handled in a similar
way. 

  Exploiting H\"older's inequality we reach
  \begin{align*}
    2\varepsilon \int_0^T& \int_{\Omega} \sqrt{\varrho_\varepsilon} \sqrt{\varrho_\varepsilon}(\mathbf{v}_\varepsilon)_j
    \frac{\xi''(r)}{(c\varepsilon)^2} n_i^2\mathbf{u}^E_j dxdt
    \\
   & \leq  \varepsilon^2 C\int_0^T \int_{\Omega}
     \sqrt{\varrho_\varepsilon}
     \frac{\sqrt{\varrho_\varepsilon}(\mathbf{u}_\varepsilon)_j}
     {d_\Omega(x)} \frac{\xi''(r)}{(c\varepsilon)^2}
     n_i^2\mathbf{u}^E_j dxdt
    \\
   &\quad  +  2\varepsilon^2 \int_0^T \int_{\Omega}
     \sqrt{\varrho_\varepsilon} \sqrt{\varrho_\varepsilon}(\nabla_x\log \varrho_\varepsilon)_j
     \frac{\xi''(r)}{(c\varepsilon)^2}
     n_i^2\mathbf{u}^E_j dxdt
    \\
    &\leq \varepsilon^2C\int_0^T \| \sqrt{\varrho_\varepsilon} \|_{L^{2\gamma}(\Omega)}
    \left\| \frac{\sqrt{\varrho_\varepsilon}(\mathbf{u}_\varepsilon)_j}{d_\Omega(x)}
    \right\|_{L^2(\Omega)}  \left\| \frac{\xi''(r)}{(c\varepsilon)^2} n_i^2\mathbf{u}^E_j
    \right\|_{L^\infty(\Omega)}  \|  1  \|_{L^q(\Gamma_\varepsilon)} dt
    \\
&\quad +  \frac{2}{c} \int_0^T \int_{\Omega}
     \sqrt{\varrho_\varepsilon}(\nabla_x \sqrt{\varrho_\varepsilon})_j
     {\xi''(r)} n_i^2\mathbf{u}^E_j dxdt
   \\
  &\leq  \bar c \varepsilon^{\frac{1}{q}} \left\| \xi''(r) n_i^2\mathbf{u}^E_j
    \right\|_{L^\infty(0, T; L^\infty(\Omega))}  \int_0^T \| \sqrt{\varrho_\varepsilon} \|_{L^{2\gamma}(\Omega)}
    \left\| \frac{\sqrt{\varrho_\varepsilon}(\mathbf{u}_\varepsilon)_j}{d_\Omega(x)}
    \right\|_{L^2(\Omega)}  dt
    \\
    & \quad + \hat c \varepsilon^{\frac{1}{q}}\left\| \xi''(r) n_i^2\mathbf{u}^E_j
      \right\|_{L^\infty(0, T; L^\infty(\Omega))} \int_0^T  \| \sqrt{\varrho_\varepsilon} \|_{L^{2\gamma}(\Omega)}
  \left\| \nabla_x \sqrt{\varrho_\varepsilon}
  \right\|_{L^2(\Omega)}   dt
     \triangleq I_{11}+I_{12}
  \end{align*}
  where $\frac{1}{q}+\frac{1}{2\gamma}=\frac{1}{2}$ and
  \begin{equation*}
    \|1 \|_{L^q(\Gamma_\varepsilon)} 
    = \varepsilon^{1/q},\,\, 
    \textrm{ with }\,\, q = \frac{2\gamma}{\gamma - 1}.
  \end{equation*}
  By using the smoothness of $\xi''(r)n_i^2\mathbf{u}^E_j$ and
  applying H\"older's inequality in time, we infer
  \begingroup \allowdisplaybreaks
  \begin{align*}
   I_{11} &\leq c'\varepsilon^\frac{1}{q} \left\| \xi''(r) n_i^2\mathbf{u}^E_j
      \right\|_{L^\infty(0, T; L^\infty(\Omega))}  \left(\int_0^T \| \sqrt{\varrho_\varepsilon}
      \|_{L^{2\gamma}(\Omega)}^2dt\right)^\frac{1}{2} \left(\int_0^T \left\|
      \frac{\sqrt{\varrho_\varepsilon}(\mathbf{u}_\varepsilon)_j}{d_\Omega(x)}
      \right\|_{L^2(\Omega)}^2dt\right)^\frac{1}{2} 
    \\
    & \leq c''\varepsilon^\frac{1}{q} \left( \int_0^T \int_\Omega 
      \frac{{\varrho_\varepsilon}|(\mathbf{u}_\varepsilon)_j|^2}{d^2_\Omega(x)} dx dt
      \right)^{1/2}
    = c''  \left(\varepsilon^\frac{\gamma-1}{\gamma} \int_0^T \int_\Omega
      \frac{{\varrho_\varepsilon}|(\mathbf{u}_\varepsilon)_j|^2}{d^2_\Omega(x)}dxdt\right)^{1/2}
      %
  \end{align*}
  \endgroup
  Then, for $I_{11}$, we conclude
  \begin{equation*}
   I_{11}^2\leq \varepsilon^\frac{\gamma-1}{\gamma} C\int_0^T \int_\Omega
      \frac{{\varrho_\varepsilon}|(\mathbf{u}_\varepsilon)_j|^2}{d^2_\Omega(x)}dxdt
    \rightarrow 0\,\, \textrm{ as }\,\, \varepsilon \rightarrow 0,
  \end{equation*}
  under the assumption (\ref{cond-conv})$_2$. Moreover, for $I_{12}$,
  with analogous calculations, we have
  \begin{equation}
      I_{12}^2 
      \leq C  \varepsilon^\frac{\gamma-1}{\gamma} \int_0^T \int_\Omega
      |(\nabla_x \sqrt{\varrho_\varepsilon})_j|^2dxdt   \rightarrow
      0\,\, \textrm{ as }\,\, \varepsilon \rightarrow 0.
  \end{equation}

Similarly, as in $R_6$, the first term on the right-hand-side of \eqref{second} is such that
\begin{equation*}
  4\varepsilon \int_0^T \int_{\Omega}
  \sqrt{\varrho_\varepsilon}(\mathbf{u}_\varepsilon)_j
  \partial_i \sqrt{\varrho_\varepsilon} \partial_i (\mathbf{v}_{bl})_j
  \longrightarrow 0 \, \, \textrm{ as } \, \, \varepsilon \to 0,
\end{equation*}
thanks again to $\eqref{cond-conv}_2$.

Now, we consider the second term in the right-hand side of \eqref{second}. We have
\begin{equation*}
  4\varepsilon \int_0^T \int_{\Omega}
  \sqrt{\varrho_\varepsilon}(\mathbf{w}_\varepsilon)_j \partial_i \sqrt{\varrho_\varepsilon} \partial_i (\mathbf{v}_{bl})_j
  dxdt  =  8\varepsilon^2 \int_0^T \int_{\Omega}
  (\nabla_x  \sqrt{\varrho_\varepsilon})_j (\nabla_x\sqrt{\varrho_\varepsilon})_i \  \partial_i (\mathbf{v}_{bl})_j
  dxdt
\end{equation*}
that goes to zero as $\varepsilon \to 0$.  Similarly we can handle the
terms $I_3$ and $I_4$.

The terms $I_5$ and $I_6$ go to zero as $\varepsilon \to 0$. For $I_7$, we have
\begin{equation*}
\left[ -2\varepsilon \int_0^T \int_\Omega \sqrt{\varrho_\varepsilon}
\mathcal{S}(\mathbf{u}_\varepsilon) \mathbb{D}(\mathbf{u}^E)dxdt
+2\varepsilon \int_0^T \int_\Omega \sqrt{\varrho_\varepsilon} \mathcal{S}(\mathbf{u}_\varepsilon)
\mathbb{D}(\mathbf{v}_{bl})dxdt  \right] \longrightarrow 0 
  \,\,\, \text{as}   \,\,\, \varepsilon \to 0,
\end{equation*}
where, in particular, for the second term we have
\begin{equation*}
  \begin{aligned}
    2\varepsilon   \int_0^T \int_\Omega&  \sqrt{\varrho_\varepsilon}
    \mathcal{S}(\mathbf{u}_\varepsilon)  \mathbb{D}(\mathbf{v}_{bl})dxdt
\\
   & \leq  \varepsilon  \|  \mathbb{D}(\mathbf{v}_{bl}) \|_{L^\infty([0,T];L^\infty(\Omega))}
    \int_0^T  \|  \sqrt{\varrho_\varepsilon}  \|_{L^{2\gamma}(\Omega)}  \|
    \mathcal{S}(\mathbf{u}_\varepsilon)  \|_{L^2(\Omega)}
    \|   1  \|_{L^{\frac{2\gamma}{\gamma-1}}(\Gamma_\varepsilon)}  dt
\\
   & \leq C \varepsilon^{\frac{\gamma-1}{2\gamma}}
    \left(  \int_0^T  \|  \varrho_\varepsilon  \|_{L^\gamma(\Omega)}^\gamma
    \right)^{\frac{1}{2\gamma}}   \to 0 \, \, \textrm{ as } \, \, 
    \varepsilon \to 0,
  \end{aligned}
\end{equation*}
thanks to \eqref{cond-conv}$_1$. Similar analysis could be done for the term
\begin{equation*}
-2\varepsilon \int_0^T \int_\Omega \sqrt{\varrho_\varepsilon}
\mathcal{A}(\mathbf{u}_\varepsilon) \mathbb{A}(\mathbf{u}^E -
\mathbf{v}_{bl}) dxdt.
\end{equation*}

For $I_8$, we have
\begin{equation*}
\varepsilon \int_0^T \int_\Omega \varrho_\varepsilon |\mathbb{D}(\mathbf{u}^E)|^2dxdt
-2\varepsilon \int_0^T \int_\Omega  \varrho_\varepsilon  \mathbb{D}(\mathbf{u}^E)\mathbb{D}( \mathbf{v}_{bl})dxdt
+\varepsilon \int_0^T \int_\Omega \varrho_\varepsilon  |\mathbb{D}( \mathbf{v}_{bl})|^2dxdt,
\end{equation*}
where the first term goes to zero as $\varepsilon \to 0$, while the second and
the third term could be handled similar as above, namely
\begin{equation*}
  \begin{aligned}
  -2\varepsilon \int_0^T \int_\Omega & \varrho_\varepsilon
  \mathbb{D}(\mathbf{u}^E):\mathbb{D}(\mathbf{v}_{bl}) dxdt
  + \varepsilon \int_0^T \int_\Omega \varrho_\varepsilon
  \mathbb{D}(\mathbf{v}_{bl}): \mathbb{D}(\mathbf{v}_{bl})   dxdt
\\
&  \leq  C \varepsilon^{-\frac{1}{\gamma}}
  \| \varrho_\varepsilon \|_{L^\gamma([0,T];L^\gamma(\Omega))}
  \to 0 \,\,\, \textrm{ as } \,\,\, \varepsilon \to 0,
\end{aligned}
\end{equation*}
thanks again to \eqref{cond-conv}$_1$.
With a similar we can estimate the term
\begin{equation*}
\varepsilon \int_0^T \int_\Omega 
\varrho_\varepsilon |\mathbb{A}(\mathbf{u}^E - \mathbf{v}_{bl})|^2 dxdt.
\end{equation*}
In conclusions, we have $R_9 \to 0$ as $\varepsilon \to 0$.

For  the term $R_{10}$, we have
\begin{equation*}
  \begin{aligned}
-\int_0^T& \int_\Omega \left[-p(\varrho^E)\textrm{div}_x \mathbf{u}^E
+ p(\varrho_\varepsilon) \textrm{div}_x \overline{\mathbf{v}}
- p'(\varrho^E)(\varrho_\varepsilon-\varrho^E)\text{div}_x \mathbf{u}^E \right] dxdt
\\
& \quad + \varepsilon \int_0^T \int_\Omega
\frac{\varrho_\varepsilon}{\varrho^E}p'(r)\nabla_x \varrho^E \left(
\frac{\nabla_x \varrho^E}{\varrho^E} - \frac{\nabla_x \varrho_\varepsilon}{\varrho_\varepsilon}
\right) dxdt
\\
& \quad - \varepsilon \int_0^T \int_\Omega p'(\varrho_\varepsilon)\nabla_x \varrho_\varepsilon
\frac{\nabla_x \varrho^E}{\varrho^E} dxdt
\\
& = -\int_0^T \int_\Omega \Big[-p(\varrho^E)\textrm{div}_x \mathbf{u}^E
+ p(\varrho_\varepsilon) \textrm{div}_x \overline{\mathbf{u}}
\\
&\hspace{4 cm}
+ p(\varrho_\varepsilon) \textrm{div}_x \overline{\mathbf{w}}
- p'(\varrho^E)(\varrho_\varepsilon-\varrho^E)\text{div}_x \mathbf{u}^E
\Big] dxdt
\\
&\quad + \varepsilon \int_0^T \int_\Omega
\frac{\varrho_\varepsilon}{\varrho^E}p'(\varrho^E)\nabla_x \varrho^E \left(
\frac{\nabla_x \varrho^E}{\varrho^E} - \frac{\nabla_x \varrho_\varepsilon}{\varrho_\varepsilon}
\right) dxdt
\\
&\quad - \varepsilon \int_0^T \int_\Omega
p'(\varrho_\varepsilon)\nabla_x \varrho_\varepsilon \frac{\nabla_x \varrho^E}{\varrho^E} dxdt
\\
& = -\int_0^T \int_\Omega \Big[-p(\varrho^E)\textrm{div}_x \mathbf{u}^E
+ p(\varrho_\varepsilon) \textrm{div}_x (\mathbf{u}^E-\mathbf{v}_{bl})
\\
&\hspace{4 cm}
+ \varepsilon p(\varrho_\varepsilon) \textrm{div}_x \nabla_x \log \mathbf{u}^E
- p'(\varrho^E)(\varrho_\varepsilon-\varrho^E)\text{div}_x \mathbf{u}^E \Big] dxdt
\\
& \quad + \varepsilon \int_0^T \int_\Omega
\frac{\varrho_\varepsilon}{\varrho^E}p'(\varrho^E)\nabla_x \varrho^E \left(
\frac{\nabla_x \varrho^E}{\varrho^E} - \frac{\nabla_x \varrho_\varepsilon}{\varrho_\varepsilon}
\right) dxdt
\\
&\quad - \varepsilon \int_0^T \int_\Omega p'(\varrho_\varepsilon)\nabla_x
\varrho_\varepsilon \frac{\nabla_x \varrho^E}{\varrho^E} dxdt,
\end{aligned}
\end{equation*}
where
\begin{equation*}
- \varepsilon \int_0^T \int_\Omega p(\varrho_\varepsilon) \textrm{div}_x \nabla_x \log \mathbf{u}^E
dxdt - \varepsilon \int_0^T \int_\Omega p'(\varrho_\varepsilon)\nabla_x \varrho_\varepsilon
\frac{\nabla_x \varrho^E}{\varrho^E} dxdt = 0.
\end{equation*}
Consequently, we consider the following quantity
\begin{equation*}
  \begin{aligned}
-\int_0^T &\int_\Omega  \left[- p(\varrho^E)\textrm{div}_x \mathbf{u}^E
+ p(\varrho_\varepsilon) \textrm{div}_x (\mathbf{u}^E-\mathbf{v}_{bl})
- p'(\varrho^E)(\varrho_\varepsilon-\varrho^E)\text{div}_x \mathbf{u}^E \right] dxdt
\\
&\quad + \varepsilon \int_0^T \int_\Omega
\frac{\varrho_\varepsilon}{\varrho^E}p'(\varrho^E)\nabla_x \varrho^E \left(
\frac{\nabla_x \varrho^E}{\varrho^E} - \frac{\nabla_x \varrho_\varepsilon}{\varrho_\varepsilon}
\right) dxdt
\\
&  = -\int_0^T \int_\Omega  \left[-p(\varrho^E) + p(\varrho_\varepsilon)
- p'(\varrho^E)(\varrho_\varepsilon-\varrho^E) \right] \textrm{div}_x \mathbf{u}^E dxdt
\\
&\quad + \int_0^T \int_\Omega p(\varrho_\varepsilon) \text{div}_x \mathbf{v}_{bl} dxdt
+ \varepsilon \int_0^T \int_\Omega \frac{\varrho_\varepsilon}{\varrho^E}p'(\varrho^E)\nabla_x \varrho^E \left(
\frac{\nabla_x \varrho^E}{\varrho^E} - \frac{\nabla_x \varrho_\varepsilon}{\varrho_\varepsilon} \right) dxdt,
\end{aligned}
\end{equation*}
where
\begin{equation*}
\left| \int_0^T \int_\Omega \left[-p(\varrho^E)
+ p(\varrho_\varepsilon) - p'(\varrho^E)(\varrho_\varepsilon-\varrho^E) \right]
\textrm{div}_x \mathbf{u}^E dxdt \right| \leq C \int_0^T \mathcal{E}(t,\cdot) dt
\end{equation*}
and
\begin{equation*}
\left[ \varepsilon \int_0^T \int_\Omega
\frac{\varrho_\varepsilon}{\varrho^E}p'(\varrho^E)\nabla_x \varrho^E \left(
\frac{\nabla_x \varrho^E}{\varrho^E} - \frac{\nabla_x \varrho_\varepsilon}{\varrho_\varepsilon}
\right) dxdt \right] \longrightarrow 0  \,\,\, \text{as}  \,\,\, \varepsilon \to 0.
\end{equation*}
For the remaining term, we have
\begin{equation*}
\int_0^T \int_\Omega p(\varrho_\varepsilon) \text{div}_x \mathbf{v}_{bl} dxdt 
=  \int_0^T \int_\Omega (p(\varrho_\varepsilon) - p(\varrho^E)) \text{div}_x \mathbf{v}_{bl}dxdt
+ \int_0^T \int_\Omega p(\varrho^E) \text{div}_x \mathbf{v}_{bl} dxdt,
\end{equation*}
where the second term goes to zero as $\varepsilon \to 0$.
In order to handle the first term, we observe that since $p(\varrho_\varepsilon)$ is strictly convex, the quantity
\begin{equation*}
\varrho_\varepsilon \rightarrow H(\varrho_\varepsilon|\varrho^E)
= H(\varrho_\varepsilon) - H(\varrho^E) - H'(\varrho^E)(\varrho_\varepsilon - \varrho^E)
\end{equation*}
is non-negative strictly convex function on $(0,\infty)$ equal to zero when
$\varrho_\varepsilon = \varrho^E$ and growing at infinity with the rate $\varrho_\varepsilon^\gamma$. 
Consequently, the integral
$\int_{\Omega}H\left(\varrho_\varepsilon|\varrho^E\right)\left(t,\cdot\right)dx$ 
provides a control of $\left(\varrho_\varepsilon-\varrho^E\right)\left(t,\cdot\right)$
in $L^{2}$ over the sets $\left\{ x:\left|\varrho_\varepsilon-\varrho^E\right|\left(t,x\right)<1\right\} $
and in $L^{\gamma}$ over the sets $\left\{ x:\left|\varrho_\varepsilon-\varrho^E\right|\left(t,x\right)\geq\ 1\right\}$,
such that
\begin{equation} \label{H(rho|r)}
  H(\varrho_\varepsilon|\varrho^E)\approx\left|\varrho_\varepsilon
    -\varrho^E\right|^{2}1_{\left\{ \left|\varrho_\varepsilon-\varrho^E\right|<1\right\} }+\left|\varrho_\varepsilon
    -\varrho^E\right|^{\gamma}1_{\left\{ \left|\varrho_\varepsilon-\varrho^E\right|\geq 1\right\} },\;\;\;\forall\varrho_\varepsilon\geq0,
\end{equation}
in the sense that $H(\varrho_\varepsilon|\varrho^E)$ gives an upper and lower bound in term of the right-hand side quantity
(see, for example, Sueur \cite{Su} Section~2.1 relations (18)--(20)).
In particular, since $\Omega$ is bounded, there exists a constant $c>0$
such that (see Sueur \cite{Su} Section~2.1, relation (20))
\begin{equation*}
    c\| \varrho_\varepsilon - \varrho^E \|_{L^\gamma(\Omega)}
    \leq  \left(  \int_\Omega H(\varrho_\varepsilon|\varrho^E)dx
    \right)^{\gamma}  + 
    \int_\Omega H(\varrho_\varepsilon|\varrho^E)dx,
  \end{equation*}
\begin{equation} \label{H-gamma}
    c\int_\Omega H(\varrho_\varepsilon|\varrho^E)dx
    \leq     \|  \varrho_\varepsilon - \varrho^E   \|_{L^\gamma(\Omega)}^{\gamma}
    +\|  \varrho_\varepsilon - \varrho^E \|_{L^\gamma(\Omega)}^2.
\end{equation}
Consequently, we can write
\begin{equation*}
  \begin{aligned}
\int_0^T & \int_\Omega (p(\varrho_\varepsilon) - p(\varrho^E)) \text{div}_x \mathbf{v}_{bl}dxdt
\\
&= \int_0^T \int_\Omega
(p(\varrho_\varepsilon) -p'(\varrho^E)(\varrho_\varepsilon - \varrho^E) - p(\varrho^E))
\text{div}_x \mathbf{v}_{bl}dxdt
\\
&\quad + \int_0^T \int_{\Omega \cap {\left\{ \left|\varrho_\varepsilon-\varrho^E\right|<1\right\}}}
p'(\varrho^E)(\varrho_\varepsilon - \varrho^E) \text{div}_x \mathbf{v}_{bl}dxdt
\\
&\quad + \int_0^T \int_{\Omega \cap {\left\{ \left|\varrho_\varepsilon-\varrho^E\right|\geq 1\right\}}}
p'(\varrho^E)(\varrho_\varepsilon - \varrho^E) \text{div}_x \mathbf{v}_{bl}dxdt
\end{aligned}
\end{equation*}
and
\begin{equation*}
\left| \int_0^T \int_\Omega (p(\varrho_\varepsilon) - p(\varrho^E))
\text{div}_x \mathbf{v}_{bl}dxdt \right| \leq
C\varepsilon + C\int_0^T \mathcal{E}(t,\cdot) dt.
\end{equation*}

\subsubsection{Damping term}
For the term $R_{11}$, we have
\begin{align*}
r_1\int_0^T & \int_{\Omega}
\varrho_\varepsilon |\mathbf{v}_\varepsilon - \mathbf{w}_\varepsilon|(\mathbf{v}_\varepsilon - \mathbf{w}_\varepsilon)
\cdot \overline{\mathbf{v}} dxdt
\\
& = r_1\int_0^T \int_{\Omega} \varrho_\varepsilon |\mathbf{u}_\varepsilon |\mathbf{u}_\varepsilon
     \cdot (\overline{\mathbf{u}} + \overline{\mathbf{w}}) dxdr
     \\
& =
r_1\int_0^T \int_{\Omega} \varrho_\varepsilon |\mathbf{u}_\varepsilon |\mathbf{u}_\varepsilon
\cdot (\mathbf{u}^E - \mathbf{v}_{bl} + \varepsilon \nabla_x \log \varrho^E) dxdt
\\
&= 
\left[
r_1\int_0^T \int_{\Omega} \sqrt{\varrho_\varepsilon}|\mathbf{u}_\varepsilon| \sqrt{\varrho_\varepsilon}
\mathbf{u}_\varepsilon \mathbf{u}^E dxdt
- r_1\int_0^T \int_{\Omega}
\sqrt{\varrho_\varepsilon}|\mathbf{u}_\varepsilon| \sqrt{\varrho_\varepsilon}
\mathbf{u}_\varepsilon \mathbf{v}_{bl} dxdt \right.
\\
& \hspace{2 cm} \left.
+ r_1\varepsilon \int_0^T \int_{\Omega}
\sqrt{\varrho_\varepsilon}|\mathbf{u}_\varepsilon| \sqrt{\varrho_\varepsilon}
\mathbf{u}_\varepsilon \nabla_x \log \varrho^E
     dxdt \right] \longrightarrow 0  \,\,\, \text{as}
     \,\,\, (r_1,\varepsilon) \to 0.
\end{align*}

\subsubsection{Proof of Theorem \ref{main}}
From the previous estimates, back to (\ref{step-8}), we end up with
\begin{equation} \label{step-9}
 \hspace{-0.8 cm} \begin{aligned}
  \mathcal{E}(T,\cdot) - E(0,\cdot)   +\varepsilon &\int_0^T\!\! \int_\Omega\!\!
  \varrho_\varepsilon   (p'(\varrho_\varepsilon)\nabla_x \log \varrho_\varepsilon
  -p'(\varrho^E)\nabla_x \log r)(\nabla_x \log \varrho_\varepsilon - \nabla_x \varrho^E)   dxdt
\\
&  \leq 
C\varepsilon^{\frac{1}{p}} 
+C\varepsilon
+C\eta(\varepsilon) + C\int_0^T \mathcal{E}(t,\cdot)dt
\end{aligned}
\end{equation}
where $\eta(\varepsilon)$ is such that 
$\eta(\varepsilon) \to 0$ as $\varepsilon \to 0$ and represents the terms analyzed above that go to zero as $\varepsilon \to 0$.

In order to conclude the proof of Theorem \ref{main} we need to handle the term
\begin{equation*}
\varepsilon \int_0^T \int_\Omega   \varrho_\varepsilon
(p'(\varrho_\varepsilon)\nabla_x \log \varrho_\varepsilon-p'(\varrho^E)\nabla_x
\log \varrho^E)(\nabla_x \log \varrho_\varepsilon - \nabla_x \log \varrho^E)   dxdt.
\end{equation*}
This can be as in the same spirit of Bresch et al. \cite{BNV-2}. We have
\begin{equation} \label{press-contr}
  \begin{aligned}
    \varrho_\varepsilon & (p'(\varrho_\varepsilon)\nabla_x \log \varrho_\varepsilon
    - p'(\varrho^E)\nabla_x \log \varrho^E)(\nabla_x \log \varrho_\varepsilon - \nabla_x \log \varrho^E)
    \\
   & = \varrho_\varepsilon p'(\varrho_\varepsilon) |\nabla_x \log \varrho_\varepsilon - \nabla_x \log \varrho^E|^2
   + \varrho_\varepsilon (p'(\varrho_\varepsilon)
   \\
   &\hspace{4 cm}- p'(\varrho^E))\nabla_x \log \varrho^E (\nabla_x \log \varrho_\varepsilon - \nabla_x \log \varrho^E)
    \\
   & =\varrho_\varepsilon p'(\varrho_\varepsilon) |\nabla_x \log \varrho_\varepsilon - \nabla_x \log \varrho^E|^2  + \nabla_x\big[p(\varrho_\varepsilon) - p(\varrho^E)
    \\
    &\hspace{5 cm} - p'(\varrho^E)(\varrho_\varepsilon - \varrho^E)\big] \nabla_x \log \varrho^E
    \\
    &\quad - \left[\varrho_\varepsilon(p'(\varrho_\varepsilon) - p'(\varrho^E)) - p''(\varrho^E)(\varrho_\varepsilon - \varrho^E)\varrho^E\right]|\nabla_x\log\varrho^E|^2.
\end{aligned}
\end{equation} 
The first term on the right hand side of \eqref{press-contr} is positive and thus can be neglected in \eqref{step-9}.
For the remaining terms, integrating by parts, we have
\allowdisplaybreaks[4]
\begin{align*}
  \int_0^T \int_{\Omega} & \nabla_x \left[p(\varrho_\varepsilon) -
    p(\varrho^E) - p'(\varrho^E)(\varrho_\varepsilon -
    \varrho^E)\right] \nabla_x \log \varrho^E
  \\
 &\quad  - \left[\varrho_\varepsilon(p'(\varrho_\varepsilon) - p'(\varrho^E)) -
    p''(\varrho^E)(\varrho_\varepsilon - \varrho^E)\varrho^E
  \right] |\nabla_x\log\varrho^E|^2 dxdt
  \\
&  = - \int_0^T \int_{\Omega} |(p(\varrho_\varepsilon) - p(\varrho^E)
  - p'(\varrho^E) (\varrho_\varepsilon - \varrho^E)| |\Delta \log \varrho^E| dxdt
  \\
 &\quad  - \int_0^T \int_{\Omega} \left[\varrho_\varepsilon
    (p'(\varrho_\varepsilon) - p'(\varrho^E)) -
    p''(\varrho^E)(\varrho_\varepsilon - \varrho^E)\varrho^E\right]
  |\nabla_x \log \varrho^E|^2 dxdt.
\end{align*}
Now, from Lemma~2.2 in \cite{BaNg}, we observe that
\begin{equation*}
  \left[\varrho_\varepsilon (p'(\varrho_\varepsilon)
    - p'(\varrho^E)) - p''(\varrho^E)(\varrho_\varepsilon - \varrho^E)\varrho^E\right]
\approx H(\varrho_\varepsilon|\varrho^E).
\end{equation*}
Consequently, we have
\begin{equation*}
  \begin{aligned}
\varepsilon \int_0^T \int_\Omega \varrho_\varepsilon
(p'(\varrho_\varepsilon)\nabla_x \log \varrho_\varepsilon & -p'(\varrho^E)\nabla_x
\log \varrho^E)(\nabla_x \log \varrho_\varepsilon - \nabla_x \log \varrho^E) dxdt
\\
&\leq C\varepsilon \int_0^T \mathcal{E}(t,\cdot) dt.
\end{aligned}
\end{equation*}
Now, from the assumption (\ref{id-conv}) we have $E(0,\cdot) \to 0$ as $\varepsilon \to 0$.
Consequently, from (\ref{step-9}), letting $\varepsilon \to 0$ we end up with
\begin{equation} \label{step_10}
  \mathcal{E}(T,\cdot)
  \leq   C\int_0^T \mathcal{E}(t,\cdot)dt
\end{equation}
and applying Gronwall's inequality we end the proof of Theorem \ref{main}.

\subsubsection*{Acknowledgment
}
M. Caggio has been supported by the Praemium Academiae of \v S. Ne\v
casov\' a, and by
the Czech Science Foundation under the grant GA\v CR 22-01591S. 
L. Bisconti is  member of the Gruppo
Nazionale per l'Analisi Mate\-ma\-tica, la Probabilit\`a e le loro
Applicazioni (GNAMPA) of the Istituto Nazionale di Alta Matematica
(INdAM).

  \section*{Appendix}
  In the following we provide some considerations (and details) on the additional
  boundary condition introduced in $\eqref{bc-ag}_2$, i.e.
  \begin{equation}
    \big[\varrho_\varepsilon  \mathbf{w}_\varepsilon  \big] \times \mathbf{n} |_{\partial  \Omega}
    = 0,\,\, \textrm{ on  }\,\, \D\Omega\,\,\,
    \iff\,\,\, \big[\varrho_\varepsilon  \nabla \log \varrho_\varepsilon\big]
    \times \mathbf{n} |_{\partial  \Omega}= 0,\,\, \textrm{ on  } \D\Omega,
  \end{equation}
  which is considered in the sense of distribution:
  \begin{equation} \label{distribution-on-bb} C_0 ^{\infty}
    ({\D\Omega})\ni \phi \longmapsto \bigg(\big[\varrho_\varepsilon
    \nabla \log \varrho_\varepsilon\big] \times \mathbf{n} |_{\partial
      \Omega}\bigg)(\phi) \triangleq\int_{\D\Omega} \nabla
    \varrho_\varepsilon\times \mathbf{n}\cdot \phi \, ds,
  \end{equation}
  for a.a.  $t\in [0, T]$, $T>0$.  Next, to keep the notation concise,
  we omit the square brackets around $\varrho_\varepsilon
    \nabla \log \varrho_\varepsilon$.

  Let us consider $\varphi : \mathbb{R}^3 \to \mathbb{R}$ such that
  $\varphi\in C^{\infty}_c(\mathbb{R}^3)$, with
  $\textrm{supp}\,\varphi \subseteq \overline{B(0, 1)}$ and
  $\int_{B(0, 1)} \varphi \ dx = 1$.  For $n\in \mathbb{N}$, set
  $\varphi_n = (1/n)^3 \varphi(x/n)$, with
  $\textrm{supp}\,\varphi \subseteq \overline{B(0, 1/n)}$. Then, as a
  standard property of convolutions, we have that
  $\textrm{supp}\, \varphi_n\ast f \subseteq \Omega + \overline{B(0, 1/n})$ and that
 \begin{equation} \label{conv-conv}
   \|  \varphi_n \ast f - f   \|_{L^p(\Omega)} \to 0 \,\,\, \text{as} \,\,\, 
   n \to +\infty
\end{equation}
for $f \in L^p(\Omega)$, $1\leq p < +\infty$. Define $\ren = \varphi_n\ast \varrho_\varepsilon$
with $\textrm{supp}\, \ren \subseteq \Omega + \overline{B(0, 1/n})\triangleq\Omega_n$. Clearly we have that
\begin{equation} \label{cond-curl}
  \ren \nabla \log \ren = \nabla \ren,\,\, \Rightarrow \,\, \nabla \times \big(\ren \nabla \log \ren\big) = 0 ,\,\, \textrm{ in $\Omega$},\,\,\,
  \textrm{for a.a.  $t\in [0, T]$},
\end{equation}
and hence, for any $\phi \in C^{\infty}_0(\Omega_n)$, integrating by parts we have that
\begin{equation*}
  \int_{\Omega}\big(\ren \nabla \log \ren\big)\cdot \nabla\times \phi\, dx
  = \int_{\Omega}\underbrace{\nabla \times \big(\ren \nabla \log \ren \big)}_{=0,\, \textrm{ thanks to }\, \eqref{cond-curl}}\cdot \phi\, dx
  +\int_{\D\Omega} \big(\ren \nabla \log \ren\big)\times \mathbf{n}\cdot \phi\, ds,
\end{equation*}
for a.a.  $t\in [0, T]$. On the other hand, we have that $\ren \nabla \log \ren =
2 \sqrt{\ren} \ \nabla \sqrt{\ren}$, and so the above relation can be equivalently rewritten as 
\begin{equation*}
 2 \int_{\Omega}\big(\sqrt{\ren} \ \nabla \sqrt{\ren}\big)\cdot \nabla\times \phi\, dx =
  \int_{\D\Omega} \big( \nabla \ren\big)\times \mathbf{n}\cdot \phi\, ds,\,\, \textrm {for a.a.  $t\in [0, T]$}.
\end{equation*}
By using the fact that $\sqrt{\ren}\in L^2(0, T; W^{1,2}(\Omega)\big)$, and 
$\sqrt{\ren(t)}\to \sqrt{\varrho_\varepsilon(t)} \textrm{ in }  W^{1,2}(\Omega)$ strongly, for a.a.  $t\in [0, T]$,
then $ \nabla \ren = 2 \sqrt{\ren}\, \nabla \sqrt{\ren}$ is uniformly
bounded, with respect to $n$, in $L^2(0, T; L^{3/2}(\Omega)\big)$ (see
Remark~\ref{rmk-rmk} below), and
\begin{align}
  & \nabla \ren \rightharpoonup \mathcal{B}\,\, \textrm{ in }\,\,
    L^2(0, T; L^{3/2}(\Omega)\big) \textrm{ weakly}, \label{weak-B}
    \intertext{and}
  &\nabla \ren (t)= 2\sqrt{\ren(t)}\,\nabla \sqrt{\ren(t)} \to 2\sqrt{\varrho_\varepsilon(t)}\,
    \nabla \sqrt{\varrho_\varepsilon(t)}\,\, \textrm{ in }\,\,
    L^{3/2}(\Omega) \textrm{ strongly},
\end{align}
for a.a.  $t\in [0, T]$. Also, since $\varrho_\varepsilon\in L^{\infty}(0, T;
L^1(\Omega)\cap L^\gamma(\Omega))$, we have that $\ren(t)\to
\varrho_\varepsilon(t) \textrm{ in }  L^1(\Omega)\cap
L^\gamma(\Omega)$ strongly, for a.a.  $t\in [0, T]$. Due to the
uniqueness of the limit in the above convergence types, we have that
$\nabla \varrho_\varepsilon(t) \triangleq \mathcal{B}(t) =
2\sqrt{\varrho_\varepsilon(t)}\, \nabla
\sqrt{\varrho_\varepsilon(t)}\in L^{3/2}(\Omega)$, for a.a.  $t\in [0, T]$.
In particular, as a consequence of \eqref{weak-B},
we have that $\nabla \varrho_\varepsilon\in L^2(0, T; L^{3/2}(\Omega)$.

Then $\varrho_\varepsilon(t)\in W^{1, \frac{3}{2}}(\Omega)$ and
$\varrho_\varepsilon(t)\in W^{\frac{1}{3}, \frac{3}{2}}(\D\Omega)$, for a.a.  $t\in [0,
T]$. Thus, in \eqref{distribution-on-bb}, we have that $\big(\nabla
\varrho_\varepsilon\times \mathbf{n}\big)(t)\in \Big(W^{\frac{2}{3},
  3}(\D\Omega)\Big)^\ast=W^{-\frac{2}{3}, \frac{3}{2}}(\D\Omega)$, for
a.a.  $t\in [0, T]$. 
\begin{remark} \label{rmk-rmk}
  For any $n$, the vector fields $\varphi_n$ and $\varrho_\varepsilon$ are
  extended to zero outside $\Omega_n$. We have that
  \begin{equation*}
  \begin{aligned}
    \int_0^T\|\nabla \ren(t)\|_{L^{3/2}(\Omega)}^2dt 
    &\leq  \int_0^T\|\nabla \varphi_n\ast
    \varrho_\varepsilon(t)\|_{L^{3/2}(\Omega_n)}^2dt
    \\
    &=   \int_0^T\|\nabla \varphi_n\ast
    \varrho_\varepsilon(t)\|_{L^{3/2}(\mathbb{R}^3)}^2dt
    \\
    & \leq
   \|\nabla \varphi_n\|_{L^q(B(0, 1/n))}^2  \int_0^T \|
   \varrho_\varepsilon(t)\|_{L^{\gamma}(\Omega)}^2 dt\\
   &\leq C T \|
   \varrho_\varepsilon(t)\|_{L^\infty(0, T;L^{\gamma}(\Omega))}^2\leq CT,
    \end{aligned}
  \end{equation*}
  where we used the Young inequality for convolutions with $p=\gamma$,
  $q=3\gamma/(5\gamma -3)$, $r=3/2$ and 
  $\frac{1}{p}+\frac {1}{q} =\frac {1}{r}+1$, and $q\geq 1$.
\end{remark}


\end{document}